\documentclass[smallextended,runningheads]{svjour3}
\smartqed  
\usepackage{graphicx}
\journalname{ }
\usepackage{latexsym}
\usepackage{amsmath}
\usepackage{url}
\usepackage{amssymb}
 
\title
{Incompatible sets of gradients and metastability}
\author{J.M. Ball \and R.D. James}
\institute{ Mathematical Institute,   University of Oxford,
Andrew Wiles Building,
Radcliffe Observatory Quarter,
Woodstock Road,
Oxford,
OX2 6GG,
 U.K. \email{ball@maths.ox.ac.uk}\and Department of Aerospace
Engineering and Mechanics, University of Minnesota, Minneapolis, MN 55455, USA \email{james@umn.edu}}

\date{\today}
\newcommand{\lbl}[1]{\label{#1}}
\newcommand{\arr}{\rightarrow}
\newcommand{\be}{\begin{eqnarray}}
\newcommand{\ee}{\end{eqnarray}}

\newcommand{\R}{\mathbb{R}}
\newcommand{\half}{\frac{1}{2}}
\newcommand{\ep}{\varepsilon}
\newcommand{\eps}{\varepsilon}
\newcommand{\calM}{\mathcal M}
\newcommand{\calC}{\mathcal C}
\newcommand{\om}{\Omega}
\newcommand{\Div}{{\rm div}\,}

\newcommand{\supp}{{\rm supp}\,}

\newcommand{\rank}{{\rm rank}\,}
\newcommand{\dist}{{\rm dist}\,}
\newcommand{\Dist}{\rm dist}

\newcommand{\sign}{{\rm sign}\,}

\newcommand{\1}{{\bf 1}}

\newcommand{\weak}{\rightharpoonup}
\newcommand{\weakstar}{\stackrel{*}{\rightharpoonup}}

\newcommand{\av}{-\hspace{-.15in}\int}
\newcommand{\avsmall}{-\hspace{-.112in}\int}
\newcommand{\mL}{{\mathcal L}}
\newcommand{\mC}{{\mathcal C}}
\newcommand{\mK}{{\mathcal K}}
\newcommand{\mE}{{\mathcal E}}
\newcommand{\qc}{{{\rm qc}}}

\newtheorem{thm}{Theorem}
\newtheorem{cor}[thm]{Corollary}
\newtheorem{lem}[thm]{Lemma}
\newtheorem{prop}[thm]{Proposition}

\newtheorem{defn}{Definition}
\newtheorem{rem}{Remark}

\numberwithin{equation}{section}

\begin{document}
 \maketitle
 \markboth{ }{}
\begin{abstract}  We give a mathematical analysis of a concept of metastability
induced by incompatibility.  The physical setting is a single parent phase, just about
to undergo transformation to a product phase of lower energy density.  
Under certain conditions of incompatibility of
the energy wells of this energy density, we show that
the parent phase   is metastable in a  strong sense, namely it is
a local minimizer of the free energy in an
$L^1$ neighbourhood of its deformation.  The reason behind this result 
is that, due to the incompatibility of the energy wells, a small nucleus of the product
phase is necessarily accompanied by a stressed transition layer whose energetic
cost exceeds the energy lowering capacity of the nucleus.  We define and characterize incompatible sets of 
matrices, in terms of which the transition layer estimate at the heart of the proof of metastability is expressed.  Finally we discuss
connections with experiment and place this concept of metastability in the wider
context of recent theoretical and experimental research on metastability and hysteresis.
 
\end{abstract}

\section{Introduction}
\setcounter{equation}{0}

Materials that undergo first order phase transformations without diffusion
typically exhibit hysteresis loops, that is, loops in a plot of a measured property vs.~temperature
as the temperature is cycled back and forth through the transformation temperature.
It is the rule rather than the exception that the area within these loops does not tend to zero
as the temperature is cycled more and more slowly.  Thus, while there is an issue of the 
time-scale of such experiments, hysteresis is apparently not entirely due to viscosity  
or other thermally activated mechanisms.  An alternative explanation
is metastability, as quantified by the presence of local minimizers in a continuum level elastic 
energy.  This paper is a mathematical analysis of this possibility appropriate to cases in
which the two phases are geometrically incompatible in a certain precise sense. 

To illustrate our analysis in a simple case, let $\om\subset\R^n$ be a bounded domain with sufficiently smooth boundary $\partial\om$, and consider the energy functional
\be 
\lbl{simplecase}
I(y)=\int_\om W(Dy(x))\,dx,
\ee
defined for mappings $y:\om\to\R^m$, where $Dy(x)=\left(\frac{\partial y_i}{\partial x_\alpha}(x)\right)$ denotes the gradient of $y$, so that $Dy(x)$ belongs to the set $M^{m\times n}$  of real $m\times n$ matrices for each $x$. Suppose that $W:M^{m\times n}\to \R$ is a continuous function satisfying $W(A)\geq C(1+|A|^p)$ for constants $C>0$, $p>1$, and having exactly two local minimizers at matrices $A_1, A_2$ with  $W(A_1)> W(A_2)$. Thus, imposing no boundary conditions on $\partial\om$, the global minimizers of $I$ are given by affine mappings $y_{\rm min}(x)=a_2+A_2x, a_2\in\R^m$ having constant gradient $A_2$. Under suitable structural conditions on $W$, we prove that if $A_1, A_2$ are incompatible in the sense that $\rank (A_1-A_2)>1$, and if $W(A_1)-W(A_2)$ is sufficiently small, then  $y^*(x)=a_1+A_1x, a_1\in\R^m$ is a local minimizer of $I$ in $L^1$,  i.e. there exists $\sigma>0$ such that $I(y)\geq I(y^*)$ if $\|y-y^*\|_1<\sigma$. 

Notice that if  $\|y-y^*\|_1<\sigma$ then it can happen that $Dy(x)$ belongs to a small neighbourhood of $A_2$ on a set $E\subset\om$ of positive measure, so that $W(Dy(x))<W(A_1)$ for $x\in E$. 
The basic idea underlying the analysis is that, if a nucleus $E$ of the product phase of
arbitrary form is introduced in this way 
so as to lower the energy,
then, due to the incompatibility between the two phases, this nucleus is necessarily accompanied by
a  transition layer that interpolates between the nucleus and the parent phase $A_1$.  This transition layer costs
more energy than the lowering of energy due to the presence of the new phase.  The analysis
is delicate because the energy \eqref{simplecase} contains no contribution from interfacial energy that would
dominate at small scales.  Thus, for example, scaling down of the nucleus and transition
layer using geometric similarity preserves the ratio of transition layer and nucleus energies.

The above result is a special case of  the considerably more general metastability theorem (Theorem \ref{metstab}) proved in this paper, in which the parent and product  phases are represented by disjoint compact sets of matrices $K_1$ and $K_2$ respectively. Since the multiwell elastic energies we consider can exhibit nonattainment of the minimum of $I$, we formulate the problem more generally in terms of gradient Young measures, so that the metastability theorem applies to microstructures.  We assume that $K_1, K_2$ are incompatible in the sense that 
if an $L^\infty$ gradient Young measure $\nu=(\nu_x)_{x\in\om}$ is such that $\supp\nu_x\subset K_1\cup K_2$ for a.e. $x\in\om$, then either $\supp\nu_x\subset K_1$ for a.e. $x\in\om$, or  $\supp\nu_x\subset K_2$ for a.e. $x\in\om$.
 We can then 
estimate the energy of a transition layer that must be present if a gradient Young measure has nontrivial support
near both $K_1$ and $K_2$.  The delicate case is when the support of the gradient Young measure
near either $K_1$ or $K_2$ is vanishingly small; to handle this, we find a way of
moving and rescaling suitable convex subsets of $\Omega$ so as to get  half of the support  
of the gradient  Young measure in the subset near $K_1$, and half near $K_2$, which enables us to use a version of the Vitali covering lemma to obtain the desired estimate.  This method of varying the volume
fractions of a gradient Young measure has other applications 
and will be developed in a forthcoming paper \cite{u6}.  Using the estimate
for the energy of the transition layer, we show that a gradient Young measure  supported on $K_1$ is a local minimizer with respect to the $L^1$ norm of the difference between the underlying deformations, for energy densities
that have a well at $K_1$ and a slightly lower well at $K_2$.

The shape of the domain $\Omega$ matters for our analysis.  It is possible to 
defeat metastability as discussed here using the ``rooms and passages'' domain of Fraenkel \cite{fraenkel79}, which consists of a bounded domain formed from an infinite sequence of  rooms of vanishingly small diameters, each connected to the two adjacent rooms by passages of even smaller diameter.  
For such a domain the parent phase is not an $L^1$ local minimizer, because one can reduce the energy through  deformations that are arbitrarily close in $L^1$, whose gradients lie entirely in the parent phase except for a nucleus of the product phase occupying a single room, together with transition layers in the two adjacent passages.  To quantify the
effect of domain shape on metastability we introduce a concept of a  
domain connected with respect to rigid-body motions of a convex set $C$ (see Section \ref{tp}), for which the method outlined in the previous paragraph can be applied,  the constants in the transition layer estimate depending on $C$. This shape dependence is
expected to have physical implications regarding the size of the hysteresis, 
for example in more conventional domains with
sharply outward pointing corners.  This phenomenon is therefore different
from the well-known lowering of hysteresis that occurs in magnetism due to sharp
{\it inward} pointing corners, and which is one explanation of the coercivity paradox.

In applications,  $K_2$ usually grows with a parameter, either stress or temperature
(Section \ref{sectapp}). As discussed by C. Chu and the authors \cite{j43}, one can derive
upper bounds to the size of the hysteresis by considering test functions.  The
easiest upper bound is found when the stress, say, reaches a point where $K_2$
has grown sufficiently that there are matrices $A \in K_1$ and $B \in K_2$ such that
$\rank(B - A) = 1$.  This upper bound is directly related to the Schmid Law \cite{schmidboas50},
though the conventional reasoning behind this law is completely different than
the one offered here (see Section \ref{sectvar}).  In fact, for the problem of variant rearrangement discussed in \cite{j43} and 
Section \ref{sectvar} there is a more complicated test function that implies a loss of metastability
earlier than the simple rank-one connection between $A$ and $B$ \cite{j43}.
Curiously, these more complicated test functions require   $\partial \Omega$ to have a sharp corner.
A more careful analysis of Forclaz \cite{forclaz14} seems to suggest that this is
necessary.

Our differential constraint implying compatibility conditions is curl$\,F = 0$,
where $F$ is a gradient.  Our framework applies to other constraints in the
theory of compensated compactness, except possibly that, in the case of compact sets $K_1$ and $K_2$,
we use Zhang's lemma (see \cite{zhang92} and Lemma \ref{zhanglemma}) to show that the
definition of incompatibility is independent of the Sobolev exponent $p$. 
The interesting question 
of what are the incompatible sets for other important differential 
constraints seems not to have been explicitly investigated.

The first metastability result of the type given here is due to Kohn \& Sternberg \cite{kohnsternberg89} who used $\Gamma$-convergence to prove under quasiconvexity assumptions the existence of local minimizers for \eqref{simplecase} with gradient near $A_1$ (see also \cite{kohnetal2000} for an improved version in particular showing that $y^*$ is a local minimizer). 
Our work is also related to the important results of Grabovsky \& Mengesha \cite{grabovskymengesha09}. 
They prove, under assumptions of quasiconvexity, quasiconvexity at the
boundary, and nonegativity of the second variation,
all imposed locally at the gradients of a  $C^1$ solution of the Euler-Lagrange equations,
that this solution is an $L^{\infty}$-local minimizer.  
Our approach differs from theirs in that we assume a multiwell structure of
the energy, but make much weaker assumptions on the eventual local minimizer, which in our
case is allowed to be a gradient Young measure.
The idea behind the concept of metastability that we discuss here was first introduced
without proof by C. Chu and the authors in \cite{j43,j44}. 

The plan of the paper is as follows. In Section \ref{tp} we give some necessary technical background and preliminary results concerning  gradient Young measures, quasiconvexity and quasiconvexifications, and define and discuss $C$-connected 
domains. In Section \ref{is} we define incompatible sets, and characterize them in terms of quasiconvexity, analyzing various examples. The fundamental transition layer estimate is proved in Section \ref{trans}, and applied to prove metastability in Section \ref{meta}. Finally, in Section \ref{sectapp} we give various applications of the metastability theorem. The first application is to the experiments of Chu \& James on variant rearrangement in CuAlNi single crystals under biaxial dead loads, which originally motivated this paper. Then we discuss purely dilatational phase transformations, and the interesting case of Terephthalic acid. Finally, in Section \ref{sectper} we give a perspective on
metastability and hysteresis, discussing in particular 
 other concepts of metastability \cite{james05,james09,kneupfer11,delville10,james13,kneupfer13,zwicknagl13} that have  recently
appeared in the literature, as well as experiments that show a dramatic 
dependence of the size of the hysteresis  on conditions of compatibility
 \cite{cui06,james09,zarnetta10,song13}.  These 
observations answer some questions and raise others.

\section{Technical preliminaries}
\label{tp}
\subsection{Gradient Young measures and quasiconvexity}Let $m\geq 1, n\geq 1$.  We denote by
$M^{m\times n}$ the set
of real $m\times n$ matrices, and by $SO(n)$ the rotation group of matrices $R\in M^{n\times n}$ with $R^TR=\1,\,\det R=1$. Lebesgue measure in $\R^n$ is denoted by $\mL^n$. Let $\om\subset\R^n$ be  a bounded domain. Fix $p$ with $1 \leq p
\leq \infty$.  We consider $\R^m-$valued distributions $y$ in $\Omega$ whose gradients $Dy$ belong to $L^p(\Omega;M^{m\times n})$. Without further hypotheses on $\Omega$ such distributions need not in general belong to $W^{1,p}(\Omega;\R^m)$, but it is proved in Maz'ya \cite[p.~21]{mazya2011} that they do so if $\Omega$ satisfies {\it the cone condition} with respect to a fixed cone $C^*=\{x\in\R^n:|x|\leq\rho,\;x\cdot e_1\geq |x|\cos\alpha\}$, where $\rho>0,\;0<\alpha<\frac{\pi}{2}$; that is,  any point $x\in\om$ is the vertex of a cone congruent to $C^*$ and contained in $\om$, so that $x+QC^*\subset\om$ for some $Q\in SO(n)$.  

 Given a sequence $y^{(j)}$ such that $Dy^{(j)}$ is weakly  convergent in $L^{p}(\Omega;M^{m\times n})$ (weak* if $p=\infty$) there exist (see, for example, \cite{p21}) a subsequence $y^{(\mu)}$ and a family of probability measures $(\nu_x)_{x\in\Omega}$ on   $M^{m\times n}$, depending measurably on $x\in\Omega$,  such that for any continuous function $f:M^{m\times n}\to\R$ and measurable $G\subset\Omega$ 
$$f(Dy^{(\mu)})\weak \langle \nu_x,f\rangle \mbox{ in } L^1(G)$$
whenever this weak limit exists. We call the family $\nu=(\nu_x)_{x\in\Omega}$   the $L^{p}$ {\it gradient Young measure} generated by the sequence $Dy^{(\mu)}$ (alternative names in common use are $W^{1,p}$ gradient Young measure, or $p$-gradient Young measure). If $\nu_x=\nu$ is independent of $x$ we say that the gradient Young measure is {\it homogeneous}. If $1<p\leq\infty$ then the weak relative compactness condition is equivalent to boundedness of $Dy^{(j)}$ in $L^{p}(\Omega; M^{m\times
n})$, whereas if $p=1$ it is equivalent to equi-integrability of $Dy^{(j)}$. If   $K\subset M^{m\times n}$ is closed and  $Dy^{(\mu)}\to K$ in measure, that is $$\lim_{j\to\infty}\mL^n(\{x\in\Omega:\dist(Dy^{(\mu)}(x),K)>\varepsilon\})= 0\mbox{ for all }\varepsilon>0,$$ 
then  $\supp\nu_x\subset K$ for a.e. $x\in\Omega$.

\begin{defn}
\lbl{quasiconvexdef}
A  function $\varphi:M^{m\times n}\to\R$ is {\it quasiconvex} if 
\be\lbl{quasiconvexity}\av_G\varphi(A+D\theta(x))\,dx\geq \varphi(A)\ee
for any bounded open set $G\subset\R^n$, all $A\in M^{m\times n}$ and any $\theta\in W_0^{1,\infty} (G;\R^m)$, whenever the integral on the left-hand side exists. 
\end{defn}
As is well known (see, for example, \cite[p.~172]{dacorogna}) this definition does not depend on $G$. Also any quasiconvex function $\varphi:M^{m\times n}\to\R$ is rank-one convex and thus continuous (see \cite[Lemma 4.3]{muller99}).

We recall the characterization of $L^{p}$ gradient Young measures in terms of quasiconvexity due to Kinderlehrer \& Pedregal. In the following statement we combine together various of their results. 
\begin{thm}[Kinderlehrer \& Pedregal \cite{kinderlehrerpedregal91,kinderlehrerpedregal94}]\label{kp}
Let $1\leq p\leq\infty$. A family $\nu=(\nu_x)_{x\in\Omega}$ of probability measures on $M^{m\times n}$, depending measurably on $x$, is an $L^{p}$ gradient Young measure if and only if
\\ $(i)$ $\bar\nu_x:=\int_{M^{m\times n}}A\,d\nu_x(A)=Dy(x)$ for a.e. $x\in\Omega$ and some $\R^m-$valued distribution $y$ with  $Dy\in L^{p}(\Omega;M^{m\times n})$\\
$(ii)$ for any quasiconvex $\varphi:M^{m\times n}\to\R$ satisfying $|\varphi(A)|\leq C(1+|A|^p)$ for all $A\in M^{m\times n}$, where $C>0$ is constant, (no growth condition required if $p=\infty$) we have 
$$\langle \nu_x,\varphi\rangle :=\int_{M^{m\times n}}\varphi(A)\,d\nu_x(A)\geq \varphi(\bar\nu_x)\mbox{ for a.e. }x\in\Omega,$$
$(iii)$ if $1\leq p<\infty$ then $\int_\Omega\int_{M^{m\times n}}|A|^p\,d\nu_x(A)\,dx<\infty$; if $p=\infty$ then $\supp \nu_x\subset G$ for some compact $G\subset M^{m\times n}$.\\
Furthermore, if $1\leq p<\infty$ any $L^{p}$ gradient Young measure $(\nu_x)_{x\in\Omega}$ is generated by some sequence of gradients $Dz^{(j)}$ (possibly different from the generating sequence $Dy^{(j)}$ in the definition) such that $|Dz^{(j)}|^p$ converges weakly in $L^1(\Omega)$ to some $g\in L^1(\Omega)$. 
\end{thm}
\begin{rem} \rm In \cite{kinderlehrerpedregal94} no assumption is stated concerning the bounded domain $\Omega$, but the  proof  uses the Sobolev embedding theorem for $\Omega$ and thus implicitly makes some assumption. However, the proof can be easily modified by, in Lemma 5.1, writing $\Omega$ as a disjoint union of scaled copies of  a cube, rather than of scaled copies of $\Omega$.  For an alternative approach to $L^{p}$ gradient Young measures see Sychev \cite{sychev99}.
\end{rem}
We will make frequent use of the following version of Zhang's lemma that is a consequence of M\"uller \cite[Corollary 3]{muller99a}. The original version is due to Zhang \cite[Lemma 3.1]{zhang92}.
\begin{lemma} 
\label{zhanglemma}
Let $K\subset M^{m\times n}$ be compact, and suppose $\nu=(\nu_x)_{x\in\Omega}$ is an $L^{1}$ gradient Young measure with $\supp\nu_x\subset K$ for a.e. $x\in\Omega$. Then $\nu$ is an $L^{\infty}$ gradient Young measure; that is it can be generated by a sequence $z^{(j)}$ whose gradients $Dz^{(j)}$ are  bounded in $L^{\infty}(\Omega;M^{m\times n})$.
\end{lemma}

\subsection{Quasiconvex functions taking the value $+\infty$}
Some care needs to be taken when defining quasiconvexity for functions which take the value $+\infty$. For example, as pointed out in \cite[Example 3.5]{j26}, the function $\varphi$ defined by $\varphi(0)=\varphi(a\otimes b)=0, \varphi(A)= +\infty$ otherwise, where $a\in \R^m, b\in \R^n$ are nonzero vectors, satisfies \eqref{quasiconvexity} for any bounded open set $G\subset\R^n$, all $A\in M^{m\times n}$ and any $\theta\in W_0^{1,\infty} (G;\R^m)$, even though $\varphi$ is not rank-one convex and $I(y)=\int_\Omega\varphi(Dy)\,dx$ is not sequentially weak* lower semicontinuous in $W^{1,\infty}(\Omega;\R^m)$. See \cite[p.~9]{p31} for further discussion, and another example related to Example \ref{4matrices}. In this paper we will define quasiconvexity for functions which take the value $+\infty$ differently from in \cite{j26}, as follows. 
\begin{defn}
\lbl{qcinfty}
A function $\varphi:M^{m\times n}\arr\R\cup\{\infty\}$ is quasiconvex if there exists a nondecreasing
sequence $\varphi^{(j)}:M^{m\times n}\arr\R$ of continuous quasiconvex functions with
$$\varphi(A)=\lim_{j\arr\infty}\varphi^{(j)}(A)\hspace{.2in}\text{for all }\ A\in M^{m\times n}.$$
\end{defn}
\begin{rem}\rm
\label{qclsc}
Note that any quasiconvex $\varphi:M^{m\times n}\arr\R\cup\{\infty\}$ is lower
semicontinuous because it is the supremum of continuous functions.
\end{rem}
\begin{rem}\rm
\lbl{consistency}
Suppose $\varphi:M^{m\times n}\to \R$ is quasiconvex according to the above definition. Let $G$ be a bounded domain, $A\in M^{m\times n}$, $\theta\in W^{1,\infty}_0(G;\R^m)$. Then for each $j$ we have
$$\av_G\varphi(A+D\theta(x))\,dx\geq\av_G\varphi^{(j)}(A+D\theta(x))\,dx\geq \varphi^{(j)}(A),$$
the left-hand integral being well defined by Remark \ref{qclsc}, so that passing to the limit $j\to\infty$  we deduce that \eqref{quasiconvexity} holds.  Thus $\varphi$ is quasiconvex in the sense of Definition \ref{quasiconvexdef}.
\end{rem}
Let $\varphi:M^{m\times n}\arr\R\cup\{\infty\}$ be quasiconvex.  Let $(\nu_x)_{x\in\om}$ be an $L^{\infty}$ gradient Young measure    corresponding to a sequence $y^{(k)}$ with
$Dy^{(k)}\weakstar Dy$   in $L^{\infty}(\om;\R^m)$. For each $j$ we have
$$\int_\om\varphi^{(j)}(Dy^{(k)})\,dx\leq\int_\om\varphi(Dy^{(k)})\,dx.$$
Since $\varphi^{(j)}$ is quasiconvex, letting $k\arr\infty$ we obtain, using the lower semicontinuity of $\int_\Omega\varphi^{(j)}(Dz)\,dx$ with respect to weak* convergence in $W^{1,\infty}(\Omega;\R^m)$ (see, for example, \cite[p. 369]{dacorogna}),
$$\int_\om\varphi^{(j)}(Dy)\,dx\leq\int_\om\langle\nu_x,\varphi^{(j)}
\rangle\,dx\leq\liminf_{k\arr\infty}\int_\om\varphi(Dy^{(k)})\,dx.$$
(In order to apply the lower semicontinuity when we just have $Dy^{(k)}\weakstar Dy$ in $L^\infty$, we can, for example, write $\Omega$ as a disjoint union of cubes. In each cube we can fix $y^{(k)}$ to be zero at the centre of the cube, from which weak* convergence in $W^{1,\infty}$ follows. Thus we have the desired lower semicontinuity on each cube, from which that on $\Omega$ follows.)
Letting $j\arr\infty$, noting that $\varphi^{(j)}(Dy)\geq \varphi^{(1)}(Dy)$, and using monotone convergence, it follows that
$$-\infty<\int_\om\varphi(Dy)\,dx\leq\int_\om\langle\nu_x,\varphi\rangle\,dx\leq
\liminf_{k\arr\infty}\int_\om\varphi(Dy^{(k)})\,dx.$$
Thus the functional
$$I(y)=\int_\om\varphi(Dy)\,dx$$
is sequentially   lower semicontinuous with respect to weak* convergence of the gradient in $L^\infty$.
Also, if $\nu_x=\nu$ is homogeneous then we obtain
$$\varphi(\bar\nu)\leq\langle\nu,\varphi\rangle.$$
\begin{lemma}
\label{qcequiv}
Assume that $\varphi:M^{m\times n}\to\R\cup\{+\infty\}$ is such that ${\rm dom}\,\varphi=\{A\in M^{m\times n}:\varphi(A)<\infty\}$ is bounded. Then $\varphi$ is quasiconvex if and only if $\varphi$ is lower semicontinuous and 
$\langle\mu,\varphi\rangle\geq\varphi(\bar\mu)$ for all homogeneous $L^{\infty}$ gradient Young measures $\mu$.
\end{lemma}
\begin{proof}
The necessity of the conditions has already been proved without the extra condition on $\varphi$. Conversely, suppose that $\varphi$ is lower semicontinuous and that $\langle\mu,\varphi\rangle\geq\varphi(\bar\mu)$ for all homogeneous gradient Young measures $\mu$. Since ${\rm dom}\,\varphi$ is bounded and $\varphi$ lower semicontinuous, $\varphi$ is bounded below. Also, the lower semicontinuity implies (for example by  \cite[Theorem 3.8, p.~76]{mcshane}) that there is a nondecreasing sequence of continuous functions $\psi^{(j)}$ such that $\lim_{j\to\infty}\psi^{(j)}(A)=\varphi(A)$ for all $A\in M^{m\times n}$. Since ${\rm dom}\,\varphi$ is bounded we may also assume that $\psi^{(j)}(A)\geq C|A|^p-C_1$ for all $A\in M^{m\times n}$, where $C>0$ and $C_1$ are constants and $p>1$. Let $\varphi^{(j)}=(\psi^{(j)})^{\rm qc}$ be the quasiconvexification of $\psi^{(j)}$, that is the supremum of all continuous real-valued quasiconvex functions less than or equal to $\psi^{(j)}$. Then $\varphi^{(j)}$ is continuous and quasiconvex \cite[p.~271]{dacorogna}, and it suffices to show that $\lim_{j\to\infty}\varphi^{(j)}(A)=\varphi(A)$ for all $A$. Suppose this is not the case, that there exists $A\in M^{m\times n}$ with $\varphi^{(j)}(A)\leq M<\infty$, where $M<\varphi(A)$. By the characterization \cite[p.~271]{dacorogna} of quasiconvexifications,
$$\varphi^{(j)}(A)=\inf_{\theta\in W^{1,\infty}_0(Q;\R^m)}\av_Q\psi^{(j)}(A+D\theta)\,dx,$$
where $Q=(0,1)^n$.
  Hence there exist $\varepsilon>0$ and a sequence $\theta^{(j)}\in W^{1,\infty}_0(Q;\R^m)$ such that 
$$\av_Q\psi^{(j)}(A+D\theta^{(j)})\,dx\leq M+\varepsilon<\varphi(A).$$
Thus for any $j\geq k$ we have
$$\av_Q\tilde\psi^{(k)}(A+D\theta^{(j)})\,dx\leq\av_Q\psi^{(k)}(A+D\theta^{(j)})\,dx\leq M+\varepsilon,$$
where $\tilde\psi^{(k)}=\min(k,\psi^{(k)})$.
From the growth condition on $\psi^{(j)}$, a subsequence (not relabelled) of $A+D\theta^{(j)}$ generates an $L^p$ gradient Young measure $(\nu_x)_{x\in \Omega}$.
Passing to the limit $j\to\infty$, noting that $\tilde\psi^{(k)}$ is bounded, we deduce that 
$$\av_Q\langle\nu_x,\tilde\psi^{(k)}\rangle\,dx\leq M+\varepsilon,$$
and then letting $k\to\infty$ we obtain by monotone convergence that
$$\av_Q\langle\nu_x,\varphi\rangle\,dx\leq M+\varepsilon.$$
But then $\langle\mu,\varphi\rangle\leq M+\varepsilon$, where $\mu=\avsmall_Q\nu_x\,dx$, which by \cite[Theorem 3.1]{kinderlehrerpedregal94} is a homogeneous $L^p$ gradient Young measure with centre of mass $\bar\mu=A$. Since $\varphi(A)=\infty$ for $A\not\in {\rm dom}\,\varphi$ we deduce that $\supp\mu\subset \overline{ {\rm dom}\,\varphi}$. Since $\overline{ {\rm dom}\,\varphi}$ is compact, it follows from   Lemma \ref{zhanglemma} that $\mu$ is an $L^{\infty}$ gradient Young measure. Hence by our assumption we have that $\varphi(A)\leq M+\varepsilon<\varphi(A)$, a contradiction.
\end{proof}
\begin{rem}\rm
\lbl{newrem}
Lemma \ref{qcequiv} is a $p=\infty$ version of a result of Kristensen \cite{kristensen94}, who showed using a similar argument that 
  if $\varphi:M^{m\times n}\to\R\cup\{+\infty\}$  satisfies the growth condition
\be 
\lbl{growthp}
\varphi(A)\geq C|A|^p-C_1 \mbox{ for all }A\in M^{m\times n},
\ee
for some $C>0,C_1,  p>1$, then $\varphi$ is the supremum of a nondecreasing sequence of continuous quasiconvex functions $\varphi^{(j)}:M^{m\times n}\to\R$ satisfying $M\leq \varphi^{(j)}(A)\leq \alpha_j|A|^p+\beta_j$ for constants $\alpha_j>0, \beta_j,M$ (so that in particular $\varphi$ is quasiconvex according to Definition \ref{qcinfty}) if and only if $\varphi$ is lower semicontinuous and 
\be 
\lbl{qcineq}
\langle\mu,\varphi\rangle\geq\varphi(\bar\mu)
\ee
for any homogeneous $L^p$ gradient Young measure $\mu$ (i.e. $\varphi$ is closed $W^{1,p}$ quasiconvex in the sense of Pedregal \cite{pedregal94}).

Note, however, that \eqref{qcineq} is not in general a necessary condition for such $\varphi$ to be quasiconvex, as can be seen by taking $\varphi$ to be a finite quasiconvex function satisfying \eqref{growthp} that is not $W^{1,p}$ quasiconvex (see  \cite{j26} with, for example, $m=n=3,\, p=2$).
\end{rem}
\begin{rem}\rm
\label{localqc}
The same proof shows that if $\varphi:M^{m\times n}\to\R\cup\{+\infty\}$ is  a lower semicontinuous function with ${\rm dom}\,\varphi$ bounded, and if $A\in M^{m\times n}$, then there exists a nondecreasing sequence of continuous quasiconvex functions $\varphi^{(j)}:M^{m\times n}\to\R$ with $\varphi^{(j)}(A)\to \varphi(A)$ if and only if $\varphi$
\be 
\lbl{qcF}
\langle \mu,\varphi\rangle\geq \varphi(A)
\ee 
for all homogeneous gradient Young measures $\mu$ with $\bar\mu=A$.
\end{rem}

\subsection{Quasiconvexification of sets}
A closed set $G\subset M^{m\times n}$ is {\it quasiconvex} if $G=\varphi^{-1}(0)$ for some nonnegative finite quasiconvex function $\varphi$. Given $H\subset M^{m\times n}$ we can thus define the {\it quasiconvexification} $H^{\rm qc}$ of $H$ by
$$H^{\rm qc}=\bigcap \{ G\supset H: G \mbox{ quasiconvex} \}.$$
We recall the following equivalent characterizations of $K^{\rm qc}$ for compact $K\subset M^{m\times n}$:
\begin{prop}
\label{qcequiva}
If $K\subset M^{m\times n}$ is compact then 
\begin{eqnarray*}
 K^{\rm qc}&=& \{\bar\nu:\nu\; \mbox{\rm  a homogeneous gradient Young measure with }\supp\nu\subset K\}\\
&=&\{A\in M^{m\times n}:\varphi(A)\leq \max_{B\in K}\varphi(B)\;\mbox{\rm  for all finite quasiconvex }\varphi\}\\
&=& (\Dist_K^{\rm qc})^{-1}(0),
\end{eqnarray*}
where $\Dist_K$ is the distance function to the set $K$.
\end{prop}
\begin{proof} The equality of the three sets in the proposition is proved in \cite[Theorem 4.10, p.~54]{muller99}. Since $(\Dist_K^{\rm qc})^{-1}(0)$ is quasiconvex and $\Dist_K^{\rm qc}(A)=0$ for all $A\in K$, we have that $K^{\rm qc}\subset (\Dist_K^{\rm qc})^{-1}(0)$. But if $\varphi(A)\leq \max_{B\in K}\varphi(B)$ for all finite quasiconvex $\varphi$ then $A$ belongs to any quasiconvex set $G\supset K$. Hence $K^{\rm qc}\subset \{A\in M^{m\times n}:\varphi(A)\leq \max_{B\in K}\varphi(B)\mbox{ for all finite quasiconvex }\varphi\}\subset K^{\rm qc}$, so that all three sets in the proposition equal $K^{\rm qc}$.
\end{proof}

\begin{thm}
\label{splitting}
Let $K_1,\ldots, K_N$ be compact subsets of $M^{m\times n}$ whose quasiconvexifications $K_r^{\rm qc}$ are disjoint. Let $\nu=(\nu_x)_{x\in\Omega}$ be an $L^{\infty}$ gradient Young measure such that $\supp\nu_x\subset\bigcup_{r=1}^NK_r^{\rm qc}$ for a.e. $x\in\Omega$. Then there is an $L^{\infty}$ gradient Young measure $\nu^*=(\nu^*_x)_{x\in\Omega}$ such that $\supp\nu^*_x\subset\bigcup_{r=1}^NK_r$, $\bar\nu^*_x=\bar\nu_x$ and $\nu^*_x(K_r)=\nu_x(K_r^{\rm qc}), r=1,\ldots,N,$ for a.e. $x\in\Omega$. If $\nu$ is homogeneous then $\nu^*$ can be chosen to be homogeneous.
\end{thm}
In order to prove Theorem \ref{splitting} we will need two technical lemmas. Let $\mathcal P(M^{m\times n})$ denote the set of probability measures on $M^{m\times n}$. Given a compact set $K\subset M^{m\times n}$ we denote by ${\rm GYM}(K)$ the set of homogeneous ($L^{\infty}$) gradient Young measures $\mu$ with $\supp \mu\subset K$.
\begin{lemma}
\label{multifunction}
Let $K\subset M^{m\times n}$ be compact. For $A\in K^{\rm qc}$ define
$$F(A)=\{\mu\in  {\rm GYM}(K):  \bar\mu=A\}.$$
Then $F(A)$ is a nonempty, sequentially weak* closed subset of $\mathcal P(M^{m\times n})$.
\end{lemma}
\begin{proof}
Let $\mu_j\in F(A)$ with $\mu_j\weakstar\mu$ (that is $\langle\mu_j,f\rangle\to\langle\mu,f\rangle$ for all $f\in C_0(M^{m\times n})$, where $C_0(M^{m\times n})$ denotes the space of all continuous functions $f:M^{m\times n}\to\R$ such that $\lim_{|A|\to\infty}f(A)=0$).  If $\psi\in C_0(M^{m\times n})$ with $\psi=0$ on $K$, then $\langle\mu,\psi\rangle=\lim_{j\to\infty}\langle\mu_j,\psi\rangle=0$, and so $\supp\mu\subset K$. Then, choosing $f\in C_0(M^{m\times n})$ with $f=1$ on $K$, and noting that $\langle\mu_j,f\rangle=1$ we have that $\langle\mu,f\rangle=\mu(K)=1$, and so $\mu\in\mathcal P(M^{m\times n})$. Let $h\in C_0(M^{m\times n})$ with $h(B)=B$ for all $B\in K$. Then $A=\bar\mu_j=\langle\mu_j,h\rangle$, so that $\lim_{j\to\infty}\langle\mu_j,h\rangle=\langle\mu,h\rangle=\bar\mu=A$. If $g$ is finite and quasiconvex, we have by Theorem  \ref{kp} that $\langle\mu_j,g\rangle\geq g(A)$ for all $j$, so that passing to the limit (using $\supp\mu_j\subset K$) we obtain $\langle\mu,g\rangle\geq g(A)$, so that, again using Theorem \ref{kp}, we have $\mu\in{\rm GYM}(K)$ as required.
\end{proof}
\begin{lemma}
\label{selection}
There is a Borel measurable map $A\mapsto \mu_A$ from $K^{\rm qc}$ to the set $\mathcal P(K)$ of probability measures on $K$ endowed with the weak* topology, such that $\mu_A\in F(A)$ for all $A\in K^{\rm qc}$.
\end{lemma}
\begin{proof}
By Parthasarathy \cite[Theorems 6.3 6.4, 6.5 pp.~44-46]{Parthasarathy} $\mathcal P(K)$ endowed with the weak* topology is a Polish space, i.e.   separable and completely metrizable. We first claim that the multivalued map $F:K^{\rm qc}\to \mathcal P(K)$ is upper semicontinuous, i.e. for every closed  $G\subset \mathcal P(K)$ the set $\{A\in K^{qc}:F(A)\cap G\neq\emptyset\}$ is closed in $M^{m\times n}$. Indeed if $A_j\in K^{qc}$ with $\mu_{A_j}\in F(A_j)\cap G$ and $A_j\to A$ then we may assume that $\mu_{A_j}\weakstar \mu$ (since $\mu_{A_j}$ is bounded in the dual space of $C_0(M^{m\times n})$, namely the space of measures). By a similar argument to that of the proof of Lemma \ref{multifunction} we deduce that $\mu\in F(A)\cap G$ as required. 

We now apply the measurable selection theorem of Kuratowski \& Ryll-Nardzewski \cite{Kuratowski}, which in the statement by Wagner \cite[Theorem 4.1]{wagner} implies that a Borel measurable selection $\mu_A\in F(A)$ exists whenever $F(A)$ is closed for all $A\in K^{\rm qc}$ and $A\mapsto F(A)$ is weakly measurable. In our case weak measurability means that $\{A\in K^{qc}:F(A)\cap U\neq\emptyset\}$ is Borel measurable, and it is shown in \cite[Theorem 4.2]{wagner} that this is implied by upper semicontinuity, giving the required result since $F(A)$ is closed by Lemma \ref{multifunction}. 
\end{proof} 
\begin{proof}{\it  of Theorem \ref{splitting}} Let $\mathcal K=\cup_{r=1}^NK_r$. We apply Lemma \ref{selection} to each compact set $K_r$, and denote the corresponding Borel measurable selection $\mu_A^r$, so that for each $r=1,\ldots,N$ and $A\in K_r^{\rm qc}$ we have $\mu_A^r\in {\rm GYM}(K_r)$ with $\bar\mu_A^r=A$. 
We then define the required gradient Young measure $\nu^*=(\nu^*_x)_{x\in\Omega}$ by the action of $\nu_x^*$ on functions $f\in C(\mathcal K)$ through the formula
\begin{equation}\label{splittingformula}
\langle\nu^*_x,f\rangle=\sum_{r=1}^N\langle\nu_x,\langle\mu_A^r,f\rangle\rangle,
\end{equation}
  that is
\begin{equation}\label{splittingformula1}
\langle\nu^*_x,f\rangle=\sum_{r=1}^N\int_{K_r^{\rm qc}}\int_{K_r}f(B)\,d\mu_A^r(B)\,d\nu_x(A).
\end{equation}
(Note that $\langle\nu^*_x,f\rangle$ is well defined because we can extend $f$ outside the compact set $\mathcal K$ to a function $f\in C_0(M^{m\times n})$ and $\supp\mu_A^r\subset K_r$.)
Since $\langle\nu^*_x,f\rangle \geq 0$ for $f\geq 0$, $\nu^*_x$ is a positive measure. Choosing   $f=1$  we see that $\int_{M^{m\times n}}d\nu_x(A)=\int_{M^{m\times n}}d\nu_x^*(A)=1$, so that $\nu_x^*\in \mathcal P(\mathcal K)$. Similarly, choosing   $f(A)=A$  we deduce that $\bar\nu_x^*=\sum_{r=1}^N\int_{K_r^{\rm qc}}A\,d\nu_x(A)=\bar\nu_x$. In particular $\bar\nu_x^*=Dy(x)$ for some $Dy\in L^{\infty}(\Omega;M^{m\times n})$. If $\varphi$ is finite and quasiconvex, then 
\begin{eqnarray*}
\langle\nu^*_x,\varphi\rangle&\geq&\sum_{r=1}^N\int_{K_r^{\rm qc}}\varphi(\bar\mu_A^r)\,d\nu_x(A)\\ &=&\sum_{r=1}^N\int_{K_r^{\rm qc}}\varphi(A)\,d\nu_x(A)\\
&=&\int_{M^{m\times n}}\varphi(A)\,d\nu_x(A)\geq\varphi(\bar\nu_x), 
\end{eqnarray*}
where we have used the necessity of condition (ii) of Theorem \ref{kp} twice. By construction $\supp \nu^*_x\subset\mathcal K$.  Hence, by the sufficiency part of Theorem \ref{kp}, $\nu^*$ is an $L^{\infty}$ gradient Young measure, which is homogeneous if $\nu$ is homogeneous. Finally, choosing $f$ to be the characteristic function of $K_s$ we see that 
$\nu^*_x(K_s)= \nu_x(K_s^{\rm qc})$ as required.
\end{proof}

\subsection{Domains connected with respect to rigid motion of a convex set}
\lbl{convex}
Let $n>1$.
We recall that two subsets $G_1,G_2$ of $\R^n$ are {\it directly congruent} if
\be
\lbl{c1}
G_1=\xi +QG_2\;\;\text{ for some } \xi\in\R^n, Q\in {\rm SO}(n).
\ee
 
Let $\om\subset\R^n$ be a bounded domain, and let $C\subset\R^n$ be bounded,
open and convex.  We suppose without loss of generality that $0\in C$; this implies in
particular that $\lambda C\subset C$ for any $\lambda\in[0,1]$.

We define the {\it outer radius} $R(C)$ by
\be
\lbl{c3a}
R(C)=\inf\{\rho>0:B(a,\rho)\supset C \text{ for some }a\in \R^n\},
\ee
the {\it inner radius} $r(C)$ by
\be
\lbl{c3b}
r(C)=\sup\{\rho>0:B(a,\rho)\subset C \text{ for some }a\in \R^n\},
\ee
and the {\it eccentricity} $E(C)$ by
\be
\lbl{c3c}
E(C)=\sqrt{1-\frac{r(C)^2}{R(C)^2}}.
\ee
Note that there exists a unique minimal ball $B(a(C),R(C))$ containing $C$, but that there may be
infinitely many maximal balls $B(b(C),r(C))$ contained in $C$.
\begin{defn}
\lbl{c4}
$\om$ is {\it $C$-filled} if any $x\in\om$ belongs to a subset of $\om$ that is directly
congruent to $C$.
\end{defn}
Thus $\om$ is $C$-filled if and only if
\be
\lbl{c5}
\om=\bigcup\{G\subset\om:G \text{ directly congruent to }C\}.
\ee
\begin{defn}
\lbl{c5a}
Let $C_1,C_2$ be subsets of $\om$ directly congruent to $C$. We say that $C_1,C_2$ are
{\it congruently connected}, written $C_1\sim C_2$, if $C_1$ can be  moved continuously to $C_2$
as a rigid body while remaining in $\om$, i.e. there exist continuous maps
$\xi:[0,1]\arr\R^n$, $Q:[0,1]\arr SO(n)$, such that $C_1=\xi(0)+Q(0)C,\;C_2=\xi(1)
+Q(1)C$ and $\xi(t)+Q(t)C\subset\om$
for all $t\in[0,1]$.
\end{defn}
Clearly $\sim$ is an equivalence relation on the family $\mK(C)$ of subsets
of $\om$ that are directly congruent to $C$.
 
\begin{defn}
\lbl{c6}
$\om$ is {\it $C$-connected} if there is an equivalence class of $\mK(C)$
with respect to $\sim$ that covers $\om$. $\om$ is {\it strongly $C$-connected} if
it is $C$-filled and every pair of subsets of $\om$ directly congruent to $C$ are congruently
connected.
\end{defn}
Thus $\om$ is $C$-connected if $\om$ is covered by a collection of directly congruent
copies of $C$ any pair of which can be moved from one to the other as a rigid
body while remaining in $\om$, while $\om$ is strongly $C$-connected if in addition
there is a single equivalence class with respect to $\sim$.
 Example \ref{ex2} below shows that $C$-connectedness does not imply strong
 $C$-connectedness.

  \begin{prop}
  \lbl{c6a}
  Let $0<\lambda\leq 1$, and let $\om$ be convex. Then the subsets of $\om$ of the form
  $a+\lambda\om$, $a\in\R^n$, cover $\om$ and are pairwise congruently connected. In
  particular $\om$ is strongly $\lambda\om$-connected.
  \end{prop}
  \begin{proof}
  Let $x\in\om$. Since $\om$ is convex, $\lambda(\om-x)\subset\om-x$, and hence
  $x\in(1-\lambda)x+\lambda\om\subset\om$. Thus the subsets of $\om$ of the form
  $a+\lambda\om$ cover $\om$.

  If $a_1+\lambda\om$ and $a_2+\lambda\om$ are two such subsets then $t\mapsto
  (1-t)a_1+ta_2+\lambda\om$, $t\in[0,1]$, defines a suitable continuous path of directly 
  congruent subsets of $\om$
  joining them.
  \end{proof}
 
 If $\om$ is $C$-connected then obviously $\om$ is $C$-filled.
The following example shows that if $\om$ is
 $C$-filled then it need not be $C$-connected.
 \begin{example}
\lbl{c7}For $0<\alpha<1$ define $\om^\alpha\subset\R^n$ by
$$\om^\alpha=B(0,1)\cup B((2-\alpha)e_1,1).$$
 Then $\om^\alpha$ is $B(0,1)$-filled but is
only $B(0,r)$-connected for $0<r\leq
r_\alpha=\sqrt{\alpha-\frac{1}{4}\alpha^2}$, since the diameter of the opening
joining the two balls comprising $\om^\alpha$ is $2r_\alpha$.
\end{example}

\begin{prop}
\lbl{c8}
If $\om$ is $C$-filled, it is $\lambda C$-connected for all sufficiently small
$\lambda>0$.
\end{prop}
To prove Proposition \ref{c8} we need the following definition and lemma.
\begin{defn}
\lbl{deltatube}
If $\delta>0$, a $\delta${\it -tube joining} $x_1, x_2\in\om$ is a continuous path $\xi:[0,1]\arr\om$ with
$\xi(0)=x_1, \xi(1)=x_2$ such that $\xi(t)+\overline{B(0,\delta)}\subset\om$ for all $t\in[0,1]$.
\end{defn}
\begin{lem}
\lbl{tube}
Let $\om$ be a bounded domain and let $\ep>0$ be sufficiently small. Then there exists
$\delta=\delta(\ep)>0$ such that any pair of points
$x_1,x_2\in\om$ with $\dist(x_i,\partial\om)\geq\ep$ are joined by a $\delta$-tube.
\end{lem}
\begin{proof}
Fix $\bar x \in\om$ with $\dist(\bar x,\partial\om)\geq\ep$. For $\delta>0$ let
$E_\delta=\{ x\in\om:$  there exists a $\delta$-tube joining $\bar x$   and  $x\}$. We claim that
$E_\delta\supset\{ x\in\om:\dist(x,\partial\om)\geq\ep\}$ for $\delta$ sufficiently small. If not there would
exist $x^{(j)}\in\om$ with $\dist(x^{(j)},\partial\om)\geq\ep$ such that there is no $\frac{1}{j}$-tube
joining $\bar x$ to $ x^{(j)}, j=1,2,\ldots$. But we may assume that $x^{(j)}\arr x$ with $\dist(x,\partial\om)
\geq \ep$. Since $\om$ is connected there is a $\delta$-tube joining $\bar x$ to $x$ for some $\delta>0$,
so that this path followed by the straight line from $x$ to $x^{(j)}$ defines a $\frac{1}{j}$-tube for large $j$,
a contradiction. Hence for $\delta$ sufficiently small any points $x_1, x_2\in\om$ with
 $\dist(x_i,\partial\om)\geq\ep$ are joined to $\bar x$, and hence to each other, by a $\delta$-tube.
\end{proof}

\begin{proof} {\it of Proposition \ref{c8}} 
Let $\ep>0$ be such that $B(0,\ep)\subset C$, and let $\delta=\delta(\ep)$
be as in Lemma \ref{tube}. Pick $\lambda>0$ sufficiently small so that $\lambda C\subset
B(0,\delta)$.

Let $\mE_\lambda(C)=\{ b+\lambda QC: b\in\R^n, Q\in SO(n),
b+\lambda QC\subset a +QC\subset\om$ for some $a\in\R^n\}$. Since $\om$ is $C$-filled,
$\mK(C)$ covers $\om$, and by Proposition \ref{c6a} applied to $a+QC$, so does
$\mE_\lambda(C)$.

Suppose that $b_i+\lambda Q_iC\in \mE_\lambda(C),\;i=1,2$. Then by
Proposition \ref{c6a}, $b_i+\lambda Q_iC$ is congruently connected to $a_i+\lambda
Q_iC$,  where $a_i+Q_iC\subset\om$, $i=1,2$. But $a_i+\lambda Q_iC\subset B(a_i,\delta)$
and $\dist(a_i,\partial\om)\geq\dist (a_i,\partial(a_i+Q_iC))\geq\ep$. Hence by
Lemma \ref{tube} there exists a $\delta$-tube  $\xi:[0,1]\arr\om$
 joining $a_1$ and $a_2$. Let $Q:[0,1]\arr SO(n)$ be continuous with $Q(0)=Q_1$,
 $Q(1)=Q_2$. Then $\xi(t)+\lambda Q(t)C\subset\om$ for all $t\in[0,1]$, and so
 $a_1+\lambda Q_1C$, $a_2+\lambda Q_2C$ are congruently connected. Hence $b_1+\lambda
 Q_1C$, $b_2+\lambda Q_2C$ are congruently connected. Hence $\om$ is
 $\lambda C$-connected.
\end{proof}

The following example shows that Proposition \ref{c8} does not hold for strong
$C$-connectedness. That is, a bounded domain may be $C$-filled but not strongly
$\lambda C$-connected for all sufficiently small $\lambda>0$.
\begin{example}
\lbl{ex2}
Let $C\subset\R^2$ be the interior of the equilateral triangle of side $1$ with
vertices at $(0,0),\;(\frac{\sqrt{3}}{2},{\pm\half})$. Let $\om$ consist of a large
ball $B(0,R)$ from which the origin $(0,0)$ and the points $A_i=(\frac{2^{-i}}{\sqrt{3}},
\frac{2^{-i}}{3}),\;B_i=(\frac{2^{-i}}{\sqrt{3}},-\frac{2^{-i}}{3}),\;i=0,1,2,..,$ are
removed. The points $A_i$, $B_i$ lie on the half-lines $L_A$ and $L_B$ defined by
$\{\sqrt{3}x_2-x_1=0,x_1\geq 0\}$ and $\{\sqrt{3}x_2+x_1=0,x_1\geq 0\}$ respectively,
which meet at the origin at an angle of $60^\circ$.
Then $\om$ is $C$-filled. Indeed $\om$ consists of $C$ together
with points lying outside $C$ which are clearly inside congruent copies of $C$
lying in $\om$ (for example, for the  points on $L_A,\,L_B$ we can use an equilateral
triangle of side $1$ which lies outside $C$ except for a small region near one of
its vertices).

Now consider the open equilateral triangle $\Delta$ of side $1$ with vertices at
$(\frac{2}{\sqrt{3}},0)$ and $(\frac{1}{2\sqrt{3}},\pm\half)$, and   the corresponding
scaled
equilateral triangles
$\Delta_i=2^{-i}\Delta$ of side $2^{-i}$. Note that $\Delta_i\subset\om$, and that the
edges of $\Delta_i$ intersect $L_A$ and $L_B$ in the points $A_i, A_{i+1}$ and
$B_i,B_{i+1}$ respectively. We
claim that $\Delta_i$ cannot be continuously moved to a position far from the origin
while remaining in $\om$. This is even true for a slightly smaller equilateral triangle
contained in $\Delta_i$. A rigorous proof can be constructed by noting that the width of
$\Delta_i$, that is the minimal distance between parallel lines that enclose $\Delta_i$,
is $2^{-(i+1)}\sqrt{3}$, which is greater than any of the distances of the openings
through which it would have to pass, namely $|A_iA_{i+1}|=|B_iB_{i+1}|=\frac{2^{-i}}{3}$ and
$|A_iB_i|=\frac{2^{1-i}}{3}$ (see Strang \cite{strang82}). Hence $\om$ is not strongly $\lambda
C$-connected for sufficiently small $\lambda>0$.
\end{example}

\begin{prop}
\lbl{cone}
The bounded domain $\om$ is $C$-connected for some bounded open convex $C$ if and
only if $\om$ satisfies the cone condition with respect to some cone $C^*$.
\end{prop}
\begin{proof}
Let $\om$ satisfy the cone condition with respect to $C^*$. If $x\in\om$ with
$x+QC^*\subset\om$ then
$x\in x+Q(C^*-\ep e_1)\subset\om$ for $\ep>0$ sufficiently small. Hence  $\om$
is $(\text{int}\,C^*)$-filled, and hence, by Proposition \ref{c8},
$\lambda(\text{int}\,C^*)$-connected for sufficiently small $\lambda>0$.

Conversely, let $\om$ be $C$-connected for some $C$. Since $C$ is convex it is Lipschitz (see Morrey \cite[p.~72]{morrey66}) and hence satisfies the cone condition with respect to some $C^*$.
Since $\om$ is $C$-filled it follows immediately that $\om$ also satisfies the cone condition
 with respect to $C^*$.
\end{proof}
Despite this result, the concept of $C$-connectedness is of interest
 since we will show that the constants in the transition layer estimate of Theorem \ref{transest}
 can be chosen to
 depend on $\om$ through $C$.
\subsection{The Vitali Covering Lemma}
The following simpler version \cite{stein70} of the  Vitali covering lemma is used in an important way
in the transition layer estimate.
\begin{lem}
Let $G$ be a measurable subset of $\R^n$ which is covered by the union of a family of
balls $\{B_i\}$ of bounded diameter.  From this family we can select a countable or
finite   disjoint subsequence $B_{i(k)}, k = 1, 2, \dots$ such that
\begin{displaymath}
\sum_k \mL^n(B_{i(k)}) \ge c_n \mL^n(G).
\end{displaymath}
Here, $c_n>0$ depends only on the dimension $n$.  The choice $c_n = 5^{-n}$ suffices.
\lbl{stein}
\end{lem}
\section{Incompatible sets} \label{is}
 Let $\om\subset \R^n$ be a bounded domain.    Fix $p$ with $1 \leq p
\leq \infty$.  

\begin{defn}
\lbl{ymd}
The   closed subsets $K_1,\ldots,K_N$ of $M^{m\times
n}$ are $L^p$
 {\it  incompatible} if they are disjoint, and if
whenever $\nu=(\nu_{x})_{x\in\om}$ is an $L^p$ gradient Young measure  
satisfying
$$\supp\nu_{x} \subset \bigcup _{r=1}^{N} K_{r} \ \ \text { for
a.e. }
 x\in \om,$$
then for some $i, \ 1 \leq i \leq N$,
$$ \supp\nu_{x} \subset K_{i} \ \  \text  { for   a.e. } x\in
\om .$$
\end{defn}

\begin{rem}
\lbl{rem1}\rm  a. It is easily seen that the sets $K_{1},\ldots,
K_{N}$ are $L^p$ incompatible if and only if for each $i=1,\ldots, N$ the pair of
sets $K_{i},
\ \bigcup_{r \neq i} K_{r}$ are $L^p$ incompatible.
 The latter condition is
obviously necessary, and it is sufficient since if $\supp
\nu_{x} \subset
\bigcup_{r=1}^{N} K_{r}$ for a.e. $x\in \Omega$ then we have
for each
$i$ either
$$\supp\nu_{x} \subset K_{i} \ \  \text{ for a.e. } x \in
\Omega$$
or
$$ \supp \nu_{x} \subset \bigcup_{r \neq i} K_{r},$$
and  $\bigcap_{i=1}^{N} \bigcup_{r \neq i} K_{r}$  is empty.
For this reason we can often restrict attention to the case
$N=2$.

\medskip

\noindent b.  The definition does not
 depend on $\Omega$.  By the above remark we may
assume that $N=2$. So let
  $K_1, K_2$
be  $L^p$ incompatible with respect to $\om$
and let $\tilde\om\subset\R^n$ be another bounded domain. Let
$D\tilde
y^{(j)}$ be a
sequence of gradients that is relatively weakly compact in $L^p(\tilde\om;M^{m\times
n})$ with corresponding 
gradient Young measure $(\tilde\nu_x)_{x\in\tilde\om}$
satisfying
$\supp \tilde\nu_x\subset K_1\cup K_2$ for a.e.
$x\in\tilde\om$.
Let $G_1=\{x\in\tilde\om:\supp\tilde\nu_x\cap
K_1\ne\emptyset\}$,
$G_2=\{x\in\tilde\om:\supp\tilde\nu_x\cap K_2\ne\emptyset\}$
and suppose for
contradiction that ${\mathcal L}^n (G_1)>0,  {\mathcal L}^n
(G_2)>0$. By
hypothesis we have that
\be
\lbl{1.0}
{\mathcal
L}^n(\tilde\om\backslash(G_1\cup G_2))=0.
\ee
Let $x_1, x_2$ be Lebesgue points of $G_1, G_2$
respectively. Since
$\tilde\om$ is connected there is a continuous arc $x(t),$
$t\in [0,1],$ with
$x(0)=x_1$, $x(1)=x_2$  and $x(t)\in \tilde\om$ for all
$t\in[0,1]$. Then there
exists $\ep_1>0$ such that $x(t)+\ep\om\subset\tilde\om$ for
all $t\in[0,1]$,
$0<\ep\leq\ep_1$. Fix $0<\ep\leq\ep_1$ sufficiently small so
that
${\mathcal L}^n((x_1+\ep\om)\cap G_1)>0$ and
${\mathcal L}^n((x_2+\ep\om)\cap G_2)>0$, which is possible
since $x_1,x_2$ are
Lebesgue points.
Define for $i=1,2$
$$f_i(t)=\frac{{\mathcal L}^n((x(t)+\ep\om)\cap G_i)}
{\ep^n{\mathcal L}^n(\om)}.$$
Then each $f_i$ is continuous in $t$, and by construction
$f_1(0)>0, f_2(1)>0$.
But from \eqref{1.0}
$$f_1(t)+f_2(t)\geq 1,$$
from which it follows easily that there exists $t_0\in
[0,1]$ with
$0<f_i(t_0)\leq 1$ for $i=1,2$, i.e.
\be
\lbl{1.1}
{\mathcal L}^n((x(t_0)+\ep\om)\cap G_1)>0,\;\;{\mathcal
L}^n((x(t_0)+\ep
\om)\cap G_2)>0.
\ee
Now let $y^{(j)}(x)=\ep^{-1}\tilde y^{(j)}(x(t_0)+\ep
x)$, which is well defined because $\tilde y^{(j)}\in L^1_{\rm loc}(\tilde\om;\R^m)$. Then
$Dy^{(j)}(x)=D\tilde y^{(j)}(x(t_0)+\ep x)$ and so
$Dy^{(j)}$ is relatively weakly compact in
 $L^p(\om;M^{m\times n})$ and has Young measure
 \be
 \lbl{1.2}
  \nu_x=\tilde\nu_{x(t_0)+\ep x}, \;\; x\in\om.
 \ee
 Furthermore $\supp \nu_x\subset K_1\cup K_2$ for a.e.
$x\in\om$, and
 so either $\supp\nu_x\subset K_1$ for a.e. $x\in\om$
  or $\supp \nu_x\subset K_2$  for
 a.e. $x\in\om$. This implies that $\supp\tilde\nu_x\subset
K_1$ for a.e.
  $x\in x(t_0)+\ep\om$
 or $\supp\tilde\nu_x\subset K_2$ for
 a.e. $x\in x(t_0)+\ep\om$, contradicting
 \eqref{1.1}. 
 
 \medskip

\noindent c. If the sets $K_{1},\ldots,K_{N}$ are compact
then the definition is
independent of $p$.  Consequently in
this case we say
simply that $K_{1},\dots,K_{n}$ are  {\it incompatible}.
 In fact suppose that $K_{1},\ldots,
K_{N}$ are compact
and $L^\infty$ incompatible.  Let $1 \leq p < \infty$
and let
$Dy^{(j)}$ be weakly relatively compact in $L^{p}$ and have Young measure
$(\nu_{x})_{x
\in \Omega}$ with  $\supp\nu_{x} \subset \bigcup_{r=1}^{N}
K_{r}$ for a.e.
$x\in \Omega.$  Then by Lemma \ref{zhanglemma} there is a sequence of 
gradients $Dz^{(j)}$ which is bounded in $L^{\infty}$ and
has the same Young
measure, so that $K_{1},\dots,K_{n}$ are $L^p$ incompatible.

\medskip

\noindent d.  {\it The case $p=1$.}   An alternative definition of $L^1$ incompatible sets would have been to replace the weak relative compactness of $Dy^{(j)}$ by boundedness of $Dy^{(j)}$ in $L^{1}(\Omega; M^{m\times
n})$. But with such a modification no family of disjoint
closed subsets of
$M^{m\times n}$ would be  $L^1$ incompatible.  In fact
if $K_1,K_2$ were a pair of $L^1$ incompatible sets in this sense, we could
let $\Omega=[-1,1]^{n}, A \in K_{1}, \ B \in K_{2}$, and
define
\begin{displaymath}
y^{(j)}(x)=\left\{ \begin{array}{ll}
Ax & {\rm if} \ x_{1} \leq 0, \\
jx_{1}Bx+(1-jx_{1})Ax & {\rm if} \ 0<x_{1}< \frac{1}{j}, \\
Bx & {\rm if} \ x_{1} \geq \frac{1}{j}.
\end{array}\right.
\end{displaymath}

\noindent Then
$$Dy^{(j)}(x)=jx_{1}B+(1-jx_{1})A+j(B-A)x \otimes e_{1}$$
for $0<x_{1} < \frac{1}{j}$, so that
$$\int_{[-1,1]^{n}} \mid Dy^{(j)} \mid dx \leq C < \infty.$$
 But the corresponding Young measure $(\nu_{x})_{x \in
\Omega}$ is given by
\begin{displaymath}
\nu_{x}=\left\{ \begin{array}{ll}
\delta_{A} \ {\rm if} \ x_{1} < 0, \\
\delta_{B} \ {\rm if} \ x_{1} > 0.
\end{array}\right.
\end{displaymath}
\end{rem}

\begin{defn}
\lbl{hymd}
 The closed subsets $K_{1},\ldots,K_{n}$ of $M^{m\times n}$
are  {\it homogeneously $L^p$
incompatible} if they are disjoint, and if
whenever $\nu$ is a homogeneous $L^p$ gradient Young measure
generated by a sequence
 satisfying
$$ \supp  \nu \subset \bigcup_{r=1}^{N}K_{r},$$
 then for some $i$, $1 \leq i \leq N$,
$$\supp \nu \subset K_{i}.$$
\end{defn}
 The same arguments as in Remark \ref{rem1} show that this
definition too is
independent of $\Omega$ and, in the case when the $K_{r}$
are compact, also of
$p$ with $1 \leq p \leq \infty$ (in which case we say that the
$K_{r}$ are
 {\it homogeneously incompatible}).

\begin{defn}
  The closed subsets
$K_{1},\ldots,K_{N}$ of $M^{m\times n}$ are
$L^p$ {\it gradient incompatible} if they are disjoint, and if
whenever $Dy \in L^{p}(\Omega;M^{m\times n})$
 with
$$Dy(x) \in  \bigcup_{r=1}^{N}K_{r} \ \ {\rm for \ a.e.}
\ x \in \Omega$$
 then
$$Dy(x) \in K_{i} \ \  {\rm for \ a.e.} \ x \in \Omega$$
 for some $i$.
\end{defn}
Again the definition is independent of $\om$ and, in the case when the $K_r$ are
compact, also of $p$ with $1\leq p\leq\infty$ (in which case we say that the
$K_r$ are {\it gradient incompatible}).

Note that if $n=1$ or $m=1$ then no pair of disjoint nonempty closed sets $K_1, K_2$ can be homogeneously $L^p$ incompatible, since if $A_1\in K_1$, $A_2\in K_2$ then $\rank (A_1-A_2)=1$, so that $\frac{1}{2}(\delta_{A_1}+\delta_{A_2})$ is a homogeneous $L^\infty$ gradient Young measure supported nontrivially on $K_1\cup K_2$; similarly $K_1$ and $K_2$ are not $L^\infty$ gradient incompatible. Thus most of the results of this paper are only relevant for $n\geq 2$ and $m\geq 2$.

 Of course if $K_{1},\ldots,K_{N}$ are $L^p$ incompatible they are also
$L^p$ gradient incompatible.  However the converse is false 
(for other examples see Examples \ref{3matrices}, \ref{4matrices}).
\begin{example}
 Let $m=n=2$, $\{e_1, e_2\}$ be an orthonormal basis of $\R^2$,  $K_{1}=\{\1\},  K_{2}=\{{\bf 0},2e_{2}
\otimes e_{2}\}$.
Then $K_{1},K_{2}$ are not incompatible.  To see this note
that $\1=e_{1} \otimes
e_{1}+e_{2} \otimes e_{2}$,  so that
\be
\nonumber
\1-e_{2} \otimes e_{2}&=&e_{1} \otimes e_{1}\\
e_{2} \otimes e_{2}&=& \frac{1}{2}({\bf 0} + 2e_{2} \otimes
e_{2}).
\ee
 Thus a double laminate can be constructed having
homogeneous gradient Young measure
$$\nu= \frac{1}{2} \delta_{\1} + \frac{1}{4} \delta_{\bf 0}+
\frac{1}{4} \delta_{2e_{2}\otimes e_2}.$$
However, $K_{1},K_{2}$ are gradient incompatible.
In fact if
$Dy(x)\in K_{1} \cup K_{2}$ a.e. in
$\Omega=(0,1)^{2}$, we have
$$Dy(x)= \lambda(x)\1+2\mu (x)e_{2} \otimes e_{2},$$
where $\lambda(x)\mu (x)=0$,  $\lambda(x) \in
\{0,1\}$ and $\mu(x)
\in \{0,1\}$ almost everywhere. Hence $y_{,1}= \lambda e_{1},
y_{,2}=(\lambda+2\mu)e_{2}$ and so
$$\lambda_{,2}=(\lambda+2\mu)_{,1}=0$$
 in the sense of distributions.  Hence
$\lambda=\lambda(x_{1}), \lambda +2\mu=f(x_{2})$
from which it follows easily that either $\lambda=0$ a.e. or
$\lambda=1$ a.e.
as required.
\end{example}

\subsection{Characterization of incompatible
sets}

Clearly if $K_{1},\ldots,K_{N}$ are $L^p$ incompatible they
are homogeneously
$L^p$ incompatible.  We do not know if the converse holds,
even if the $K_{r}$ are
compact (but see Remark \ref{twograds} for the case $m=n=2$).  It is possible to characterize homogeneously incompatible
sets in terms
of quasiconvex functions.  We first prove some preliminary results relating incompatibility of the sets $K_r$ to that of the sets $K_r^{\qc}$.

\begin{lemma}
\lbl{prop1}
 If $K_{1},\ldots, K_{N}$ are homogeneously  $L^\infty$ incompatible
  then  \\
$(\bigcup_{r=1}^{N}K_{r})^\qc$ is the disjoint union of the
sets $K_{r}^\qc$.
\end{lemma}
\begin{proof}
 We first show that $K_{r}^\qc \bigcap K_{s}^\qc$ is empty if
$r \neq s$.  Suppose
the contrary, that there exists an $A \in K_{r}^\qc \bigcap
K_{s}^\qc$.
Then there exist homogeneous $L^\infty$ Young measures
$\nu^{r}$ and $\nu^s$
with $\supp\nu^r \subset K_{r}$, $\supp \nu^s \subset
K_{s}$ and
$\bar \nu^r=\bar \nu^s=A$.  But the set of homogeneous
$L^\infty$ Young
measures with a given centre of mass $A$ is convex
(Kinderlehrer \& Pedregal
\cite{kinderlehrerpedregal91}),
and thus $\nu= \frac{1}{2}(\nu^r+\nu^s)$ is a
homogeneous $L^\infty$ Young
measure with $\supp\nu \subset K_{r} \cup K_{s}$ and both
$\supp
\nu \cap K_{r}$ and $\supp \nu \cap K_{s}$ nonempty.  Thus
$K_{1},\ldots,K_{N}$ are not homogeneously $L^\infty$ incompatible.
 
 Next, let $A \in (\bigcup_{r=1}^{N}K_{r})^\qc$.  Then $A= \bar
\nu$ for some
homogeneous $L^\infty$ Young measure $\nu$ with $\supp \nu
\subset
\bigcup_{r=1}^{N}K_{r}$, and by hypothesis $\supp  \nu
\subset K_{i}$ for
some $i$.  Hence $A \in K_{i}^\qc$, completing the proof.
\end{proof}
\begin{prop}
\lbl{quasi}
The   compact sets $K_1,\ldots,K_N$ are incompatible (resp. homogeneously incompatible) if and only if $K_1^\qc,\ldots,K_N^\qc$ are  incompatible (resp. homogeneously incompatible).
\end{prop}
\begin{proof}
Suppose that $K_1,\ldots,K_N$ are incompatible. By Lemma \ref{prop1} the $K_r^{\rm qc}$ are disjoint. Let $\nu=(\nu_x)_{x\in\Omega}$ be an $L^{\infty}$ gradient Young measure with $\supp \nu_x\subset \bigcup_{r=1}^{N}K_{r}^\qc$ for a.e. $x\in\om$. Then by Theorem \ref{splitting} there is an $L^{\infty}$ gradient Young measure $\nu^*=(\nu^*_x)_{x\in\om}$ with $\supp \nu^*_x\subset \bigcup_{r=1}^N K_r$ and $\nu^*_x(K_r)=\nu_x(K_r^\qc)$ for all $r$ and a.e. $x\in\om$. Since the $K_r$ are incompatible, we have that $\nu^*_x(K_i)=1$ for some $i$ and a.e. $x\in\om$. Hence $\nu_x(K_i^\qc)=1$ for a.e. $x\in\om$ and thus $K_1^\qc,\ldots, K_N^\qc$ are incompatible. The same argument shows that if the $K_r$ are homogeneously incompatible then so are the $K_r^\qc$. The converse direction is obvious.
\end{proof}
\begin{thm}
\lbl{homcom}
 The   compact sets $K_{1},\ldots, K_{N}$ are
homogeneously incompatible if
and only if
 
\noindent{\rm (}i{\rm )}  the sets $K_{r}^\qc$, $r=1,\ldots, N$, are
disjoint,

\noindent{\rm (}ii{\rm )} for each $i=1, \ldots, N$ the function
$\varphi_{i}: M^{m\times n} \longrightarrow
[0,\infty]$ defined by
\begin{displaymath}
\varphi_{i}(A)=\left\{ \begin{array}{ll}
1 & {\rm if} \ A \in K_{i}^\qc ,\\
0 & {\rm if} \ A \in \bigcup_{r \neq i} K_{r}^\qc, \\
+ \infty & {\rm otherwise,}
\end{array}\right.
\end{displaymath}
 is quasiconvex.
\end{thm}
\begin{proof}
 Let $K_{1},\ldots,K_{N}$ be homogeneously incompatible.  Then
$(i)$ holds by
Lemma \ref{prop1}.  To prove (ii), by Lemma \ref{qcequiv} with $K=\bigcup_{r=1}^NK_r$ it suffices to show that
\begin{equation}\label{phii}
\langle\mu,\varphi_i\rangle\geq\varphi_i(\bar\mu)
\end{equation}
for any homogeneous $L^{\infty}$ gradient Young measure $\mu$. Since \eqref{phii} obviously holds if 
 $\langle\mu,\varphi_i\rangle=\infty$, we may assume that $\supp\nu\subset\bigcup_{r=1}^NK_r^\qc$. Then it follows from Proposition \ref{quasi} that $\supp\mu\subset K_j^\qc$ for some $j$, so that also $\bar\mu\in K_j^\qc$. Thus if $j\neq i$ both sides of \eqref{phii} are zero, while if $j=i$ then both sides are one.

 Conversely, suppose that (i) and (ii) hold, and let
$\nu$ be a homogeneous $L^{\infty}$ gradient Young
measure with $\supp \ \nu \subset
\bigcup_{r=1}^{N}K_{r}$.  Then $\nu =
\sum_{r=1}^{N} \lambda_{r} \nu^r$, where $\lambda_{r} \geq
0, \sum_{r=1}^{N}
\lambda_{r}=1$ and $\nu^r$ is a probability measure with
${\rm supp} \ \nu^r
\subset K_{r}$.  For any $k$ we have (since $\varphi_{k}$ is
quasiconvex)

$$\varphi_{k}(\bar{\nu}) \leq \langle\nu, \varphi_{k}\rangle= \lambda_{k}.$$
 In particular $\varphi_{k}(\bar{\nu}) < \infty$
and so $\bar{\nu} \in K_{i}^\qc$
for some $i$.  Choosing $k=i$ we obtain $\lambda_{i} \geq 1$
and so
$\nu=\nu^{i}$ and $\supp  \nu \subset K_{i}$.  Hence
$K_{1},\ldots,K_{N}$ are
homogeneously incompatible.
 \end{proof}
\begin{thm}
\lbl{com}
The   compact sets $K_1,...,K_N$ are incompatible if and only if

\noindent{\rm (}i{\rm )}  The sets $K_1^{\text qc},...,K_N^{\text qc}$ are gradient incompatible,

\noindent{\rm (}ii{\rm )} for each $i=1, \ldots, N$ the function
$\varphi_{i}: M^{m\times n} \longrightarrow
[0,\infty]$ defined by
\begin{displaymath}
\varphi_{i}(A)=\left\{ \begin{array}{ll}
1 & {\rm if} \ A \in K_{i}^\qc ,\\
0 & {\rm if} \ A \in \bigcup_{r \neq i} K_{r}^\qc, \\
+ \infty & {\rm otherwise,}
\end{array}\right.
\end{displaymath}
 is quasiconvex.
\end{thm}
\begin{proof}
Let $K=\bigcup_{r=1}^NK_r$.
Suppose that $K_1,\ldots,K_N$ are incompatible. Then $K_1,\ldots,K_N$ are homogeneously
incompatible, so that by Lemma \ref{prop1} and Theorem \ref{homcom} the sets $K_r^\qc$ are disjoint,
$K^{\qc}=\bigcup_{r=1}^NK_r^{\qc}$ and (ii) holds.
To show that the $K_r^\qc$ are gradient incompatible, suppose that $Dy\in L^{\infty}(\om; M^{m\times n})$
satisfies $Dy(x)\in K^\qc$ almost everywhere.  It follows from Theorem \ref{splitting} applied to the gradient Young measure $\nu=(\delta_{Dy(x)})_{x\in\om}$ that there exists a gradient Young measure $(\nu^*_x)_{x\in\om}$ with $\supp\nu^*_x\subset K$ and $\bar\nu^*_x=Dy(x)$ almost everywhere. But then by hypothesis $\supp\nu^*_x\subset K_s$ a.e. for some $s$ and so $Dy(x)\in K_s^\qc$ almost everywhere.

Conversely, let (i) and (ii) hold, and let $(\nu_x)_{x\in\om}$ be an $L^{\infty}$ gradient Young measure with $\supp\nu_x\subset\bigcup_{r=1}^NK_r$ almost everywhere. Then, for a.e. $x\in\om$, $\nu_x$ is a homogeneous $L^{\infty}$ gradient Young measure, and so by Theorem \ref{homcom} $\supp\nu_x\subset K_{r(x)}$ for some $r(x)$, and hence $\bar\nu_x\in K_{r(x)}^\qc$. Thus $Dy(x)=\bar\nu_x\in\bigcup_{r=1}^NK_r^\qc$ a.e., and so $Dy(x)\in K_s^\qc$ a.e. for some $s$. Since the $K_r^\qc$ are disjoint, $r(x)=s$ a.e. and hence $\supp\nu_x\subset K_s$
almost everywhere.
\end{proof}

\begin{cor}
\lbl{quasiz}
The   compact sets $K_1,\ldots,K_N$ are incompatible if and only if $K_1,\ldots,K_N$ are homogeneously incompatible and $K_1^\qc,\ldots,K_N^\qc$ are gradient incompatible.
\end{cor}
\begin{proof}
This follows immediately from Theorems \ref{homcom}, \ref{com}.
\end{proof}
\begin{rem}\lbl{twograds}\rm
When $m=n=2$, Kirchheim \& Sz\'ekelyhidi   \cite{kirchheimszekelyhidi08}, using results from Faraco \& Sz\'ekelyhidi \cite{faracoszek}, show that two disjoint compact sets $K_1,K_2$ are incompatible if and only if $(K_1\cup K_2)^{\rm rc}$ is the disjoint union of $K_1^{\rm rc}$ and $K_2^{\rm rc}$, where $K^{\rm rc}$ denotes the rank-one convexification of a compact set $K\subset M^{m\times n}$ defined by $$K^{\rm rc}=\{A\in M^{m\times n}:\varphi(A)\leq\max_{B\in K}\varphi(B) \mbox{ for all finite rank-one convex }\varphi\}.$$ 
 They also show that $K_1, K_2$ are incompatible if and only if they are homogeneously incompatible, and if and only if they are incompatible for laminates. Since Sz\'ekelyhidi \cite{szekelyhidi2005} has provided a simple and algorithmically testable criterion for incompatibility of $K_1, K_2$ for laminates, this completely classifies incompatible compact subsets of $M^{2\times 2}$.  Using these results Heinz \cite{heinz} found necessary and sufficient conditions for incompatibility for compact sets $K_1, K_2\subset M^{2\times 2}$ that are left invariant under $SO(2)$ and consist of matrices with positive determinant.
\end{rem}

\subsection{Examples}

A necessary condition that $K_{1},\ldots,K_{N}$ be
homogeneously $L^\infty$ incompatible
is that there are no rank-one connections between any of the
$K_{r}$.  This follows from Lemma \ref{prop1} and the fact that quasiconvex
sets are rank-one convex.   However the absence of
such rank-one connections
is not sufficient (see the well-known Example \ref{4matrices} below).
 \begin{example} [Two matrices]\lbl{twomatrices}
 If $K_{1}=\{A\}, K_{2}=\{B\}$, where $A,B \in
M^{m\times n}$ with  $\rank(A-B)>1$,
then $K_{1},K_{2}$ are $L^p$ incompatible for any $p>1$.  We
give two proofs of
this fact.

\noindent{\it First proof}
Let $(\nu_{x})_{x \in \Omega}$ be an $L^{p}$ gradient Young measure with $\supp \nu_{x} \subset
\{A,B\}$ for a.e. $x \in \Omega$, i.e. $\nu_{x}=
\lambda(x)\delta_{A}+(1- \lambda
(x))\delta_{B}$ where $0 \leq \lambda(x) \leq 1$.  In
particular $\supp
\nu_{x}$ is contained in a bounded set for a.e. $x$, and so
 $(\nu_{x})_{x \in \Omega}$ is an $L^{\infty}$ gradient Young measure by Lemma \ref{zhanglemma}.  Thus
by the results
in  \cite{j32}, based on the weak
continuity of minors,
$\nu_{x}=\delta_{A}$ for a.e. $x \in \Omega$ or
$\nu_{x}=\delta_{B}$ for a.e. $x
\in \Omega$ as required.
\\
 \noindent{\it Second proof}
 This was communicated to us by V. \v{S}ver\'{a}k (see \cite{Sverak93a} and M\"uller \cite[Section 2.6]{muller99}).    Without loss of
generality we suppose that $A=0$
and define $h(D)=(\dist(D,L))^{2}$ for $D \in M^{m\times n}$,
where $L=\{tB;t\in\R\}$.  Thus
$$h(D)=|D|^{2}- \frac{(D \cdot B)^{2}}{ |B |^{2}}.$$
 $h$ is quadratic and strongly elliptic, since $tB$
is not rank-one for any $t$.
If $Dy^{(j)}$ is bounded in $L^{p}(\Omega; M^{m\times n})$ with
 $\supp\nu_{x}
\subset \{A,B\}$ then $Dh(Dy^{(j)}) \arr 0$ in
measure, and hence
$Dh(Dy^{(j)}) \arr 0$ strongly in $L^{s}
(\Omega;M^{m\times n})$ if $1<s<p$. So
$$\Div \ Dh(Dy^{(j)})= \Div \ f^{(j)}, \ x \in \Omega,$$
where $f^{(j)} \arr 0$ strongly in
$L^{s} (\Omega;M^{m\times n})$.  By elliptic
regularity theory  this implies that
$Dy^{(j)}$ is relatively compact
in $L_{\text{loc}}^{s}  (\Omega;M^{m\times n})$, so that 
$\nu_{x}=\delta_{Dy(x)}$ a.e. for some $y$ with $Dy(x)\in\{A,B\}$ almost everywhere. But elliptic regularity implies that $Dy$ is smooth, so that 
$\nu_x=\delta_A$ a.e. or $\nu_x=\delta_B$ a.e., as required.
\end{example}
\begin{example} [3 matrices]
\lbl{3matrices}
 Let $K_{1}=\{A_{1}\}, K_{2}=\{A_{2}\},
K_{3}=\{A_{3}\}$, where $A_{r} \in
M^{m\times n}$ with $\rank(A_{r}-A_{s})>1$ for $r \neq s$.  Then
$K_{1},K_{2},K_{3}$
are incompatible.  This is a consequence of a deep result of
\v{S}ver\'{a}k \cite{Sverak93a,sverak93e},  which uses in particular the result of Zhang \cite{zhang91}  that $K_{1},K_{2},K_{3}$ are gradient incompatible. See also the discussion after Corollary \ref{open}.
\end{example}
\begin{example}[4 matrices]
 \lbl{4matrices} Let $K_{r}=\{A_{r}\}, 1 \leq r \leq 4$, with
$\rank (A_{r}-A_{s})>1$ for $r \neq
s$. Then $K_{1},\ldots,K_{4}$ are not in general incompatible.  This follows
from the construction of \cite{Bhattetal94} that was motivated by
the example of  \cite{aumann86}
and Tartar \cite{tartar93}. Chleb\'ik \& Kirchheim \cite{chlebik} showed that $K_1,\ldots,K_4$ are nevertheless gradient incompatible.
\end{example}
\begin{example} [5 matrices]
\lbl{5matrices}
Let $K_{r}=\{A_{r}\}, 1 \leq r \leq 5$, with
$\rank (A_{r}-A_{s})>1$ for $r \neq
s$. Then $K_{1},\ldots,K_{5}$ are not in general gradient incompatible (Kirchheim \& Preiss \cite{kirchheim01,kirchheim03}).
\end{example}
\begin{example}[Incompatible energy wells
in $M^{2\times 2}$]
  Let $K_{r}= SO(2)U_{r}, 1 \leq r \leq N$, where
$U_{r}=U_{r}^{T}>0$ and there
are no rank-one connections between the different $K_{r}$.
Then
$K_{1},\ldots,K_{N}$ are incompatible.  This follows from the
result of Firoozye
\cite{Bhattetal94,firoozye90} and \v{S}ver\'{a}k \cite{sverak93e}.
\end{example}
\begin{example}[Incompatible energy wells
in $M^{3\times 3}$]
Let $K_{1}=SO(3)U_{1}, K_{2}=SO(3)U_{2}$, where
$U_{1}=U_{1}^{T}>O$,
$U_{2}=U_{2}^{T}>O$, and $\rank (A_{1}-A_{2})>1$ for all
$A_{1} \in K_{1},
A_{2} \in K_{2}$.  Then it is not known whether in general
$K_{1},K_{2}$
are incompatible.  However under stronger
conditions on $U_{1},U_{2}$
incompatibility is proved by Dolzmann, Kirchheim, M\"uller \& {\v{S}ver\'ak} \cite{dolzmann00}  
(see also Matos \cite{matos92} and Kohn, Lods \& Haraux \cite{kohnetal2000}). In this case incompatibility is
equivalent to the two-well rigidity estimate of Chaudhuri \& M\"uller \cite{chaudhurimuller2004}, as proved by De Lellis \& Sz\'ekelyhidi \cite{delellisszekelyhidi2006} using the transition layer technique from (earlier expositions of) the present paper. Chaudhuri \& M\"uller \cite{chaudhurimuller2006} used their rigidity estimate to study the scaling behaviour of thin martensitic films. 
If $K_{1},K_{2},K_{3}$ are three such energy wells
without rank-one connections
then it is shown in \cite{Bhattetal94} that $K_{1},K_{2},K_{3}$ need not
be incompatible,
using Example \ref{4matrices}.
\end{example}

\section{The transition layer estimate}
\lbl{trans}
In this section we suppose that $K_1,\ldots,K_N$ are disjoint compact subsets of $M^{m\times n}$. Given
$y\in W^{1,p}(\om;\R^m)$ and $\ep>0$ we consider for $r=1,\ldots,N$ the sets
$$\om_{r,\ep}(y):=\{x\in\om:Dy(x)\in N_\ep(K_r)\},$$
where
$$N_\ep(K):=\{ A\in M^{m\times n}:\dist(A,K)\leq\ep\},$$
and the corresponding `transition layer'
$$T_\ep(y):=\{ x\in\om:Dy(x)\not\in\bigcup_{r=1}^NN_\ep(K_r)\}.$$
The main result is
\begin{thm}
\lbl{transest}
Let $1<p<\infty$ and let $\om$ be $C$-connected. Then
$K_1,\ldots,K_N$ are incompatible if and only if  there exist constants $\ep_0>0$
and $\gamma>0$
 such that if $0\leq\ep<\ep_0$ and
$y\in W^{1,p}(\om;\R^m)$ then
\be
\lbl{4.1}
\int_{T_\ep(y)}(1+|Dy|^p)\,dx\geq\gamma\max_{1\leq r\leq N}\min(\mL^n(\om_{r,\ep}(y)),
\mL^n(\bigcup_{s\neq r}\om_{s,\ep}(y))).
\ee
The constant $\ep_0$ can be chosen to depend only on the eccentricity $E(C)$, the sets $K_1,\ldots, K_N$ and $p$,
while the constant $\gamma$ can be chosen to depend only on these quantities and $\mL^n(C)/\mL^n(\om)$.
\end{thm}
\begin{rem}
\lbl{rem3.1}\rm
An alternative way of writing the right-hand side of \eqref{4.1} is
$$\gamma\min(\mL^n(\om_{\bar r,\ep}(y)),\sum_{r\neq\bar r}\mL^n(\om_{r,\ep}(y))),$$
where $\bar r=\bar r(\ep,y)$ is such that
$$\mL^n(\om_{\bar r,\ep}(y))=\max_{1\leq r\leq N}\mL^n(\om_{r,\ep}(y)).$$
To see this, fix $\ep$ and $y$ and let $a_r=\mL^n(\om_{r,\ep}(y))$. Suppose without
loss of generality that $a_N\geq a_{N-1}\geq\ldots\geq a_1$ and let $c=\sum_{r=1}^Na_r$. Then we
have to show that
$$\max_{1\leq r\leq N}\min(a_r,c-a_r)=\min(a_N,c-a_N).$$
But this follows from the fact that $a_r\leq c-a_N$ if $1\leq r<N$.
\end{rem}
We state the case $N=2$ of Theorem \ref{transest} separately.
\begin{thm}
\lbl{transest2}
Let $1<p<\infty$ and let $\om$ be $C$-connected. Two disjoint compact sets $K_1,K_2$ are incompatible  if and only if there exist
constants $\ep_0>0$ and $\gamma>0$ such that
if $0\leq \ep<\ep_0$ and $y\in W^{1,p}(\om;\R^m)$ then
\be
\lbl{4.2}
\int_{T_\ep(y)}(1+|Dy|^p)\,dx\geq\gamma\min(\mL^n(\om_{1,\ep}(y)),\mL^n(\om_{2,\ep}(y))).
\ee
The constant $\ep_0$ can be chosen to depend only on $E(C), K_1, K_2$ and $p$,
while the constant $\gamma$ can be chosen to depend only on these quantities and $\mL^n(C)/\mL^n(\om)$.
\end{thm}
Note that Theorem \ref{transest} follows from Theorem \ref{transest2} by applying it to the pair of sets
$K_r$ and $\bigcup_{s\neq r}K_s$ for each $r$, remarking that the set $T_\ep(y)$ is the same for each $r$,
and applying Remark \ref{rem1}a. It therefore suffices to prove Theorem \ref{transest2}. We use the following
lemma.

\begin{lem}
\lbl{translem}
Let $0\leq E<1$, and let $K_1,K_2$ be incompatible.  Then there exist constants $\ep_0=\ep_0(E,K_1,K_2,p)>0$ and $\gamma_0=\gamma_0(E,K_1,K_2,p)>0$
 such that
if $\tilde C\subset\R^n$ is any bounded open convex set with $E(\tilde C)\leq E$ and if $0\leq \ep<\ep_0$, $y\in W^{1,p}(\tilde C;\R^m)$, with
for some $i=1,2$
$$\frac{3}{4}\mL^n(\tilde C)\geq\mL^n(\{x\in\tilde C:Dy(x)\in N_\ep(K_i)\})\geq\frac{1}{4}\mL^n(\tilde C),$$
then
$$\int_{T_{\ep,\tilde C}(y)}(1+|Dy|^p)\,dx\geq\gamma_0\mL^n(\tilde C),$$
where $T_{\ep,\tilde C}(y):=\{ x\in \tilde C:Dy(x)\not\in N_\ep(K_1)\cup N_\ep(K_2)\}$.
\end{lem}

\begin{proof}
Suppose not. Then for $j=1,2,\ldots$ there exist $\ep^{(j)}\leq 1/j$, bounded open convex sets $C^{(j)}\subset\R^n$ with
$E(C^{(j)})\leq E$ and mappings
$y^{(j)}\in W^{1,p}(C^{(j)};\R^m)$ with for some $i=1,2$ (independent of $j$)
\be
\lbl{4.2aa}
\frac{3}{4}\mL^n( C^{(j)})\geq\mL^n(\{x\in C^{(j)}:Dy^{(j)}(x)\in N_{\ep^{(j)}}(K_i)\})\geq\frac{1}{4}
\mL^n(C^{(j)}),
\ee
and
$$\int_{T_{\ep^{(j)},C^{(j)}}(y^{(j)})}(1+|Dy^{(j)}|^p)\,dx\leq\frac{1}{j}\mL^n(C^{(j)}).$$
For definiteness we suppose \eqref{4.2aa} holds for $i=1$.  Let $B(a^{(j)},R_j)$ be the unique minimal ball containing $C^{(j)}$, so that $R_j=R(C^{(j)})$. We normalize $C^{(j)}$ by setting
\be
\lbl{4.2ab}
\tilde C^{(j)}=\frac{1}{R_j}(C^{(j)}-a^{(j)}).
\ee
Thus $R(\tilde C^{(j)})=1$ and $B(0,1)$ is the unique minimal ball containing $\tilde C^{(j)}$.
  Define
$z^{(j)}\in W^{1,p}(\tilde C^{(j)};\R^m)$ by
\be
\lbl{4.2ac}
z^{(j)}(x)=\frac{1}{R_j}y^{(j)}(a^{(j)}+R_jx).
\ee
Then
\be
\lbl{4.2ad}
Dz^{(j)}(x)=Dy^{(j)}(a^{(j)}+R_jx)
\ee
and we have that
\be
\lbl{4.2a}
\frac{3}{4}\mL^n(\tilde C^{(j)})\geq\mL^n(\{x\in \tilde C^{(j)}:Dz^{(j)}(x)\in N_{\ep^{(j)}}(K_1) \})
\geq\frac{1}{4}
\mL^n(\tilde C^{(j)})
\ee
and
\be
\lbl{4.3}
\int_{T_j}(1+|Dz^{(j)}|^p)\,dx\leq\frac{1}{j}\mL^n(\tilde C^{(j)}),
\ee
where $T_j:=\{x\in \tilde C^{(j)}:Dz^{(j)}(x)\not\in N_{\ep^{(j)}}(K_1) \cup N_{\ep^{(j)}}(K_2)\}.$
Since the closures $D^{(j)}$ of $\tilde C^{(j)}$ lie in $\overline{B(0,1)}$, a subsequence
(which we do not relabel) of the $D^{(j)}$ converges in the
Hausdorff metric to a closed convex set $D\subset \overline{B(0,1)}$. Since $E(\tilde C^{(j)})=E(C^{(j)})\leq E$,
there is a closed ball $B_j$ contained in $D^{(j)}$ with radius at least $\sqrt{1-E^2}$. We can suppose that
these balls also converge to a ball $B\subset D$ of radius at least $\sqrt{1-E^2}>0$, and hence
$D$ has nonempty interior $\tilde C$. Note that $\mL^n(\tilde C^{(j)})\arr\mL^n(\tilde C)$. Let $G$ be open and convex with $\bar G\subset\tilde C$ and $\mL^n(\tilde C\backslash G)<\frac{1}{8}\mL^n(\tilde C)$. Then for sufficiently large $j$ we have $G\subset\tilde C^{(j)}$. Hence, for sufficiently large $j$, 
\be \label{4.3aa}
\mL^n(\tilde C^{(j)})<\frac{8}{7}\mL^n(G).
\ee
 Also, by \eqref{4.2a},
\be
\lbl{4.3a}\nonumber
\mL^n(\{x\in \tilde C^{(j)}:Dz^{(j)}(x)\in N_{\ep^{(j)}}(K_1)\})\hspace{-2in}\\ \nonumber&\geq &\mL^n(\{x\in G:Dz^{(j)}(x)\in N_{\ep^{(j)}}(K_1)\})\\
\nonumber
&\geq &\mL^n(\{x\in \tilde C^{(j)}:Dz^{(j)}(x)\in N_{\ep^{(j)}}(K_1)\})-\mL^n(\tilde C^{(j)}\backslash G)\\ \nonumber
&\geq &\frac{1}{4}\mL^n(\tilde C^{(j)})-\mL^n(\tilde C^{(j)}\backslash G)\\
&\geq & \frac{1}{8}\mL^n(G).
\ee
Hence, combining \eqref{4.3aa}, \eqref{4.3a} and the left-hand inequality in \eqref{4.2a}, we have
\be
\lbl{4.3b}
\frac{6}{7}\mL^n(G)\geq\mL^n(\{x\in G:Dz^{(j)}(x)\in N_{\ep^{(j)}}(K_1)\})\geq\frac{1}{8}\mL^n(G).
\ee

Since $K_1,K_2$ are bounded, it follows in particular from \eqref{4.3} that $Dz^{(j)}$ is bounded in
$L^p(G;M^{m\times n})$, and so we may assume that $Dz^{(j)}$ generates a Young measure
$(\nu_x)_{x\in G}$.
Let $U_1, U_2$ be open neighbourhoods of $K_1, K_2$ respectively. Since
$\{ x\in G:Dz^{(j)}(x)\not\in U_1\cup U_2\}\subset T_j$ for sufficiently large $j$, and $\mL^n(T_j)\arr 0$, we have that
$Dz^{(j)}(x)\arr K_1\cup K_2$ in measure, and hence $\supp \nu_x\subset K_1\cup K_2$ for a.e.
$x\in G$. Since $K_1,K_2$ are incompatible we thus have either $\supp \nu_x\subset K_1$
 a.e. or $\supp\nu_x\subset K_2$ a.e. in $G$. Now let $\varphi_i:M^{m\times n}\arr[0,1]$,
$i=1,2$, be continuous functions such that $\varphi_i=1$ on $N_{\delta/2}(K_i),$
$\varphi_i=0$ outside $N_\delta(K_i),$ where $\delta>0$ is sufficiently small so that
$N_\delta(K_1)\cap N_\delta(K_2)$ is empty. Then from \eqref{4.3b} we have that
\be
\lbl{4.4}
\int_{G}\varphi_1(Dz^{(j)})\,dx\geq\frac{1}{8}\mL^n(G)
\ee
for all sufficiently large $j$. Since for sufficiently large $j$
\be
\nonumber
\mL^n(\{x\in G:Dz^{(j)}(x)\in N_{\ep^{(j)}}(K_2)\})\hspace{2.5in}\\ \hspace{.5in}
\geq\mL^n(G)-\mL^n(\{x\in G:Dz^{(j)}(x)\in N_{\ep^{(j)}}(K_1)\})-\mL^n(T_j),\nonumber
\ee
we have from \eqref{4.2a}, \eqref{4.3} that for sufficiently large $j$
\be
\nonumber
\mL^n(\{ x\in G:Dz^{(j)}(x)\in N_{\delta/2}(K_2)\})\geq \frac{1}{7}\mL^n(G)-\frac{1}{j}\mL^n(\tilde C^{(j)})
\ee
and thus
\be
\lbl{4.4a}
\int_{G}\varphi_2(Dz^{(j)}(x))\, dx\geq
 \frac{1}{8}\mL^n(G).
\ee
But
$$\lim_{j\arr\infty}\int_{G}\varphi_i(Dz^{(j)})\,dx=\int_{G}\langle \nu_x,\varphi_i\rangle\, dx$$
for $i=1,2$, and one of these integrals is zero, contradicting \eqref{4.4}, \eqref{4.4a}.
\end{proof}
 {\it Proof of Theorem \ref{transest2}}\\
{\it Sufficiency.} Let $Dy^{(j)}$ be bounded in $L^\infty(\om;M^{m\times n})$ and have Young measure
$(\nu_x)_{x\in\om}$ with $\supp\nu_x\subset K_1\cup K_2$ almost everywhere. Choose $\ep\in(0,\ep_0)$ sufficiently
small so that $N_\ep(K_1), N_\ep(K_2)$ are disjoint. Then since $Dy^{(j)}\arr K_1\cup K_2$ in measure
we have $\lim_{j\arr\infty}\mL^n(T_\ep(y^{(j)}))=0$ and hence by \eqref{4.2}
\be
\lbl{4.4a0}
\min(\mL^n(\om_{1,\ep}(y^{(j)})),\mL^n(\om_{2,\ep}(y^{(j)})))\arr 0.
\ee
Let $f:M^{m\times n}\arr [0,1]$ be continuous with $f=1$ on $K_1$, $f=0$ outside $N_\ep(K_1)$. Then
\be
\lbl{4.4a1}
\lim_{j\to\infty}\av_\om f(Dy^{(j)})\,dx=\av_\om \langle\nu_x,f\rangle\,dx=\av_\om\nu_x(K_1)\,dx.
\ee
From \eqref{4.4a0} there exists a subsequence $y^{(j_k)}$ of $y^{(j)}$ such that either\\
  $\mL^n(\om_{1,\ep}(y^{(j_k)}))\arr 0$ or $\mL^n(\om_{2,\ep}(y^{(j_k)}))\arr 0$, and so
from \eqref{4.4a1} we have that
$$\av_\om\nu_x(K_1)\,dx=0 \text{ or } 1,$$
implying either that $\supp\nu_x\subset K_1$ a.e. or that $\supp\nu_x\subset K_2$ a.e. as required.
\vspace{.1in}

\noindent{\it Necessity.}
Fix $\ep$ with $0\leq \ep<\ep_0$, where $\ep_0$ is given by Lemma \ref{translem} with $E$ being the eccentricity of $C$
(so that in particular $N_{\ep}(K_1)$ and $N_{\ep}(K_2)$ are disjoint), and let $y\in W^{1,p}(\om;\R^m)$.
First suppose that
$$\mL^n(T_\ep(y))\geq\frac{1}{4}\mL^n(C).$$
Then
$$\int_{T_\ep(y)}(1+|Dy|^p)\,dx\geq\frac{1}{4}\frac{\mL^n(C)}{\mL^n(\om)}\mL^n(\om),$$
so that \eqref{4.2} holds with $\gamma=\frac{1}{4}\frac{\mL^n(C)}{\mL^n(\om)}$.

We thus assume that
\be
\lbl{4.5}
\mL^n(T_\ep(y))<\frac{1}{4}\mL^n(C).
\ee

Since $\om$ is $C$-connected, there is an equivalence class $\mC$ of $\mK(C)$ with respect to $\sim$ that covers $\om$. Suppose that there exist two sets $C_1, C_2\in\mC$ (in particular, both directly congruent to $C$) such
that
\be
\lbl{4.6}
\mL^n(\{ x\in C_i:Dy(x)\in N_\ep(K_i)\})\geq \frac{1}{4}\mL^n(C)
\ee
for $i=1,2$. By the definition of $\sim$ there exist continuous functions $\xi:[0,1]\arr\om, Q:[0,1]\arr SO(n)$, such that $\xi(0)+Q(0)C=C_1, \xi(1)+Q(1)C=C_2$, and $C(t):=\xi(t)+Q(t)C\subset\om$ for all $t\in[0,1]$.
For $i=1,2$ define
$$\theta_i(t)=\frac{\mL^n(\{x\in C(t):Dy(x)\in N_\ep(K_i)\})}{\mL^n(C)}.$$
Then by \eqref{4.5} $\theta_i:[0,1]\arr [0,1]$ is continuous, $\theta_1(t)+\theta_2(t)\geq \frac{3}{4}$
 for all $ t\in [0,1]$, and by \eqref{4.6} $\theta_1(0)\geq\frac{1}{4}$,  $\theta_2(1)\geq \frac{1}{4}$.
Hence there exists $\tau\in [0,1]$ with $\theta_1(\tau)\geq\frac{1}{4}$,  $\theta_2(\tau)\geq \frac{1}{4}$. Indeed if $\theta_2(0)\geq \frac{1}{4}$ we can take $\tau=0$. Otherwise $\theta_2(0)<\frac{1}{4}$ and so there exists $\tau\in[0,1]$ with $\theta_2(\tau)=\frac{1}{4}$ and then $\theta_1(\tau)=\theta_1(\tau)+\theta_2(\tau)-\frac{1}{4}\geq\frac{1}{2}$.
By Lemma \ref{translem} applied to $\tilde C=C(\tau)$ we deduce that
$$\int_{T_\ep(y)}(1+|Dy|^p)\,dx\geq\gamma_0 \frac{\mL^n(C)}{\mL^n(\om)}\mL^n(\om)$$
so that \eqref{4.2} holds with $\gamma=\gamma_0\frac{\mL^n(C)}{\mL^n(\om)}$.

It therefore remains to consider the case when for some $i=1,2$
\be
\lbl{4.7}
\mL^n(\{ x\in D: Dy(x)\in N_\ep (K_i)\})<\frac{1}{4}\mL^n(C)
\ee
for {\it every}  $D\in\mC$.

Let $\tilde x$ be any Lebesgue point of $\om_{i,\ep}=\om_{i,\ep}(y)$. Since $\mC$ covers $\om$ there exist $\xi(\tilde x)\in\R^n$, $\tilde Q(\tilde x)\in SO(n)$ such that $\tilde C(\tilde x):=\xi(\tilde x)+\tilde Q(\tilde x)C$ belongs to $\mC$ and $\tilde x\in \tilde C(\tilde x)$. For $0<r\leq 1$ let $\tilde C_r(\tilde x)=r\tilde C(\tilde x)+(1-r)\tilde x$. Note that $\tilde x\in \tilde C_r(\tilde x)\subset \tilde C(\tilde x)$. Define
$$f(\tilde x,r)=\frac{\mL^n(\tilde C_r(\tilde x)\cap\om_{i,\ep})}{\mL^n(\tilde C_r(\tilde x))}.$$
Then $f(\tilde x,r)$ is continuous in $r\in(0,1]$, and since $\tilde x$ is a Lebesgue point of
$\om_{i,\ep}$ we have
$$\lim_{r\arr 0}f(\tilde x,r)=1.$$
But by \eqref{4.7} applied to $\tilde C(\tilde x)$, we have $f(\tilde x,1)<\frac{1}{4},$ and
so there exists $r(\tilde x)\in (0,1]$ such that
$$\mL^n(\{ x\in \tilde C_{r(\tilde x)}(\tilde x): Dy(x)\in N_\ep(K_i)\})=\half\mL^n(\tilde C_{r(\tilde x)}(\tilde x)).$$
Let $B(a(C),R(C))$ be the minimal ball containing $C$.
Then the balls $$B_{\tilde x}=B(r(\tilde x)[\tilde Q(\tilde x)a(C)+\xi(\tilde x)]+(1-r(\tilde x))\tilde x,r(\tilde x)R(C))$$ are such that $C_{r(\tilde x)}(\tilde x)\subset B_{\tilde x}$ and in particular they cover the set of Lebesgue points of $\om_{i,\ep}$. It follows from Lemma \ref{stein} that there exists a finite or countable
disjoint subfamily $\{B_j\}$, where $B_j=B_{\tilde x_j}$, such that
$$\sum_j\mL^n(B_j)\geq c_n\mL^n(\om_{i,\ep}).$$
Hence by Lemma \ref{translem}, writing $\tilde C_j=\tilde C_{r(\tilde x_j)}(\tilde x_j)$,
\be
\int_{T_\ep(y)}(1+|Dy|^p)\,dx&\geq&\sum_j\int_{T_\ep(y)\cap\tilde C_j}(1+|Dy|^p)\,dx\nonumber\\
&\geq&\gamma_0\sum_j\mL^n(\tilde C_j)\nonumber\\
&=&\gamma_0 \frac{\mL^n(C)}{\mL^n(B(0,R(C)))}\sum_j\mL^n(B_j)\nonumber\\
&\geq&\gamma_0c_n\frac{\mL^n(C)}{\mL^n(B(0,R(C)))} \mL^n(\om_{i,\ep})\nonumber\\&\geq&
\gamma_0c_n(1-E^2)^\frac{n}{2}\mL^n(\om_{i,\ep}).
\ee
Combining this with the previous cases we deduce that  \eqref{4.4} holds with
\begin{equation}\lbl{gammavalue}
\gamma=\min
[\gamma_1\frac{\mL^n(C)}{\mL^n(\om)},\gamma_0c_n(1-E^2)^\frac{n}{2}],
\end{equation}
 where $\gamma_1=\min(\gamma_0,\frac{1}{4})$.
\\

The transition layer estimate can be given an equivalent formulation in terms of gradient Young measures.

\begin{thm}
\lbl{transym}
Let $1<p<\infty$ and let $\om$ be $C$-connected. Then
$K_1,\ldots,K_N$ are incompatible if and only if  there exist constants $\ep_0>0$
and $\gamma>0$
 such that if $0\leq\ep<\ep_0$ and
$(\nu_x)_{x\in\om}$ is an $L^{p}$ gradient Young measure then
\be
\nonumber
 \int_\om\int_{\left[ \bigcup_{r=1}^NN_\ep(K_r)\right]^c} (1+|A|^p)\,d\nu_x(A)\,dx&\geq&\\
\lbl{4.8}& &\hspace{-2in} \gamma\max_{1\leq r\leq N}\min  \left(\int_\om\nu_x(N_\ep(K_r))\,dx,\int_\om\nu_x(\bigcup_{s\neq r}N_\ep(K_s))\,dx\right).
\ee
The constant $\ep_0$ can be chosen to depend only on $E(C), K_1,\ldots, K_N$and $p$,
while the constant $\gamma$ can be chosen to depend only on these quantities and $\mL^n(C)/\mL^n(\om)$.
\end{thm}

Note that Theorem \ref{transest} corresponds to the special case when $\nu_x=\delta_{Dy(x)}$. Again we need only prove the case $N=2$ of Theorem \ref{transym}, namely

\begin{thm}
\lbl{transym1}
Let $1<p<\infty$ and let $\om$ be $C$-connected. A pair of disjoint compact sets
$K_1,K_2$ are incompatible if and only if  there exist constants $\ep_0>0$
and $\gamma>0$
 such that if $0\leq\ep<\ep_0$ and
$(\nu_x)_{x\in\om}$ is an $L^{p}$ gradient Young measure then
\be
\nonumber
\int_\om\int_{\left[ N_\ep(K_1)\cup N_\ep(K_2)\right]^c} (1+|A|^p)\,d\nu_x(A)\,dx&\geq&\\
\lbl{4.9} & &\hspace{-1.5in}\gamma \min\left(\int_\om\nu_x(N_\ep(K_1))\,dx,\int_\om\nu_x\left(N_\ep(K_2)\right)\,dx\right).
\ee
The constant $\ep_0$ can be chosen to depend only on $E(C), K_1, K_2$ and $p$,
while the constant $\gamma$ can be chosen to depend only on these quantities and $\mL^n(C)/\mL^n(\om)$.
\end{thm}

\noindent{\it Proof of Theorem \ref{transym1}} \hspace{.1in} 
Since Theorem \ref{transest2} is a special case of Theorem \ref{transym1} we need only show that
if $K_1, K_2$ are incompatible then \eqref{4.9} holds. Let $\ep_0, \gamma$ be as in Theorem \ref{transest2}, and let
$0\leq\ep<\ep'<\ep_0$. Let $(\nu_x)_{x\in\om}$ be an $L^{p}$ gradient Young measure.
By Theorem \ref{kp}, we may suppose that $(\nu_x)_{x\in\om}$ is generated by a sequence $Dy^{(j)}$ of gradients which is such that $|Dy^{(j)}|^p$ is weakly convergent in $L^1(\om)$. Also
\be
\lbl{4.9a}
\int_\om\int_{M^{m\times n}}|A|^pd\nu_x(A)\,dx<\infty.
\ee
 For $k=1,2, \ldots$ let $\varphi_k:M^{m\times n}\arr [0,1]$ be continuous and satisfy
\be
\lbl{4.10}
\varphi_k(A)=\left\{\begin{array}{ll}1& \mbox{if } A\in [N_{\ep'}(K_1)\cup N_{\ep'}(K_2)]^c,\\
0& \mbox{if } A\in N_{\ep'-\frac{1}{k}}(K_1)\cup N_{\ep'-\frac{1}{k}}(K_2),
\end{array}\right.
\ee
with $\varphi_k$ nonincreasing in $k$. The existence of $\tilde\varphi_k$ satisfying all but the last condition follows from Urysohn's lemma, and then we may set $\varphi_k=\min_{j\leq k}\tilde\varphi_j$. Clearly $\varphi_k\arr \chi_{\ep'}$ pointwise, where $\chi_{\ep'}$ is the characteristic function of the closure of $[N_{\ep'}(K_1)\cup N_{\ep'}(K_2)]^c$. Similarly, for $l=1,2$ let $\varphi^l_k:M^{m\times n}\arr [0,1]$ be continuous and satisfy
\be
\lbl{4.11}
\varphi_k^l(A)=\left\{\begin{array}{ll}0& \mbox{if } A\in N_{\ep'}(K_l)^c,\\
1& \mbox{if } A\in N_{\ep'-\frac{1}{k}}(K_l),\end{array}\right.
\ee
with $\varphi^l_k$ nondecreasing in $k$. Clearly $\varphi^l_k\arr \chi({\rm int}\,N_{\ep'}(K_l))$ pointwise.

For each $j,k$ we have by Theorem \ref{transest2} that
\be
\nonumber
\int_\om\varphi_k(Dy^{(j)})(1+|Dy^{(j)}|^p)\,dx&\geq& \int_{T_{\ep'}(y^{(j)})}(1+|Dy^{(j)}|^p)\,dx\\
\nonumber&\geq&\gamma\min\left(\mL^n(\om_{1,\ep'}(y^{(j)})), \mL^n(\om_{2,\ep'}(y^{(j)}))\right)\\
\nonumber&\geq&\gamma\min\left(\int_\om \varphi^1_k(Dy^{(j)})\,dx, \int_\om \varphi^2_k(Dy^{(j)})\,dx\right).
\ee
Since $|Dy^{(j)}|^p$ is weakly convergent in $L^1(\om)$, it is equi-integrable, and hence so is
$\varphi_k(Dy^{(j)})(1+|Dy^{(j)}|^p)$, which thus has an $L^1$ weakly convergent subsequence. Letting $j\arr\infty$ in this subsequence we deduce from the fundamental properties of Young measures that
\be
\lbl{4.12}
\int_\om\langle\nu_x,\varphi_k(A)(1+|A|^p)\rangle\,dx\geq\gamma\min\left(\int_\om \langle\nu_x,\varphi^1_k\rangle\,dx,  \int_\om \langle\nu_x,\varphi^2_k\rangle\,dx\right).
\ee
Passing  to the limit $k\arr\infty$, using the everywhere monotone convergence of $\varphi_k, \varphi_k^1,\varphi_k^2$, we  obtain
\be
\nonumber
\int_\om\int_{M^{m\times n}}\chi_{\varepsilon'}(A)(1+|A|^p)\,d\nu_x(A)\,dx&&\\
\nonumber\mbox{ }& &\hspace{-1in}\geq\gamma\min\left(\int_\om \nu_x({\rm int}\,N_{\varepsilon'}(K_1))\,dx,\int_\om \nu_x({\rm int}\,N_{\varepsilon'}(K_2))\,dx\right)\\ \nonumber
&&\hspace{-1in}\geq\gamma \min\left(\int_\om\nu_x(N_\varepsilon(K_1))\,dx,\int_\om\nu_x(N_\varepsilon(K_2))\,dx\right).
\ee
Letting $\varepsilon'\arr\varepsilon +$, and noting that $\chi_{\varepsilon'}\arr\chi([N_\varepsilon(K_1)\cup N_\varepsilon(K_2)]^c)$ monotonically, we deduce by \eqref{4.9a} and monotone convergence that
 \eqref{4.9} holds.
\\

\begin{cor}\lbl{epnbhds}
Let $K_1,\ldots K_N$ be incompatible. Then there exists $\varepsilon_0>0$ such that $N_\varepsilon(K_1),\ldots,N_\varepsilon(K_N)$ are incompatible for $0\leq \varepsilon<\varepsilon_0$.
\end{cor}
\begin{proof}
By Remark \ref{rem1}b we may assume that $\om$ is $C$-connected, while by Remark \ref{rem1}a we need only show that $N_\varepsilon(K_r)$ and $\bigcup_{s\neq r}N_\varepsilon(K_s)$ are incompatible. Let $\supp\nu_x\subset\bigcup_{r=1}^NN_\varepsilon(K_r)$ almost everywhere. Then the left-hand side of \eqref{4.8} is zero. Hence for each $r$ either $\nu_x(N_\varepsilon(K_r))=0$ a.e. or $\nu_x(\bigcup_{s\neq r}N_\varepsilon(K_s))=0$ a.e., and hence either $\supp\nu_x\subset \bigcup_{s\neq r}N_\varepsilon(K_s)$ a.e. or
$\supp\nu_x\subset N_\varepsilon(K_r)$ a.e., giving the result.
\end{proof}
Applying the above Corollary \ref{epnbhds} to the case when each $K_r$ consists of a single matrix we immediately obtain
\begin{cor}\lbl{open}
For any $N$ the set of points $(A_1,\ldots,A_N)\in (M^{m\times n})^N$ with $\{A_1\},\ldots, \{A_N\}$ incompatible is open.
\end{cor}
When $N=2$ this already gives interesting information. Indeed it implies a special case of \v{S}ver\'{a}k's three matrix theorem \cite{Sverak93a}.  In fact if $A_1, A_2\in M^{m\times n}$, with $\rank(A_1-A_2)>1$, we have that $\{A_1\}, \{A_2\}$ are incompatible, and so if $A_3$ is taken sufficiently close to $A_2$ with $\rank(A_2-A_3)>1$ we have that the sets $K_1=\{A_1\}$ and $K_2=\{A_2,A_3\}$ are incompatible. Thus if
$(\nu_x)_{x\in\om}$ is a gradient Young measure with $\supp \nu_x\subset \{A_1,A_2,A_3\}$ a.e. then either $\nu_x=\delta_{A_1}$ a.e. or $\supp \nu_x\subset \{A_2,A_3\}$ almost everywhere. In the latter case, since $\{A_2\}, \{A_3\}$ are incompatible, we have that either $\nu_x=\delta_{A_2}$ a.e. or $\nu_x=\delta_{A_3}$ almost everywhere. Hence $\nu_x=\delta_{A_i}$ a.e. for some $i$, which is the statement of \v{S}ver\'{a}k's theorem in this special case. As remarked to us by V. \v{S}ver\'{a}k, this special case cannot be proved
using the minors relations alone. For example, taking $m=n=2$, the probability measure
$$\nu=\frac{\varepsilon^2}{4-\ep^2}\delta_0+\frac{2-\ep^2}{4-\ep^2}(\delta_\1+\delta_{A_\ep}),$$
where $A_\ep=\left(\begin{array}{cc}1-\ep&0\\0&1+\ep\end{array}\right)$ and $\ep>0$ is sufficiently small, satisfies the minors relation $\det\bar\nu= \langle\nu,\det\rangle$, but by the above $\{0\},\{\1\},\{A_\ep\}$ are incompatible. By Theorem \ref{homcom}, Corollary \ref{epnbhds} thus implies the existence of quasiconvex functions that are not polyconvex. In 
\cite{u6} we give a new proof of the three matrix theorem in the general case, using similar techniques as in the proof of Theorem \ref{transest}
plus ingredients from the theory of quasiregular maps. Another proof using results from the theory of two dimensional quasiregular maps is due to Astala \& Faraco \cite{astalafaraco}. 

The following simple example shows that Theorems \ref{transest}, \ref{transest2}, \ref{transym}, \ref{transym1} are not true if $1+|A|^p$ is replaced by $|A|^p$ in the integrals over the transition layer, even when the volume of the transition layer is arbitrarily small.
\begin{example}\lbl{noone}
Let $m=n=2$, $\om=(0,1)^2$ and let $A_1=e_2\otimes e_2$, $A_2=(e_1+e_2)\otimes (e_1+e_2)$. Then $K_1=\{A_1\},K_2=\{A_2\}$ are incompatible, but $0$ is compatible with both $A_1$ and $A_2$. Define for small $\delta>0$ and for $x\in\om$,
\begin{eqnarray*}
y_\delta(x)=\left\{\begin{array}{ll}x_2e_2&\text{ if }0<x_2<1-\delta,\\
(1-\delta)e_2&\text{ if }x_2\geq 1-\delta,\; x_1+x_2\leq 2-\delta,\\
(e_1+e_2)(x_1+x_2)+(\delta-2)e_1-e_2&\text{ if }x_2\geq 1-\delta,\; x_1+x_2>2-\delta.\end{array}
\right.
\end{eqnarray*}
Then
\begin{eqnarray*}
Dy_\delta(x)=\left\{\begin{array}{ll}A_1&\text{ if }0<x_2<1-\delta,\\
0&\text{ if }x_2\geq 1-\delta,\; x_1+x_2\leq 2-\delta,\\
A_2&\text{ if }x_2\geq 1-\delta,\; x_1+x_2>2-\delta,\end{array}
\right.
\end{eqnarray*}
and we have for any $p>1$
\begin{eqnarray*}
\int_{T_0(y_\delta)}|Dy_\delta|^p\,dx=0,\;\;\;\; \min \{\mL^2(\Omega_{1,0}(y_\delta)),\mL^2( \Omega_{2,0}(y_\delta))\}=\half\delta^2.
\end{eqnarray*}
\end{example}
\vspace{.2in}

We now show that Theorems \ref{transest}, \ref{transest2}, \ref{transym}, \ref{transym1} do not hold for general bounded domains $\om$. Since by Proposition \ref{cone} the hypothesis in these theorems that $\om$ be $C$-connected is equivalent to the cone condition, for a counterexample we need a domain not satisfying the cone condition.

\begin{example}\lbl{rooms} We take $\om$ to be the `rooms and passages' domain of Fraenkel \cite{fraenkel79}. For simplicity we let $m=n=2$. This domain $\om$ consists of the union of a sequence of square rooms $Q_j=(a_j,0)+ h_j(-1,1)^2$, $j=1,2,\ldots$, of decreasing side $2h_j>0$, centred at the points $(a_j,0)\in\R^2$ on the $x_1$-axis, where $a_1=0, a_j>0$, together with the rectangular connecting corridors $C_j=[a_j+ h_j, a_{j+1}-h_{j+1}]\times (-d_j, d_j)$ of length $l_j=a_{j+1}-h_{j+1}-(a_j+
h_j)>0$ and thickness $2d_j$, where $0<d_j<h_{j+1}$. In order for $\om$ to be bounded, we require that $\sum _{j=1}^\infty (2h_j+l_j)<\infty$.

Let $A_1,A_2\in M^{2\times 2}$ with $\rank(A_1-A_2)>1$, for example $A_1=0, A_2=\1$. Thus by Example \ref{twomatrices} the sets $K_1=\{A_1\}, K_2=\{A_2\}$ are incompatible. We define $y^{(j)}:\om\arr\R^2$ by
\begin{eqnarray*}
y^{(j)}(x)=\left\{\begin{array}{ll}A_1x &\text{ if } x\in \om_j,\\
\frac{x_1-a_{j-1}-h_{j-1}}{l_{j-1}}A_2x+\left(1-\frac{x_1-a_{j-1}-h_{j-1}}{l_{j-1}}\right)A_1x &\text{ if }x\in C_{j-1},\\A_2x &\text{ if }x\in Q_j,\\
\frac{x_1-a_{j}-h_{j}}{l_{j}}A_1x+\left(1-\frac{x_1-a_{j}-h_{j}}{l_{j}}\right)A_2x &\text{ if }x\in C_{j},
\end{array}\right.
\end{eqnarray*}
where $\om_j=\om\backslash (C_{j-1}\cup Q_j\cup C_j)$.
Thus in the corridor $C_{j-1}$
$$|Dy^{(j)}(x)|\leq c_0+\frac{c_1}{l_{j-1}},$$
while in the corridor $C_j$
$$|Dy^{(j)}(x)|\leq c_0+\frac{c_1}{l_{j}},$$
where $c_0, c_1$ are constants independent of $j$. Thus taking $\ep=0$, we have
\begin{eqnarray*}
\int_{T_0(y^{(j)})}(1+|Dy^{(j)}|^p)\,dx&=&\int_{C_{j-1}\cup C_j}(1+|Dy^{(j)}|^p)\,dx\\
&\leq &2l_{j-1}d_{j-1}\left[ 1+\left(c_0+\frac{c_1}{l_{j-1}}\right)^p\right] \\
&+& 2l_{j}d_{j}\left[ 1+\left(c_0+\frac{c_1}{l_{j}}\right)^p\right],
\end{eqnarray*}
while \begin{eqnarray*}
\min(\mL^2(\om_{1,0}(y^{(j)})),\mL^2(\om_{2,0}(y^{(j)}))= \mL^2(Q_j)=4h_j^2.
\end{eqnarray*}
Thus, fixing the sequences $h_j$ and $l_j$ and letting $d_j\arr 0$ sufficiently rapidly as $j\arr \infty$, we violate the conclusion
\eqref{4.2} of Theorem \ref{transest2} for any choice of $\gamma$.
\end{example}
For applications it is important to be able to estimate the constants $\eps_0$ and $\gamma$ in Theorems \ref{transest}, \ref{transest2}, \ref{transym}, \ref{transym1} and Corollary \ref{epnbhds}. 
The proof of Theorem \ref{transest2} gives   a lower bound on $\gamma$ (see \eqref{gammavalue})  in terms of the constant $\gamma_0$ occurring in Lemma \ref{translem}. This lemma is proved by contradiction, and thus gives no estimate on $\eps_0$ or $\gamma_0$. However, Zhang \cite{zhang2003,zhang2003a,zhang2003b} has obtained estimates for the constant $\eps_0$ in Corollary \ref{epnbhds} using Schauder estimates in BMO and Campanato spaces for linear elliptic systems in the two cases (a) $m$ and $n$ arbitrary, $K_r=\{A_r\}, r=1,\ldots,N$ where the linear span of the distinct matrices $A_1,\ldots,A_N$ has no rank-one connections, and (b) $m=n=2$ and $K_r=\lambda_rSO(2), r=1,\ldots,N$ with $0<\lambda_1<\cdots<\lambda_N$. 

As regards  $\gamma$ we can obtain upper bounds by considering explicit test
functions. We illustrate this in the next example for the case when $m=n$, $p=2$, $\om$ is a ball and $K_1=\{\lambda\1\}, K_2=\{\mu\1\}$ with $\lambda\neq\mu$.
\begin{example}
\lbl{radial}
Let $m=n>1$, $\om=B(0,1)$, $A_1=\lambda\1$, $A_2=\mu\1$, $\lambda\neq\mu$. We consider for $k>1$ and sufficiently small $\ep>0$  the radial mapping
\begin{eqnarray}
y_\ep(x)=\frac{r_\ep(R)}{R}x,
\end{eqnarray}
where $R=|x|$ and
\be
r_\ep(R)=\left\{\begin{array}{ll}\lambda R&\text{ if } 0\leq R\leq\ep,\\
\mu R&\text{ if }k\ep\leq R<1.
\end{array}\right.
\ee
For $\ep<R<k\ep$ we choose $r_\ep$ so that it is continuous and  minimizes
\be
\lbl{radial0}
\int_{\{\ep<|x|<k\ep\}}(1+|Dy|^2)\,dx.
\ee
 Noting that
\be
\lbl{radial1}
|Dy|^2=(n-1){\left(\frac{r}{R}\right)}^2+(r')^2,
\ee
the Euler-Lagrange equation for the functional
\be
\lbl{radial2}
\int_\ep^{k\ep}R^{n-1}\left(1+(n-1)\left(\frac{r}{R}\right)^2+(r')^2\right)dR
\ee
has linearly independent solutions $r=R$ and $r=R^{1-n}$. Choosing constants $A,B$ so that $r(R)=AR+BR^{1-n}$ satisfies $r(\ep)=\lambda\ep, r(k\ep)=\mu k\ep$, we find that for the optimal transition layer
\be
\lbl{radial3}
r_\ep(R)= \left(\frac{k^n\mu-\lambda}{k^n-1}\right)R+\frac{(\lambda-\mu)(\ep k)^n}{k^n-1}R^{1-n}, \text{ if }\ep<R<k\ep.
\ee
(In fact by uniqueness of solutions to Laplace's equation this radial solution is the minimizer of \eqref{radial0} among all (not necessarily radial) maps matching the boundary conditions at $R=\ep, k\ep$.) Denoting by $T=\{\ep<|x|<k\ep\}$ the transition layer, we calculate using \eqref{radial1} that
the ratio
$$\rho=\frac{\int_T(1+|Dy|^2)\,dx}{\mL^n(\{x:Dy(x)=\lambda x\})}$$
is given by
\be
\nonumber
\rho&=&\frac{1}{\ep^{n-1}\omega_n}\int_\ep^{k\ep}\hspace{-.1in}R^{n-1}\left[1+n\left(\frac{k^n\mu-\lambda}{k^n-1}\right)^2+
n(n-1)\left(\frac{\lambda-\mu}{k^n-1}\right)^2\left(\frac{\ep k}{R}\right)^{2n}\right]dR\\&=&\int_1^ks^{n-1}\left[1+n\left(\frac{k^n\mu-\lambda}{k^n-1}\right)^2+
n(n-1)\left(\frac{\lambda-\mu}{k^n-1}\right)^2\left(\frac{k}{s}\right)^{2n}\right]\,ds\nonumber\\
&=&\frac{k^n-1}{n}+\frac{(k^n\mu-\lambda)^2}{k^n-1}+
(n-1)\frac{(\lambda-\mu)^2k^n}{k^n-1}.\nonumber
\ee
Here $\omega_n$ denotes the $(n-1)$-dimensional measure of the unit sphere in $\R^n$. To find the optimal width of the transition layer, we minimize $\rho$ over $k>1$. Setting $\tau=k^n$ and minimizing over $\tau>1$ we find that the minimum value $\rho_{\text{min}}$ is achieved when $\tau =1+\frac{n}{\sqrt{1+n\mu^2}}|\lambda-\mu|$, and that
$$\rho_{\text{min}}=(n-1)(\lambda-\mu)^2+
2(\sqrt{1+n\mu^2}-\sign(\lambda-\mu))|\lambda-\mu|.$$
Interchanging $\lambda$ and $\mu$ we deduce finally that the constant $\gamma$ satisfies
\begin{equation}
\label{radial5}
\gamma\leq (n-1)(\lambda-\mu)^2+
2h(\lambda,\mu)|\lambda-\mu|,
\end{equation}
where $h(\lambda,\mu)=\min(\sqrt{1+n\mu^2}-\sign(\lambda-\mu),\sqrt{1+n\lambda^2}-\sign(\mu-\lambda))$. Of course this upper bound tends to zero as $\lambda\arr\mu$. Note that the upper bound is proportional to  $|\lambda-\mu|$ when both $\lambda$ and $\mu$ are near one.
\end{example}

\section{Local Minimizers and Metastability}
\lbl{meta}
In this section we apply the transition layer estimate to prove that certain maps or microstructures (in the {\it parent phase}) are local minimizers of the corresponding energy, the mechanism being that the values of the gradient that could potentially lower the energy (those of the {\it product phase}) are incompatible with those of the parent phase, so that the gain in energy due to the resulting transition layer is greater than the loss of energy in using the gradients of the product phase. In
applications to materials undergoing solid phase transformations this provides a mechanism for {\it incompatibility induced hysteresis}.

The basic integral we consider is
\be
\lbl{meta1}
I(y)=\int_\om W(Dy(x))\,dx,
\ee
where $\om\subset\R^n$ is a bounded domain that is $C$-connected. We assume that\vspace{.1in}

\noindent (H1)  $W:M^{m\times n}\arr \R\cup\{+\infty\}$ is lower semicontinuous,

\noindent (H2) There exist constants $c_0\in\R, c_1>0, p>1$ such that
\be
\lbl{meta2}
W(A)\geq c_0+c_1|A|^p\mbox{ for all }A\in M^{m\times n}.
\ee
\vspace{.02in}

More generally we will consider the extension (relaxation) of \eqref{meta1} to gradient Young measures
\be
\lbl{meta3}
I(\nu)=\int_\om \int_{M^{m\times n}}W(A)\,d\nu_x(A)\,dx,
\ee
where $\nu=(\nu_x)_{x\in\om}$ is the Young measure corresponding to a sequence $Dy^{(j)}$ that is bounded in $L^p(\om;M^{m\times n})$. The functional \eqref{meta1} corresponds to the special case when $\nu_x=\delta_{Dy(x)}$ for some $y\in W^{1,p}(\om;\R^m)$.

We suppose that the parent and product phases are represented by the compact sets $K_1, K_2\subset M^{m\times n}$ respectively, where $K_1, K_2$ are incompatible. Let $\ep_0=\ep_0(E(C),K_1,K_2,p)$ be as in Theorem \ref{transest2}, and fix $\ep$ with $0<\ep<\ep_0$. We assume that\vspace{.1in}

\noindent (H3) $\min_{A\in N_{\ep/2}(K_1)}W(A)=0,\; W(A)\geq 0 \mbox{ for all }A\in N_\ep(K_1)$,

\noindent (H4) $W(A)\geq -\delta \mbox{ for all }A\in N_\ep (K_2) \mbox{ and some }\delta>0$,

\noindent (H5) $W(A)\geq \alpha \mbox{ for all }A\in [N_\ep(K_1)\cup N_\ep(K_2)]^c\mbox{ and some } \alpha>0 $.\vspace{.1in}

Thus $W$ has a local minimizer near the well $K_1$, with minimum value zero, and a possibly lower local minimizer near the well $K_2$. We will assume later that $\delta>0$ is sufficiently small, while $\alpha>0$ remains fixed. The hypotheses (H1)-(H5) are satisfied if $W$ is a small perturbation of some $W_0$ which has local minimizers with equal minimum value zero at the wells $K_1, K_2$, as we now show.
\begin{prop}
\lbl{Wep}
Assume that

$(\rm{H}1)'$ $W_\tau:M^{m\times n}\arr \R\cup\{+\infty\}$ is lower semicontinuous in $(\tau,A)\in[0,1]\times M^{m\times n}$, with $W_\tau(A)$ continuous in $\tau$ for all $A\in M^{m\times n}$,

$(\rm{H}2)'$ $W_0(A)\geq 0\mbox{ for all }A\in M^{m\times n}, \mbox{ and }W_0^{-1}(0)=K_1\cup K_2$,

$(\rm{H}3)'$ $\min_{A\in N_\ep(K_1)}W_\tau(A)=0$ for all $\tau\in[0,1]$,

$(\rm{H}4)'$ $W_\tau(A)\geq c_0+c_1|A|^p\mbox{ for all } \tau\in[0,1], A\in M^{m\times n}$.
\vspace{.1in}

Then, for sufficiently small $\tau>0$, $W_\tau$ satisfies $(\rm{H}1)-(\rm{H}5)$ for some fixed  $\alpha>0$ and $\delta=\delta(\tau)$ satisfying
\be
\lbl{meta4}
\lim_{\tau\arr 0+}\delta(\tau)=0.
\ee
\end{prop}
\begin{proof}
Clearly $W_\tau$ satisfies (H1), (H2). To prove (H3) note that by (H3)$'$ there exists $A_\tau\in N_\ep(K_1)$ with $W_\tau(A_\tau)=0$. We claim that $A_\tau\in N_{\ep/2}(K_1)$ for $\tau$ sufficiently small. If not, there would exist $\tau_j\arr 0$ with $\dist(A_{\tau_j},K_1)>\ep/2$ for all $j$, and we can suppose that $A_{\tau_j}\arr A\not\in N_{\ep/4}(K_1)$. But then by (H1)$'$
\be
\lbl{meta5}
0=\liminf_{j\arr\infty}W_{\tau_j}(A_{\tau_j})\geq W_0(A),
\ee
and so by (H2)$'$ $A\in K_1$, a contradiction.

To prove (H4) note that by (H1$'$), (H4$'$), $W_\tau$ attains a minimum on $N_\varepsilon(K_2)$ at some $B_\tau$, so that $W_\tau(A)\geq -\delta(\tau)$ for $A\in N_\varepsilon(K_2)$, where $$\delta(\tau)=\max\{-W_\tau(B_\tau),\tau\}>0.$$ Letting $\tau\arr 0+$ we have by (H1$'$) that $0\leq W_0(B)\leq \liminf_{\tau\arr 0+}W_\tau(B_\tau)$ for some $B\in N_\varepsilon(K_2)$ and so $\lim_{\tau\arr 0+}\delta(\tau)=0$.

To prove (H5) note that by (H1$'$), (H4$'$), $W_\tau$ attains a minimum on the closure of ${[N_\varepsilon(K_1)\cup N_\varepsilon(K_2)]^c}$ at some $C_\tau$, where $C_\tau$ is bounded for sufficiently small $\tau$. If (H5) were false then there would exist a sequence $\tau_j\arr 0+$ with $W_{\tau_j}(C_{\tau_j})\leq 1/j$ and we may assume that $C_j\arr C\not\in K_1\cup K_2$. But then (H2$'$) and (H1$'$) imply that $0< W_0(C)\leq \liminf_{j\arr \infty}W_{\tau_j}(C_j)\leq 0,$ a contradiction.
\end{proof}

\begin{thm}
\lbl{metstab}
Let  $\om$ be $C$-connected, and let $W$ satisfy ${\rm (H1)}-{\rm (H5)}$ with $\delta$ sufficiently small, so that $0<\delta<\delta_0$, where $\delta_0$ is a constant depending only on $K_1,K_2,p, E(C), \mL^n(C)/\mL^n(\om),\ep, c_0, c_1$ and $\alpha$. Let $\nu^*=(\nu_{x}^*)_{x\in\om}$ be an $L^{p}$ gradient Young measure with $\supp \nu^*_{x}\subset\{A\in N_\ep(K_1): W(A)=0\}$ and $\bar\nu^*_{x}=Dy^*(x)$, where $y^*\in W^{1,p}(\om;\R^m)$. Then there exists $\sigma>0$, depending on the above quantities and $\mL^n(\om)$, such that
\be
\lbl{meta6}
I(\nu)\geq I(\nu^*)
\ee
for any $L^{p}$ gradient Young measure $\nu=(\nu_x)_{x\in\om}$ with $\bar\nu_x=Dy(x)$ and
\be
\lbl{meta7}
\|y-y^*\|_{L^1(\om;\R^m)}< \sigma.
\ee
The inequality in \eqref{meta6} is strict unless $\supp\nu_x\subset\{A\in N_\ep(K_1):W(A)=0\}$ for a.e. $x\in\om$.
\end{thm}
We will use the following lemmas.
\begin{lem}
\lbl{meta7a} Let $\om$ be $C$-connected.
There exist $\Delta>0$ depending only on $K_1, K_2, p, E, \ep$ and $\mL^n(C)/\mL^n(\om)$, and $\beta>0$ depending only on the eccentricity $E(C)$ and $\mL^n(C)/\mL^n(\om)$, such that if  $\nu=(\nu_x)_{x\in\om}$ is an $L^{p}$ gradient Young measure with $\bar\nu_x=Dy(x)$ for $y\in W^{1,p}(\Omega; \R^m)$ and
\be
\lbl{meta8}\hspace{-.25in}
\int_\Omega\int_{[N_\ep(K_1)\cup N_\ep(K_2)]^c}(1+|A|^p)\,d\nu_x(A)\,dx+\int_\Omega \nu_x(N_\ep(K_1))\,dx<\Delta\mL^n(\om)
\ee
then
\be
\lbl{meta9}
\|y-z\|_{L^1(\Omega;\R^m)}>\beta\Delta\mL^n(\om)^\frac{n+1}{n}
\ee
for all $z\in W^{1,p}(\Omega; \R^m)$ with $Dz(x)\in N_\ep(K_1)^{\rm qc}$ a.e. in $\Omega$.
\end{lem}
\begin{proof}
We first claim that it suffices to prove the existence of $\Delta$ in the special case when $\om$ is the open ball $B=B(0,r_n)=B(0,(n/\omega_n)^\frac{1}{n})$ for which $\mL^n(B)=1$,  with $\beta=1$. Indeed suppose this has been proved with corresponding $\Delta=\Delta_B$ and let $\om$ be $C$-connected with $E(C)=E$ and $\mL^n(C)=\kappa\mL^n(\om)$. Then since $\om$ is $C$-filled, $\om$ contains an open ball of radius $\frac{1}{2}r(C)$, and since
$R(C)\geq \left(\frac{n\mL^n(C)}{\omega_n}\right)^\frac{1}{n}=\left(\frac{n\kappa\mL^n(\om)}{\omega_n}\right)^\frac{1}{n}$, 
  $\om$ contains an open ball $B_\rho=a+\rho B(0,1)$ of radius $$\rho=\frac{1}{2}\left(\frac{n\kappa\mL^n(\om)}{\omega_n}\right)^\frac{1}{n}(1-E^2)^\frac{1}{2}.$$
Therefore if \eqref{meta8} holds with $\Delta$ given by $\Delta(E,\kappa)=2^{-n}\kappa(1-E^2)^\frac{n}{2}\Delta_B$ then
\be
\int_{B_\rho}\int_{[N_\ep(K_1)\cup N_\ep(K_2)]^c} (1+|A|^p)\,d\nu_x(A)\,dx &+&\nonumber\\ &&\hspace{-1.6in}\int_{B_\rho}  \nu_x(N_\ep(K_1))\,dx<2^{-n}\kappa(1-E^2)^\frac{n}{2}\Delta_B\mL^n(\om).\lbl{meta9a}
\ee
 Define $\mu=(\mu_x)_{x\in B}$ by $\mu_x=\nu_{a+\frac{\rho}{r_n}x}$ and $\tilde y(x)=\frac{r_n}{\rho}y(a+\frac{\rho}{r_n}x)$. Then $D\tilde y(x)=\bar\mu_x$. Hence
\be
\lbl{meta9b}\hspace{-.25in}
\int_{B}\int_{[N_\ep(K_1)\cup N_\ep(K_2)]^c}\hspace{-.3in}(1+|A|^p)\,d\mu_x(A)\,dx+\int_{B} \mu_x(N_\ep(K_1))\,dx<\Delta_B.
\ee
If $z\in W^{1,p}(\om;\R^m)$ with $Dz(x)\in N_\ep(K_1)^{\rm qc}$ a.e. and $\tilde z(x)=\frac{r_n}{\rho}z(a+\frac{\rho}{r_n}x)$ we have that $D\tilde z(x)=Dz(a+\frac{\rho}{r_n}x)\in N_\ep(K_1)^{\rm qc}$ a.e. $x\in B$. Since we are assuming the result holds for $\om=B$ and $\beta=1$ we deduce that 
$$\|\tilde y-\tilde z\|_{L^1(B;\R^m)}>\Delta_B,$$
which implies that 
$$\|y-z\|_{L^1(\om;\R^m)}\geq\|y-z\|_{L^1(B_\rho;\R^m)}>\beta(\kappa,E)\Delta(\kappa,E){\mL^n(\om)}^\frac{n+1}{n},$$
where $\beta(\kappa,E)=\frac{1}{2}\kappa^\frac{1}{n}(1-E^2)^\frac{1}{2}$, proving the claim.

Suppose then that the result is false for $\om=B$ and $\beta=1$, so that it is false for $\Delta=\frac{1}{j}$ for every $j$. Then there  exist a sequence of $L^p$ gradient Young measures $\nu^{(j)}=(\nu^{(j)}_x)_{x\in B}$, and mappings $y^{(j)}\in W^{1,p}(B; \R^m)$ with $\bar\nu^{(j)}_x=Dy^{(j)}(x)$, $z^{(j)}\in W^{1,p}(B; \R^m)$ with $Dz^{(j)}(x)\in N_\ep(K_1)^{\rm qc}$ a.e. in $B$, such that
\be\hspace{-.2in}
\int_{B}\int_{[N_\ep(K_1)\cup N_\ep(K_2)]^c}(1+|A|^p)d\nu^{(j)}_x(A)\,dx
+\int_{B} \nu^{(j)}_x(N_\ep(K_1))\,dx<j^{-1}\lbl{meta10}
\ee
and
\be
\lbl{meta11}
\|y^{(j)}-z^{(j)}\|_{L^1(B;\R^m)}\leq j^{-1}.
\ee
It follows from \eqref{meta10} and the boundedness of $N_\ep(K_1), N_\ep(K_2)$ that
\be
\lbl{meta12}
\int_B\int_{M^{m\times n}}(1+|A|^p)d\nu^{(j)}_x(A)\,dx\leq M<\infty
\ee
for all $j$.
We may suppose without loss of generality that $\int_B y^{(j)}(x)\,dx=0$. We use the inequality (see Morrey \cite[p.~82]{morrey66} for similar results and proofs)
\be
\lbl{meta12a}
\hspace{-.15in}\int_B |u|^p\,dx\leq C\left(\int_B|Du|^p\,dx + \left|\int_B u\,dx\right|^p\right)\mbox { for all } u\in W^{1,p}(B;\R^m),
\ee
where $C$ is a constant. Applying \eqref{meta12a} to $y^{(j)}$, using $\bar\nu_x^{(j)}=Dy^{(j)}(x)$ and H\"older's inequality, we deduce that $y^{(j)}$ is bounded in $W^{1,p}(B;\R^m)$. 
Extracting a subsequence (not relabelled) if necessary, we may assume that $\nu^{(j)}\weakstar \nu$ in $L_w^{\infty}(B;C_0(M^{m\times n})^{*})$, 
and hence by Sychev \cite[Proposition 4.5]{sychev99} $\nu=(\nu_x)_{x\in B}$ is an $L^{p}$ gradient Young measure. Thus $\bar\nu_x=Dy(x)$ a.e. for some $y\in W^{1,p}(B;\R^m)$ with $\int_B y\,dx=0$. We claim that $y^{(j)}\weak y$ in $W^{1,p}(B;\R^m)$. To this end let $\theta_k:[0,\infty)\to [0,1]$ satisfy $\theta_k(s)=1$ for $s\in[0,k]$, $\theta_k(s)=0$ for $s\in[k+1,\infty)$. Then if $\psi\in C_0^\infty(\om)$ we have that 
\begin{eqnarray*}
\label{est1}
\limsup_{j\to \infty}\left|\int_B \psi(x)(Dy^{(j)}(x)-Dy(x))\,dx\right| &&\\&&\hspace{-1.5in}=\limsup_{j\to\infty}\left|\int_B \psi(x)\int_{M^{m\times n}}A\,d(\nu_x^{(j)}-\nu_x)(A)\,dx\right|\\&&\hspace{-1.5in}\leq\limsup_{j\to\infty}\left|\int_B \psi(x)\int_{M^{m\times n}}\theta_k(|A|)A\,d(\nu_x^{(j)}-\nu_x)(A)\,dx\right|\\
&&\hspace{-1.4in}+\limsup_{j\to\infty}\left|\int_B\psi(x)\int_{|A|\geq k}(1-\theta_k(|A|))A\,d(\nu_x^{(j)}-\nu_x)(A)\,dx\right|\\
 &&\hspace{-1.5in}\leq
\limsup_{j\to\infty}\left|\int_B|\psi(x)|\left(\int_{|A|\geq k}|A|\,d(\nu^{(j)}_x+\nu_x)(A)\right)\right|,\\
&&\hspace{-1.5in}\leq \frac{C}{k^{p-1}},
\end{eqnarray*}
for some constant $C$, where we have used $\nu^{(j)}\weakstar \nu$ in $L_w^{\infty}(B;C_0(M^{m\times n})^{*})$, \eqref{meta12} and relation (iii) of Theorem \ref{kp}. Letting $k\to\infty$ we deduce that $Dy^{(j)}\weak Dy$ in $L^p(B;M^{m\times n})$, from which the claim follows since $\int_B y^{(j)}\,dx=\int_B y\,dx=0$.
By the compactness of the embedding we have that $y^{(j)}\arr y$ strongly in $L^p(B;\R^m)$.

Note that by \eqref{meta10} we have that 
\begin{equation}
\label{meta12b}
\int_B(1-\nu_x^{(j)}(N_\ep(K_2)))\,dx\leq\frac{1}{j}.
\end{equation}
Let $\varphi_k\in C_0(M^{m\times n})$, with $0\leq\varphi_k(A)\leq 1$,  $\varphi_{k+1}(A)\leq \varphi_k(A)$ and \\ $\lim_{k\arr\infty}\varphi_k(A)=\chi_{N_\ep(K_2)}(A)$ for all $A\in M^{m\times n}$,  where $\chi_{N_\ep(K_2)}$ is the characteristic function of $N_\ep(K_2)$. Then by \eqref{meta12b} we have that
$$\lim_{j\arr\infty}\int_B\int_{M^{m\times n}}(1-\varphi_k(A))\,d\nu^{(j)}_{x}(A)\,dx=0,$$
and so by the weak* convergence of $\nu^{(j)}$ we deduce that
$$\int_B\int_{M^{m\times n}}(1-\varphi_k(A))d\nu_x(A)\,dx=0.$$
Passing to the limit $k\arr\infty$ using monotone convergence we obtain
$$\int_B[1-\nu_x(N_\ep(K_2))]\,dx=0.$$
Thus $\supp \nu_x\subset N_\ep(K_2)$ a.e. in $\Omega$. In particular $Dy(x)\in N_\ep(K_2)^{\rm qc}$ a.e. in $B$.

But from \eqref{meta11} we deduce that $z^{(j)}\arr y$ in $L^1(B;\R^m)$. Since $Dz^{(j)}\in N_\ep(K_1)^{\rm qc}$ it follows that $Dz^{(j)}\weakstar Dy$ in $L^\infty(B;M^{m\times n})$ and thus $Dy(x)\in N_\ep(K_1)^{\rm qc}$. But  $N_\ep(K_1)^{\rm qc}$ and $N_\ep(K_2)^{\rm qc}$ are disjoint by Corollary \ref{quasi}, giving the desired contradiction.
\end{proof}
\begin{lem}
\lbl{meta13}
Let $W$ satisfy $\rm (H2)$ and $\rm (H5)$. Then
\be
\lbl{meta14}
W(A)\geq K(1+|A|^p) \mbox{ for all } A\in [N_\ep(K_1)\cup N_\ep(K_2)]^c,
\ee
where
$$ K=\left\{\begin{array}{ll}c_1&\mbox{if }c_0\geq c_1,\\
c_0&\mbox{if }\alpha \leq c_0<c_1,\\
\frac{\alpha c_1}{\alpha +c_1-c_0}&\mbox{if }\alpha>c_0, c_1>c_0.
\end{array}\right.$$
\end{lem}
\begin{proof}This is elementary.
\end{proof}
{\it Proof of Theorem {\rm \ref{metstab}}}\hspace{.1in}
With $\Delta, \beta, K$ chosen as in Lemmas \ref{meta7a}, \ref{meta13} respectively, and $\gamma>0$ the constant in the transition layer estimate \eqref{4.9}, choose $\delta>0$ with
\be
\lbl{meta15}
\delta<\frac{K}{2}\min\left(\gamma,\Delta\min(1,\gamma)\right),
\ee
and let $\sigma=\beta\Delta\mL^n(\om)^\frac{n+1}{n}$.

For $\nu,\nu^*$ as in the statement of the theorem we have that
\be
\nonumber I(\nu)-I(\nu^*)&=&I(\nu)-0\\
\nonumber &=&\int_\Omega\int_{N_\ep(K_1)}W(A)\,d\nu_x(A)\,dx +\int_\Omega\int_{N_\ep(K_2)}W(A)\,d\nu_x(A)\,dx \\
&& \nonumber\hspace{1.2in}+\int_\Omega\int_{[N_\ep(K_1)\cup N_\ep(K_2)]^c}W(A)\,d\nu_x(A)\,dx\\
\nonumber&\geq& 0-\delta\int_\Omega\nu_x(N_\ep(K_2))\,dx\\
\nonumber&&\hspace{1in}+K\int_\Omega\int_{[N_\ep(K_1)\cup N_\ep(K_2)]^c}(1+|A|^p)\,d\nu_x(A)\,dx\\ \nonumber
&\geq &-\delta\int_\Omega\nu_x(N_\ep(K_2))\,dx\\
\nonumber &&\hspace{.2in}+\frac{K}{2}\int_\Omega\int_{[N_\ep(K_1)\cup N_\ep(K_2)]^c}(1+|A|^p)\,d\nu_x(A)\,dx \\
&& \hspace{.2in}+\frac{K}{2}\gamma \min\left(\int_\Omega \nu_x(N_\ep(K_1))\,dx,\int_\Omega \nu_x(N_\ep(K_2))\,dx\right).\lbl{meta16}
\ee
If $\int_\Omega \nu_x(N_\ep(K_1))\,dx\leq \int_\Omega\nu_x(N_\ep(K_2))\,dx$ then, since $Dy^*(x)\in N_\ep(K_1)^{\rm qc}$, by Lemma \ref{meta7a} we have that
$$\int_\Omega\int_{[N_\ep(K_1)\cup N_\ep(K_2]^c}(1+|A|^p)d\nu_x(A)\,dx+\int_\Omega \nu_x(N_\ep(K_1))\,dx\geq\Delta\mL^n(\om),$$
and hence by \eqref{meta15}, \eqref{meta16}
\be
\lbl{meta17}
I(\nu)-I(\nu^*)\geq -\delta\int_\Omega\nu_x(N_\ep(K_2))\,dx+\frac{K}{2}\min(1,\gamma)\Delta\mL^n(\om)>0.
\ee
On the other hand if $\int_\Omega \nu_x(N_\ep(K_2))\,dx\leq \int_\Omega\nu_x(N_\ep(K_1))\,dx$ then
\be
\lbl{meta18}\nonumber
I(\nu)-I(\nu^*)&\geq& (\frac{K}{2}\gamma-\delta)\int_\Omega\nu_x(N_\ep(K_2)\,dx\\
&&\hspace{-.2in}+\frac{K}{2}\int_\Omega\int_{[N_\ep(K_1)\cup N_\ep(K_2)]^c}(1+|A|^p)\,d\nu_x(A)\,dx\geq 0.
\ee
From \eqref{meta17}, \eqref{meta18} we see that $I(\nu)=I(\nu^*)$ if and only if $\supp \nu_x\subset \{A\in N_\ep(K_1):W(A)=0\}$, completing the proof.

\section{Applications} \label{sectapp}
\setcounter{equation}{0}

In this section we discuss the application of the results given above to materials that undergo diffusionless phase transformations involving a change of shape, usually called martensitic phase transformations.

\subsection{Variant rearrangement under biaxial stress} \label{sectvar}
\setcounter{equation}{0}

The original motivation for this paper were experiments of Chu \& James on the response of 
single crystal plates of martensitic material to biaxial stress.  The experimental details are presented elsewhere \cite{chu93,chujames93}.
In the design of these experiments attention was paid to the design of the loading device so as to 
correspond to the total free energy \be
\label{loading}
  {\mathcal E}(y) = \int_{\om} \varphi(D y(x),\theta) - T \cdot Dy(x)\, dx,
  \lbl{en}
\ee
where $y: \om \to \R^3$, $\om$ is a thin rectangular plate-like domain in $\R^3$, $\theta >0$ is the temperature, and $T = \sigma_1 e_1 \otimes e_1 + \sigma_2 e_2 \otimes e_2, \ \sigma_1>0, \sigma_2>0$ with $e_1,  e_2 \in \R^3$ (the orthonormal ``machine basis''). The first term in \eqref{loading} represents the free energy of the transforming material, 
and the second term is the loading device energy.

  In the experiments described here the temperature was held fixed at a value $\theta_0$ below the phase transformation temperature.  For this reason we henceforth drop $\theta$ from the notation.  The assigned $\sigma_1>0, \sigma_2>0$ are interpreted as the tractions (per unit reference area) applied
  to the edges of the specimen in the directions $e_1,  e_2$, respectively.  These were varied either incrementally or
  continuously during the tests.
  The material 
  was the alloy Cu-14wt.\%Al-4.0wt.\%Ni having a cubic-to-orthorhombic phase transformation, 
  leading to six variants of martensite at the test temperature.  These are modeled as 
  energy wells of $\varphi$ of the form \be
  \varphi(A) \ge 0, \quad \varphi(A) = 0 \Longleftrightarrow
  A \in \calM = SO(3)U_1 \cup \dots \cup SO(3) U_6.
\ee
with
\be
U_1 = \left(\begin{array}{ccc}\frac{\alpha + \gamma}{2} & \frac{\alpha-\gamma}{2} & 0 \\ \frac{\alpha-\gamma}{2} & \frac{\alpha + \gamma}{2} & 0 \\
0 & 0 & \beta  \end{array} \right), \quad
U_2 = \left(\begin{array}{ccc}\frac{\alpha + \gamma}{2} & \frac{\gamma-\alpha}{2} & 0 \\ \frac{\gamma-\alpha}{2} & \frac{\alpha + \gamma}{2} & 0 \\
0 & 0 & \beta  \end{array} \right), \nonumber \\
U_3 = \left(\begin{array}{ccc}\frac{\alpha + \gamma}{2} & 0 &\frac{\alpha-\gamma}{2} \\
0 & \beta & 0 \\
\frac{\alpha-\gamma}{2} & 0  & \frac{\alpha + \gamma}{2} \end{array} \right), \quad
U_4 = \left(\begin{array}{ccc}\frac{\alpha + \gamma}{2} & 0 & \frac{\gamma-\alpha}{2}\label{owells} \\
0 & \beta & 0 \\
\frac{\gamma-\alpha}{2} & 0 &  \frac{\alpha + \gamma}{2} \end{array} \right),\\
U_5 = \left(\begin{array}{ccc}  \beta & 0 & 0  \\
0  &  \frac{\alpha + \gamma}{2} & \frac{\alpha-\gamma}{2} \\
0  &  \frac{\alpha-\gamma}{2} & \frac{\alpha + \gamma}{2} \end{array} \right), \quad
U_6 = \left(\begin{array}{ccc} \beta & 0 & 0  \\
0  &  \frac{\alpha + \gamma}{2} & \frac{\gamma-\alpha}{2}  \\
0  & \frac{\gamma-\alpha}{2} & \frac{\alpha + \gamma}{2} \end{array} \right),  \nonumber \ee 
all expressed in an orthonormal basis $\hat{e}_1, \hat{e}_2,\hat{e}_3$ (the ``material basis'').  
The  measured values of $\alpha, \beta, \gamma$ for this alloy are $\alpha = 1.0619$, $\beta = 0.9178$ 
and $\gamma = 1.0230$ (Duggin and Rachinger \cite{dugginrachinger64}, Otsuka and Shimizu \cite{otsukashimizu74}). 
 The deviation of the material basis from the machine basis measures the orientation of the specimen. 
 Several orientations were
tested.

For many purposes, including the design of the orientations of crystals used in the tests, a simpler {\it constrained theory}  was used, valid in the regime that $|T|/\kappa$ is small\footnote{Using measured moduli of Yasunaga et al. \cite{yasunagaetal82,yasunagaetal83} for this alloy gives $\kappa \sim 15$ GPa.  A typical maximum value of $|T|$ in the tests was $15$ MPa, yielding $|T|/\kappa  \sim$ 15 MPa/15 GPa $= 10^{-3}$.}, $\kappa$ being the minimum eigenvalue of the linearized elasticity tensor,
linearized about $U_1$.  The constrained theory is based on the total free energy 
\be
{\mathcal E}(\nu) =\left\{ \begin{array}{ll} -\int_{\om} \int_{\calM} T \cdot A \, d\nu_x(A) dx  &\mbox{if }\supp \nu_x\subset \calM \mbox{ a.e. }x\in\om\\
+\infty&\mbox{otherwise},\end{array}\right. \lbl{cn}
 \ee 
defined on the set of $L^\infty$ gradient Young measures $\nu=(\nu_x)_{x\in\om}$.
The constrained theory has been justified as a limiting theory for Young measures of 
low energy sequences by  Forclaz \cite{forclaz14}  using $\Gamma$-convergence, but under 
assumptions not allowing $W(A)\to\infty$ as $\det A\to 0+$; the proof is  based on replacing $\varphi$ by $k \varphi$ in (\ref{en}) and letting $k \to \infty$ (a   
  similar procedure to letting $|T|/\kappa \to 0$ but which does not require additional smoothness assumptions on $\varphi$). A more general $\Gamma$-convergence analysis including the austenite energy well and allowing $W(A)\to\infty$ as $\det A\to 0+$ is given by  \cite[Proposition 1]{j68}. 

The design of orientations was based on the minimization of (\ref{cn}), which can be done in the following way by minimizing its integrand (see Chu \cite{chu93}, Chu \& James \cite{chujames93}). The machine basis was chosen in all cases such that, for all values of $\sigma_1>0, \sigma_2>0$, \be
  \min_{A \in SO(3) U_1 \cup SO(3) U_2} -T \cdot A <
  \min_{A \in SO(3) U_3 \cup \dots \cup SO(3) U_6} -T \cdot A.
\ee
In fact, the minimizer is unique
for all points in this open quadrant, except those on a smooth, strictly monotonically increasing curve $\calC: \sigma_2 = f(\sigma_1)$, $f \in C^{\infty}(0, \infty)$, which is nearly a straight line in the range of $\sigma_1, \sigma_2$ tested.  In fact, there exist functions $R_i \in C^{\infty}((0, \infty)\times (0, \infty);SO(3)),\ i = 1,2,$ such that $A = R_1(\sigma_1, \sigma_2) U_1$ is the unique minimizer of $-T \cdot A,\ A \in \calM$, for $\sigma_2 < f(\sigma_1)$ and  
$A = R_2(\sigma_1, \sigma_2) U_2$ is
its unique minimizer on $\calM$ for $\sigma_2 > f(\sigma_1)$.
The functions $R_1, R_2$  can and will be taken as the unique minimizers of $-T \cdot A $ on their respective energy wells $SO(3)U_1$, $SO(3)U_2$ on the full quadrant $\sigma_1>0, \sigma_2>0$.
There are precisely two equi-minimizers of $-T \cdot A,\ A \in \calM$, on $\calC$ given by $R_1(\sigma_1, f(\sigma_1)) U_1$ and $R_2(\sigma_1, f(\sigma_1)) U_2$.
The tests consisted of crossing the
curve $\sigma_2 = f(\sigma_1)$ by various loading programs $(\sigma_1(t), \sigma_2(t)), t>0$, and measuring 
the volume fractions of the subregions on the specimen where $Dy \in SO(3)U_1$ (variant 1) vs. $Dy \in SO(3)U_2$ (variant 2).  

The key point for this paper is that, by direct calculation of the functions $R_{1}, R_2$, 
\be \lbl{noconnection}
{\rm rank}(R_2(\sigma_1, f(\sigma_1)) U_2 - R_1(\sigma_1, f(\sigma_1)) U_1)>1
\ee for all $\sigma_1>0$ and all orientations tested.
Thus, fixing $\sigma_1 = \sigma_1^{\circ} \in (0, \infty)$, we  let $K_1 = \{ R_1(\sigma_1^{\circ}, f(\sigma_1^{\circ})) U_1 \}$, and $K_2 = \{ R_2(\sigma_1^{\circ}, f(\sigma_1^{\circ})) U_2 \}$.  By Example {\ref{twomatrices}, $K_1$ and $K_2$ are $L^p$ incompatible for $p>1$.  Letting $T_{\tau} = \sigma_1^{\circ} e_1 \otimes e_1
+ (c_2 \tau + f(\sigma_1^{\circ})) e_2 \otimes e_2$ and
$R_1^{\tau} = R_1(\sigma_1^{\circ}, c_2\tau + f(\sigma_1^{\circ}))$ for some $c_2>0$, a suitable function $W_{\tau}$ satisfying the hypotheses of Proposition \ref{Wep} for $m=n=3$ can be defined as follows:
\be
  W_{\tau}(A) = \left\{ \begin{array}{ll} -T_{\tau}\cdot (A -  R_1^{\tau}U_1) 
  \quad  &\mbox{if } A \in \calM, \\  \infty
   \quad  &\mbox{if } A \in  \calM^c.
  \end{array} \right. \lbl{cn1}
\ee
$W_{\tau}$ clearly satisfies (H1)$'$, (H2)$'$ and (H4)$'$, while (H3)$'$ is satisfied by choosing $c_2>0$ sufficiently small that $R_1^{\tau} U_1 \in N_{\ep}(K_1)$ for $0 \le \tau \le 1$. The region occupied by the specimen was approximately a thin rectangular plate, 
so we assume $\Omega$ is a rectangular solid. In particular $\om$ is $\om$-connected.
The energy density $W_{\tau}$ differs from that of the constrained theory by a trivial additive constant.  Theorem {\ref{metstab} then implies that the Young measure $\nu^{*}_{\tau} =  \delta_{R_1^{\tau} U_1}$ is metastable for sufficiently small $\tau > 0$ in the sense given there.

In this formulation we have used $\sigma_2$ as the parameter that moves the wells up and down.  One could equally well use a  parameterization of any other curve that crosses $\calC$ transversally.

Experimentally, transformation occurred by a sudden avalanche of transformation from variant 1 to variant 2 or {\it vice-versa}.
The transformation was sufficiently abrupt that a point in the $\sigma_1, \sigma_2$ plane could be associated with the transformation.  The series of points obtained in this way from diverse monotonic loading programmes, including those for which
$\sigma_1(t) = const.$, or $\sigma_2(t) = const.$, or
$\sigma_1(t) + \sigma_2(t) = const.$, all starting from a point $\sigma_1(0), \sigma_2(0)$ satisfying
$\sigma_2(0) \ll f(\sigma_1(0))$, at which the specimen was observed to be in variant 1, gave abrupt transformation to variant 2 at points lying very near a line
$\calC^+: \sigma_2 = f^+(\sigma_1) > f(\sigma_1),\ 0< a < \sigma_1 < b $.
Similarly, the same
kinds of loading programmes but run backwards, beginning from variant 2, led to transformation to variant 1 near a line
$\calC^-: \sigma_2 = f^-(\sigma_1) < f(\sigma_1),\ 0< a < \sigma_1 < b $.
For all orientations
tested, the three curves $\calC, \calC^+, \calC^-$ were nearly parallel, but the ``width of the hysteresis'', $\dist(\calC^+, \calC^-)$, varied significantly with orientation.

The concept developed in this paper is consistent with the behaviour described above.  We can examine this further by seeking an upper bound on the value of $\tau$ in (\ref{cn1}) beyond which $\nu^{*}_{\tau} = \delta_{R_1^{\tau} U_1}$ ceases to be metastable in the sense of Theorem \ref{metstab}.  As $\tau>0$ increases, there are more and more matrices $A \in SO(3)U_2$ with a negative value of the integrand $W_{\tau}(A)$. Suppose a value $\tau^+$ is reached such that for $\tau \gtrsim \tau^+$, that is $\tau\geq \tau^+$ with $\tau-\tau^+$ sufficiently small, there is a
matrix $B \in SO(3)U_2$ with rank$(B - R_1^{\tau} U_1) =1$, such that $W_{\tau}(B)< W_{\tau}(R_1^{\tau} U_1 )$.  
Then $\nu^{*}_{\tau} =  \delta_{R_1^{\tau} U_1}$ ceases to be metastable in the sense of Theorem \ref{metstab}.
In fact, it fails to be metastable even if $L^1$ in (\ref{meta7}) is replaced by $L^{\infty}$. In the case that $B - R_1^{\tau_1} U_1 = a \otimes n$, $\tau_1  \gtrsim \tau^+$, the counterexample is the family of competitors $\nu_x = \delta_{Dy_{\xi}(x)},\ \xi>0,$ defined for $x_0 \in \om$ by the $W^{1, \infty}(\om,\R^3)$ mapping \be
y_{\xi}(x) = \left\{ \begin{array}{ll} R_1^{\tau_1} U_1 (x-x_0)
  \  & \mbox{if }   (x-x_0) \cdot n<0,
  \\  B(x-x_0) \  &\mbox{if }   0 \le (x-x_0) \cdot n  \le \xi,
  \\  R_1^{\tau_1} U_1 (x-x_0) + \xi a
  \  &\mbox{if }  (x-x_0) \cdot n >\xi.
  \end{array} \right.   \nonumber
\ee
Since
$\| y_{\xi} - R_1^{\tau_1}U_1 (x-x_0) \|_{L^{1}(\om, \R^3)} \le  C\xi |a|$ for a constant $C=C(\om)$, then $\nu$ can be
made to fall into any preassigned neighbourhood of $\nu^{*}_{\tau}$ in the sense of (\ref{meta7}) of Theorem {\ref{metstab}
by making $\xi$ sufficiently small, and this competitor also works in the $L^{\infty}$ case.  
But clearly, since $W_{\tau}(B)< W_{\tau}(R_1^{\tau} U_1 )$ we have that
${\mathcal E}(\nu) < {\mathcal E}(\nu^{*}_{\tau})$, so $\nu^{*}_{\tau}$ is not metastable for $\tau \gtrsim \tau^+$.  

This qualitative argument for the sequence stable-metastable-unstable as $\tau$ increases,
in the sense discussed here, is complete if we can show that there exists $B$ with the
properties given above.  This is true by direct calculation for all the orientations 
tested.  This is done by first calculating explicitly $R_1^{\tau} U_1$, 
and then noticing that the wells $SO(3)U_1$ and $SO(3)U_2$ are compatible. That is,
even though \eqref{noconnection} holds, there are precisely two matrices 
$\hat{R}_a^{\tau}U_2, \hat{R}_b^{\tau}U_2 \in SO(3)U_2$ that differ from $R_1^{\tau} U_1$ 
by a matrix of rank 1 for $\tau>0$, and there exists a smallest value $\tau^+ > 0$ such 
that for $\tau > \tau^+$,
$W_{\tau}(B)< W_{\tau}(R_1^{\tau} U_1 )$ where $B$ is either
$\hat{R}_a^{\tau}U_2$ or $\hat{R}_b^{\tau}U_2$.  

Unless the orientation is special,
the two matrices $\hat{R}_a^{\tau}U_2$ or $\hat{R}_b^{\tau}U_2$ do not give the 
same value of $W_\tau$, suggesting a preference for one of them, assuming that
these examples deliver the point of first loss of metastability.  Let us suppose
for definiteness that the preference is for $\hat{R}_a^{\tau}U_2$, so 
$\hat{R}_a^{\tau}U_2 - R_1^{\tau} U_1= a_{\tau} \otimes n_{\tau}$ and
$W_{\tau}(\hat{R}_a^{\tau}U_2) \le W_{\tau}(R_1^{\tau} U_1 )$ for $\tau \ge \tau^+$
with equality precisely at $\tau = \tau^+$.
Combining these two conditions, we have
\be
a_{\tau^+} \cdot T_{\tau^+} n_{\tau^+} = 0.  \label{schmid}
\ee
This is formally equivalent to the well-known Schmid law (with Schmid constant 0)
\cite{schmidboas50}.
The left hand side of (\ref{schmid}) is usually interpreted as the ``critical
resolved stress  on the twin plane'', but in that case $T n$ is 
interpreted as the actual Piola-Kirchhoff traction on a pre-existing twin plane 
with unit normal
$n$ and $a = (F^+ - F^-)n$, where $F^{\pm}$ are local limiting values of the deformation
gradient.   The Schmid law prescribes a critical
value of $a \cdot T n$ at which this plane begins to move.  
The emergence of (\ref{schmid}) here 
has apparently nothing to do with stress in the specimen at all, which is
expected to be extremely complicated once bands of the second variant appear,
but rather concerns the loading device energy.  

In fact, as discussed in \cite{j43} and \cite{forclaz14}, for a suitable $C$-connected domain
$\Omega$ with corners, these simple counterexamples to metastability do not deliver
the points of first loss of metastability.  More complicated microstructures
still in an  $L^{\infty}$ local neighborhood, which are not simply Dirac masses, serve
as counterexamples to metastability at values of $\tau \in (\tau_1^+, \tau^+)$
for some $0 < \tau_1^+ < \tau^+$.  The 
experimentally observed
microstructure at transition (i.e., near $\calC^+$) is still somewhat more
complicated than these, and is clearly not a simple laminate.  If we accept that
the basis of the Schmid law is metastability as noted above, these more complicated
examples call into question the validity of that law in this context and also
indicate a dependence of hysteresis on the shape of the domain.  The latter is also
expected based on Example \ref{rooms}.

A detailed comparison of these upper bounds, either 
the one associated to $\tau^+$ or to $\tau_1^+$, with the experimentally measured
width of the hysteresis is difficult.  Experimentally, it is easiest to identify
$\calC^+$ with a possible loss of metastability, but the shoulder of the hysteresis
loop is not perfectly sharp, and some bands appear before reaching $\calC^+$, as
$\tau$ is increased.  Because of this ambiguity,  it is unclear where one should declare that
the homogeneous variant has begun to transform.  However, the overall impression one
gets when attempting this comparison is that the upper bounds associated to both
$\tau^+$ or to $\tau_1^+$ underestimate the size of the hysteresis.
Nevertheless there is rather
good qualitative agreement, in the sense that, for two specimens of different
orientation having widths of the hysteresis $\dist(\calC^+, \calC^-)$ differing by
a factor of 2, the corresponding upper bounds for the two cases also differ by a
factor of about 2.

\subsection{Dilatational transformation strain} \label{sectdil}
\setcounter{equation}{0}

Martensitic transformations having a pure dilatational transformation strain are rare,
but some examples are known in diffusional transformations, which involve shape change
and short or long range diffusion, depending on the overall composition of the alloy.  
The best known example is perhaps the ordering transformation from a disordered FCC phase to 
an L2$_1$ phase in Ni$_3$Al \cite{Wang2007871}, for which the ideas given above may 
be relevant.  

As a general treatment of dilatational transformation strains, consider two compact disjoint subsets $k_1, k_2$ of $(0,\infty)$, and corresponding energy wells
   $K_1 = k_1 SO(3)$
and $K_2 = k_2 SO(3)$, where $k_iSO(3)=\{k SO(3): k\in k_i\}$. That $K_1$ and $K_2$ are incompatible follows from  \cite[Theorem 4.4]{j35} and Lemma \ref{zhanglemma}, and also follows from the construction below, as we will indicate. 

We will construct  a polyconvex function $W_{0}$ that vanishes 
exactly on $K_1 \cup K_2$. This construction will enable us to embed $W_0$ in a family
$W_{\tau}, 0\leq \tau \leq 1$, for which we will prove metastability in the sense of Theorem \ref{metstab}. 

Following an observation of  \cite{j24} (see also \cite{p4,j19}), let
$1< \alpha < 3$ and let $\bar{h}:\R\to[0,\infty]$ be continuous with  $\bar{h} = \infty$ on $(-\infty, 0]$, $\bar h\in C^2(0,\infty)$ and 
 $\bar{h}^{-1}(0) = \{k^3: k\in k_1\cup k_2\}$.  We assume that $\bar{h}$ 
is convex outside a compact subset  $[a,b] \subset (0, \infty)$ containing $\bar{h}^{-1}(0)$,
so that there exists $\gamma > 0$ such that $\bar{h}'' \ge - \gamma$ on $(0, \infty)$.  
Let a convex function
$\tilde{h} \in C^2(\R)$ satisfy  
\be
\tilde{h}(t) = \left\{ \begin{array}{ll} - 3 c_1 t^{\alpha/3} &\mbox{if } a < t< b, \\ 
-3 c_1 (b+1)^{\alpha/3} &\mbox{if } t > b+1.     \end{array} \right. 
\ee
Such a convex function exists because the tangent at $t=b$  to $-3 c_1 t^{\alpha/3}$   lies below the constant function $-3 c_1 (b+1)^{\alpha/3}$ at $t=b+1$.

Define $h(t) =  \bar{h}(t) + \tilde{h}(t)$.
Since  $\bar{h}'' \ge - \gamma$, $1< \alpha<3$, and
$\bar{h}$ is convex outside   $[a,b]$,
there is $c_1 >0$ such that
\be
h''(t) = \bar{h}''(t) + \frac{1}{3}c_1\alpha  (3-\alpha)t^{-2 +\alpha/3}>0
\ee
on $[a, b]$ and so $h$ is convex on $\R$ and bounded below by 
$c_0 = -3 c_1 (b+1)^{\alpha/3}$.

Define an energy density for an isotropic elastic material by
\be
\lbl{isotrop}
W_0(A) = c_1(\lambda_1^{\alpha}+\lambda_2^{\alpha}+\lambda_3^{\alpha})
+ h(\lambda_1 \lambda_2 \lambda_3),
\ee
where $\lambda_1, \lambda_2, \lambda_3$ are the eigenvalues of $\sqrt{A^TA}$.
Because $h$ is convex and $1< \alpha < 3$,  $W_0$ is polyconvex by  \cite[Theorem 5.1]{j8}. 

Now we observe that $W_0$ has strict minima on $K_1 \cup K_2$. Indeed, since $h$ is bounded below and $h(0)=\infty$,
 the function $\sum_i c_1 \lambda_i^{\alpha} + h(\lambda_1 \lambda_2 \lambda_3)$ attains a minimum for
  $\lambda_1>0,  \lambda_2>0, \lambda_3>0$, where
\be
c_1 \alpha \lambda_i^{\alpha} = -h'(\lambda_1 \lambda_2 \lambda_3)\lambda_1 \lambda_2 \lambda_3.
\ee 
Hence $\lambda_1 = \lambda_2 = \lambda_3 = t^{1/3}$, where
$c_1 \alpha t^{\alpha/3} = -h'(t)t$.  These values of $t$ are critical points of the function
$\bar{h}(t) = 3c_1 t^{\alpha/3} + h(t) = W_0(t^{1/3} I)$, which has minimizers precisely on the 
set $\bar{h}^{-1}(0)$ by construction.  Hence, $W_0(A)$ has minimizers precisely 
on $K_1 \cup K_2$, where $W_0(A)=0$. 

Since $h$ is bounded below by $c_0$, 
the energy density $W_0$ satisfies the growth condition
\be
W_0(A) = c_1(\lambda_1^{\alpha}+\lambda_2^{\alpha}+\lambda_3^{\alpha})
+ h(\lambda_1 \lambda_2 \lambda_3) \ge c_0 + c_1 |A|^{\alpha}, 
\ee
so that $W_0$ satisfies conditions 
(H1) and (H2) of Section \ref{meta} for $p = \alpha$.  

To show that $K_1, K_2$ are incompatible we can consider the special case $\alpha=2$, when 
$$W_0(A)=c_1|A|^2+h(\det A).$$
If $\nu=(\nu_x)_{x\in\om}$ is an $L^\infty$ gradient Young measure with $\supp\nu_x\subset K_1\cup K_2$ a.e., we have that 
\be 0=\langle\nu_x,W_0\rangle
=c_1\langle\nu_x,|A|^2\rangle+\langle\nu_x,h(\det A)\rangle.\nonumber
\ee
Applying Jensen's inequality for the quasiconvex functions $|A|^2$ and $h(\det A)$, we have that 
$$ \langle\nu_x,|A|^2\rangle\geq |\bar\nu_x|^2,\;\;\langle\nu_x,h(\det A)\rangle\geq h(\det \bar\nu_x).$$
But $c_1  |\bar\nu_x|^2+ h(\det \bar\nu_x)=W_0(\bar\nu_x)\geq 0$.
Hence $\langle\nu_x,|A|^2\rangle=|\bar\nu_x|^2$, so that $\langle\nu_x,|A-\bar\nu_x|^2\rangle=0$ and hence $\nu_x=\delta_{Dy(x)}$ with $\bar\nu_x=Dy(x)$. But $Dy(x)\in K_1\cup K_2$ a.e., so that 
 $y$ is   a $W^{1,\infty}$
conformal mapping in 3 dimensions.  By classic results of Reshetnyak \cite{reshetnyak67}
all such mappings are smooth and  therefore   $Dy$ cannot be supported nontrivially on disjoint
closed sets.  Thus, $K_1, K_2$ are incompatible.  

The energy density $W_0$ can easily be extended to a family $W_{\tau}$ satisfying
the hypotheses $(\rm{H}1)'$-$(\rm{H}4)'$ of Proposition \ref{Wep}.  Let $\ep_0, \gamma$
be as in the transition layer estimate (Theorem \ref{transym1}) and let $0< \ep< \ep_0$
be fixed.  Since $N_{\ep}(K_1)$ and $N_{\ep}(K_2)$ are disjoint, we can let 
\be
W_{\tau}(A)  = W_0(A) - \tau H(\det A),
\ee
where $H:\R\to[0,1]$ is a smooth function satisfying
\be H(t)=\left\{\begin{array}{ll}1&\mbox{if }t\in\{k^3:k\in k_2\}:=N_2,\\
0&\mbox{if }\dist(t,N_2)>\rho(\ep),
\end{array}\right.\nonumber
\ee
where $\rho(\ep)>0$ is sufficiently small.
Clearly, $W_{\tau}$ satisfies the hypotheses $(\rm{H}1)'$-$(\rm{H}4)'$ with  $p=\alpha$.  Therefore, any $L^p$ gradient Young measure 
$\nu^*=(\nu_{x}^*)_{x\in\om}$ satisfying 
$\supp \nu^*_{x}\subset\{A\in N_{\ep}(K_1): W(A)=0\}$   is metastable
in the sense of Theorem \ref{metstab} for sufficiently small $\tau>0$, even though $W_\tau(A)=0$ for $A\in K_1$ and $W_\tau(A)=-\tau$ for $A\in K_2$. 
In \cite[Theorem 3.5]{j24} it is shown that such a result for free-energy functions of the form \eqref{isotrop} is not valid if the second energy well is arbitrarily deep.

Depending on the structure of $K_1$ the form of these metastable Young measures is
strongly restricted by   Reshetnyak's theorem, but $\nu_{x}^* = \delta_{Dy^*(x)}$,
where $y^*$ is a conformal mapping, is a possibility.

 Although it is interesting that pure dilatational phase transformations can be described by polyconvex free-energy functions,
the functions $W_{\tau}$ also
serve as lower bounds for  free-energy functions  for which metastability
in the sense of Theorem \ref{metstab} also holds. For example, by multiplying through the metastability estimate
by a sufficiently small positive constant, $W_{\tau}$ can be a lower bound for a variety
of non-polyconvex energy densities, with various choices of positive-definite linear elastic
moduli.  Of course, this modification also decreases 
$\gamma$, including the largest value of $\gamma$ for
which there is an $\ep_0>0$ satisfying the metastability theorem.  In this sense,
softening a material, but keeping the wells the same, lowers the barrier for metastabilty.

\subsection{Terephthalic acid} \label{secttere}
\setcounter{equation}{0}

Terephthalic acid \cite{daveyetal93,baileyetal67} is an interesting example in this context, since, among all reversible
structural transformations, it has an exceptionally large transformation strain.  It is the
largest strain in  a nominally reversible transformation in terms of 
$\dist(K_1, K_2)$ of which we are aware in a material that has no
rank-one connections between $K_1$ and $K_2$, that is, no solutions $A, B \in M^{3 \times 3}$ of
rank$(B-A) =1$, $A \in K_1, B \in K_2$.  The clearly visible large change-of-shape shown by
Davey et al.~\cite{daveyetal93} is remarkable.  

Terephthalic acid undergoes the transformation from Form I to Form II between $80^{\circ}$C and $100^{\circ}$C
\cite{daveyetal93}.  The transformation is reversible upon cooling to $30^{\circ}$C, at least for
a subset of crystallites; the application of a slight stress aids the reverse transformation.
The crystal structure and lattice parameter measurements of the I-II transformation have been
determined by Bailey et al.~\cite{baileyetal67}.  
Knowledge of these two structures and lattice parameters does not imply a unique
transformation stretch matrix due to the existence of infinitely many linear transformations
that take a lattice to itself.  The transformation stretch matrix  
\be
U = \left(\begin{array}{ccc} 0.970 & 0.038 & -0.121 \\
                             0.038 & 0.835 & -0.017 \\
                             -0.121 & -0.017 & 1.298 \end{array}    \right), \label{tsm}
\ee
is the one delivered by an algorithm \cite{chen14} designed to give the smallest distortion measured
by an appropriate norm.   The associated lattice correspondence of the two phases
(i.e., which vector is transformed to which vector) agrees with descriptions of the transformation
\cite{baileyetal67} and, semi-quantitatively, with photographs of crystals of the two phases \cite{daveyetal93}.
The eigenvalues of $U$ are 1.339, 0.939, 0.825.  
Nominally, there are two wells $K_1 = SO(3), K_2 =  SO(3)U$.  In fact twinning is observed in the
Form I, but this appears to be growth twinning \cite{daveyetal93}, and not produced during transformation.
(Both phases are triclinic, so there is no lowering of symmetry during transformation.)  
Since the middle eigenvalue of $U$ is not 1, there are no rank-one
connections between $K_1$ and $K_2$ \cite{j32}.

The best  sufficient conditions known that two wells $K_1$ and $K_2$ of this form are incompatible are due to 
Dolzmann, Kirchheim, M\"uller \& {\v{S}ver\'ak} \cite{dolzmann00}.  
Condition (ii) of their Theorem 1.2 is satisfied by $U$.  Therefore, $K_1$ and $K_2$ are incompatible,
and our metastability theorem applies to this case.

\section{Perspective on metastability and hysteresis} \label{sectper}
\setcounter{equation}{0}

In recent years different but related concepts of metastability have 
appeared in the literature \cite{james05,cui06,james09,kneupfer11,delville10,james13,kneupfer13,zwicknagl13}
motivated by some experimental results on a dramatic lattice parameter
dependence of the sizes of hysteresis loops.  These observations call
for new mathematical concepts of metastability whose form is not at
all clear.

Typical martensitic materials have energy wells of the form 
$K_1 = SO(3)$ and $K_2 = SO(3)U_1\ \cup \dots \cup\ SO(3)U_n$, with 
$n \ge 1$, and positive-definite, symmetric matrices $U_1, \dots, U_n \in M^{3 \times 3}$
satisfying $\{U_1, \dots, U_n\} = \{QU_1 Q^T: Q \in G \}$, where $G$ is a finite group of
orthogonal matrices (cf., (\ref{owells})).  Modulo the comments  in Section \ref{secttere} on the difficulties of determining the  transformation stretch matrix, $U_1$
for a particular material can be inferred from  X-ray measurements. 
All first order martensitic phase transformations have some amount of thermal hysteresis,
which refers to the fact that the transformation path on cooling differs from
that on heating.  A measurement of the fraction of the sample that has transformed 
vs.~temperature during a heating/cooling cycle gives a loop, called the hysteresis loop,
whose width is a typical measure of the hysteresis.  While indicative of
dissipation, the hysteresis loop does not collapse to zero as the loop is traversed 
more and more slowly, and so
is apparently not due to thermally activated processes, or dissipative mechanisms 
like viscosity or viscoelasticity. 

The matrix $U_1$ can be changed by changing the composition of the material.  
Suppose the ordered eigenvalues of $U_1$  are $\lambda_1 \le \lambda_2 \le \lambda_3$.   The main 
experimental observation underlying the analysis of hysteresis in the papers listed above is that, if a family of alloys is
prepared having a sequence of values of $\lambda_2$ approaching 1, the hysteresis gets
dramatically small.  Experimental graphs \cite{cui06} of hysteresis vs.~$\lambda_2$ show an apparent
cusp-like singularity at $\lambda_2= 1$, i.e., an extreme sensitivity of the size of the
hysteresis to $|\lambda_2 -1|$.  Very careful changes of composition in increments of 1/4 \% 
lead to alloys with exceptionally low hysteresis of $2-3^{\circ}$C 
in a variety of systems \cite{james13,zarnetta10}.
Since $\lambda_2 =1$ is a necessary and sufficient condition 
that there is a rank-one connection between $K_1$ and $K_2$, these results indicate that
the removal of stressed transition layers by strengthening conditions of compatibility
is relevant to hysteresis.  

A strict application of the ideas in this paper does
not explain this behaviour.  That is because, in all of these cases that have been studied experimentally,
$K_1$ and $K_2$ are compatible even in the starting alloys for which $\lambda_2$ is 
relatively far from 1.  In fact, all of these cases support solutions of the
crystallographic theory of martensite \cite{Wechsler53,j32}, implying that there exist
$A, B \in K_2$ and $C \in K_1$, such that $\rank(B-A) = 1$ and 
$\rank(\lambda B + (1- \lambda) A - C) = 1$ for some $0 < \lambda < 1$.  This series of
rank one connections implies the
existence of a Young measure $(\nu_x)_{x \in \Omega}$ supported nontrivially on $K_1, K_2$,
consisting of a laminate of two martensite variants $\dots A/B/A/B \dots$  meeting the austenite $C$ phase across a
vanishingly small planar transition layer.  In fact, the laminated martensite can be confined
between two such parallel planes which can be arbitrarily close together (see \cite{j60} for details).  
This family of test measures
then provide a counterexample to the metastability of say $\nu^* = \delta_C$ in the sense of
Theorem \ref{metstab}, even if $L^1$ in (\ref{meta7}) is replaced by $L^{\infty}$.

A special family of test functions $y_{\ep}$ of the type just described - a laminate $\dots A/B/A/B \dots$
confined between parallel planes at the distance $\ep$  and interpolated with $C$ in a layer near
these planes -- can be constructed explicitly.   Its energy can then be calculated by using a bulk energy of the type
studied in this paper with a suitable elastic energy density $W_{\tau}$, together with a interfacial energy per unit area
(taken as constant) on the $A/B$ boundaries. In this case $-\tau$ is interpreted as the temperature
and $\tau = 0$ is the transformation temperature.  This has been done in \cite{james09} and improved 
by Zwicknagl \cite{zwicknagl13}.  A graph of total energy vs.~$\ep$ gives a barrier whose height
is very sensitive to $|\lambda_2 - 1|$, and  decreases with decreasing temperature $-\tau$.  
If a critical value $\ep = \ep_{\rm crit}$ is introduced (modelling a pre-existing martensite nucleus of this type), 
and the temperature $\theta_c = -\tau$ is calculated at which $\ep = \ep_{\rm crit}$,
then the resulting graph of $0-\theta_c$ vs.~$\lambda_2$, all else fixed, has a singularity at
$\lambda_2 = 1$ and a shape similar to the
experimental graph of hysteresis vs.~$\lambda_2$. 


A related idea for a geometrically linear theory of the cubic-to-tetragonal transformation
and a sharp interface model of interfacial energy
is presented by Kn\"upfer, Kohn \& Otto \cite{kneupfer13} (see also \cite{kneupfer11}).
They show that the minimal bulk + interfacial energy of an inclusion of martensite of volume
$V$ scales as the maximum of $V^{2/3}, V^{9/11}$.  Minimal assumptions are made on the
shape of the inclusion.  If a bulk term is added to this energy of the form $-c \tau V$, $c>0$, modelling a lowering
of the martensite wells as the temperature $-\tau$ is decreased below transformation temperature,
then their result gives an energy barrier of the type described above.  They note that it would
be interesting to do a similar analysis of an austenite inclusion in martensite, and they
conjecture a higher energy barrier for the reverse transformation.  This is open, as is
a similar analysis for the cubic-to-orthorhombic case, where it would be interesting to 
investigate the dependence of the predicted barrier on $\lambda_2$.

Recently, even stronger conditions of compatibility called the {\it cofactor conditions} \cite{james05,james13}
have been closely satisfied in the ZnCuAu system by compositional changes, leading to the alloy 
Zn$_{45}$Au$_{30}$Cu$_{25}$. The cofactor conditions imply not only $\lambda_2=1$ but also a variety of other
microstructures with zero elastic energy.  The alloy Zn$_{45}$Au$_{30}$Cu$_{25}$
has a transformation strain $|U-I|$
comparable to that of the  alloys tuned to satisfy only $\lambda_2 = 1$, 
but shows still smaller hysteresis than
the lowest achieved by the $\lambda_2 = 1$ alloys, and also exceptional reversibility \cite{song13}.  
This example may indicate that 
metastability in phase transformations is not only sensitive to the wells being gradient compatible,
but also to the presence of a variety of different functions whose gradients
are nontrivially supported on $K_1, K_2$.  Another possibly relevant hypothesis is that
metastability is influenced by a possible sudden increase of the size of the quasiconvex hull of
the energy wells when the cofactor conditions are satisfied.  

An apparently obvious reconciliation of these concepts is to retain the idea of metastability, quantified
by local minimization, but to include a contribution for interfacial energy.
Accepted models of this type fall into two classes: sharp interface models and gradient models.
However, when combined with accepted notions of local minimization, neither of these models
give the behaviour described above. Before commenting on these two cases, we first note that
concepts of linearized stability are not relevant:  most measured values of linearized 
elastic moduli do soften as temperature is lowered to the phase transformation temperature, 
but the limiting value
of the minimum eigenvalue of the elasticity tensor is clearly positive
at transition in most cases, and this is the rule for strongly first order phase
transformations.

A typical sharp interface model assigns an energy per 
unit area to the jump set of $Dy$.  A comparative discussion of the energy minimisation problem
for several versions of these models is discussed in \cite{j58}. 
Consider the simple but relevant case of deciding whether a linear
deformation $y^*(x)= Ax, x \in \Omega,$ is metastable in some sense, where $A \in K_1$, $W_{\tau}(A) = 0$
and $W_{\tau}(K_2) = -\tau$, with $K_1$ and $K_2$ independent of $\tau$.  Suppose we have
favoured the low hysteresis situation by tuning the material as described above 
so that there exists $B \in K_2$ such that $B - A = a \otimes n$.  Putting aside linearized stability,
relevant concepts of local minimizer have the property that competitors can have
gradients on or near $K_2$, at least on
sufficiently small sets.  Trivially, if the underlying function space allows us to
smooth jumps of $Dy$, then a mollified version of the continuous function given for $x_0\in\om$ by
\be
y_{\ep}(x) = \left\{ \begin{array}{ll} B(x-x_0) & \mbox{ if }0< (x-x_0) \cdot n < \ep, \\
                                       A(x-x_0) & \mbox{ otherwise, }\end{array}   \right.  \label{test}
\ee
defeats metastability in $L^{\infty}$ as soon as $\tau>0$, predicting zero hysteresis.  Thus, of course, we
have to prevent smoothing.  This is easily done by forcing a jump, by restricting the domain of
$W_{\tau}$ to, say, $N_{\ep}(K_1) \cup N_{\ep}(K_2)$ with $\ep$ sufficiently small. However, in
that case, the prototypical test function (\ref{test}) for $\ep$ sufficiently small has positive
energy regardless how big is the value of $\tau$.  Thus, apparently for any of the accepted notions
of local minimizer, infinite hysteresis is predicted.

This dominance of interfacial energy at small scales, which overstabilizes linear deformations,
also occurs when gradient models of interfacial energy are combined with the bulk energies studied
here, as shown in \cite{j60}.   Consider a frame-indifferent  energy 
density $W_{\tau} \in C^2(M^{3 \times 3}_+)$,
continuous in $\tau$ and 
satisfying $W_{\tau}(A) \to \infty$ as $\det A \to 0$, and having positive-definite linearized
elasticity tensor at $I$.  Suppose
$W_{\tau}(K_1) = 0$ and $W_{\tau}(K_2) = -\tau$, for disjoint sets $K_1 = SO(3)$ and 
$K_2 = SO(3)U_1\ \cup \dots \cup\ SO(3)U_n$, and assume a total
energy of the form
\be
I(y) = \int_{\Omega} W_{\tau} (Dy) + \alpha |D^2y|^2 \, dx
\ee
with $\alpha>0$. 
In \cite{j60} it is shown that $y^*(x) = Rx + c, R \in SO(3), c \in \R^3$ is a local minimizer
of $I$ in $L^1$ {\it for every} $\tau>0$.  Again, infinite hysteresis is predicted.
Note that there may or may not be rank-one connections between $K_1$
and $K_2$.   It is probable that the the model introduced in \cite{j58},
that includes contributions from both sharp and diffuse interfacial energies, also leads to a metastability
result similar to that in \cite{j60}, though this has not been checked.

This inevitability of either zero hysteresis or infinite hysteresis, or, in the case of
linearized stability, predicted hysteresis that is too large, is avoided in models with
interfacial energy if, instead of using the standard approach to local minimization, one uses
a fixed neighbourhood of the proposed metastable deformation $y^*$, e.g., 
$\parallel\! y- y^*\! \parallel_{L^1} \le \ep_{\rm crit} $.  This is similar in spirit
to the introduction of the critical nucleus size above (also called $\ep_{\rm crit}$).
While this ultimately requires the formulation of an additional theory to predict
$\ep_{\rm crit}$, it would nevertheless be interesting to know whether this approach is consistent with
the observed lattice parameter dependence of hysteresis, as mentioned above.

Exotic models of interfacial energy that decrease the interfacial energy contribution when two interfaces
get close together could also restore finite hysteresis.  These are not widely accepted.

A better accepted idea, that is related to the introduction of the fixed neighbourhood
using $\ep_{\rm crit}$, is that, above transformation temperature, there are a variety of
small nuclei of martensite, stabilized by defects, waiting to grow, and there are
similar islands of austenite below transformation temperature.  While this consistent with
the (usually mild) dependence of hysteresis on preliminary processing, it is puzzling 
how this could yield hysteresis that is observed to be quite reproducible from alloy to alloy,
given similar processing. However, such thinking is based on the idea of  a single
``most dangerous'' nucleus determining transformation.   If, on the other hand, macroscopic
transformation arises from a collective interaction among many defects, so that something like
the law of large numbers is applicable, then one can imagine a reproducible size
of the hysteresis.  This kind of collective nucleation around defects, modelled by a position
dependent dissipation rate, can be seen in the
recent numerical simulations of DeSimone \&
Kru{\v z}{\'\i}k \cite{desimonekruzik13}.

Once metastability is lost, complex dissipative dynamic processes take place, involving interface motion,
microstructural evolution, and creation and annihilation of microstructure.  There is
currently insufficient information to formulate such dynamic laws, and the mathematical
theory in general of the dynamics of microstructure is primitive. 
There are a number of known possible approaches, including constitutive modelling, the sharp interface
kinetics of Abeyaratne \& Knowles \cite{abeyaratneknowles11} and the method of quasistatic evolution of Mielke \& 
Theil \cite{mielketheil04}.  All of these are reasonable based on general principles, but the latter seems 
to be the only one at present that can deal
with sufficient complexity of microstructure to begin to contemplate faithful dynamic predictions
\cite{desimonekruzik13}.  It is not yet known if these would be consistent with the 
sensitivity to conditions of compatibility mentioned above.

The surprising influence of conditions like $\lambda_2 = 1$ suggest that
simple kinematic approaches are valuable.  Their simplicity lies in the observation
that the conditions for loss of metastability seem to be much simpler than the
description of the dynamic process that takes place once metastability is lost.
From the perspective of this paper, and the
apparent success of the cofactor  conditions, it would be interesting to have methods of
quantifying the possibility of having many functions whose gradients are supported 
nontrivially on $K_1$ and $K_2$, especially those having finite area of the jump set
of the gradient.  A step in this direction is taken in recent work of R\"uland 
\cite{rueland13}.

\section*{Acknowledgements} The research of JMB was supported by by the EC (TMR contract FMRX - CT  EU98-0229 and
ERBSCI**CT000670), by 
EPSRC
(GRlJ03466, the Science and Innovation award to the Oxford Centre for Nonlinear
PDE EP/E035027/1, and EP/J014494/1), the European Research Council under the European Union's Seventh Framework Programme
(FP7/2007-2013) / ERC grant agreement no 291053 and
 by a Royal Society Wolfson Research Merit Award.  The research of RDJ was supported by NSF-PIRE (OISE-0967140), the MURI project Managing the Mosaic of Microstructure (FA9550-12-1-0458, administered by AFOSR), ONR (N00014-14-1-0714) and AFOSR (FA9550-15-1-0207). We are especially grateful to Jan Kristensen and Vladimir \v{S}ver\'{a}k for suggestions and advice, and to David Kinderlehrer, Bob Kohn, Pablo Pedregal and Gregory Seregin for their interest and helpful comments. We also thank the referees  for their careful reading of the paper and useful suggestions for improvement.

\bibliography{gen2,balljourn,ballconfproc,ballprep}

\begin{thebibliography}{10}

\bibitem{abeyaratneknowles11}
R.~Abeyaratne and J.~K. Knowles.
\newblock {\em Evolution of Phase Transitions: A Continuum Theory}.
\newblock Cambridge University Press, 2011.

\bibitem{astalafaraco}
K.~Astala and D.~Faraco.
\newblock Quasiregular mappings and {Y}oung measures.
\newblock {\em Proc. Roy. Soc. Edinburgh Sect. A}, 132(5):1045--1056, 2002.

\bibitem{aumann86}
R.~Aumann and S.~Hart.
\newblock Bi-convexity and bi-martingales.
\newblock {\em Israel J. Math.}, 54:159--180, 1986.

\bibitem{baileyetal67}
M.~Bailey and C.~J. Brown.
\newblock The crystal structure of terephthalic acid.
\newblock {\em Acta Crystallographica}, 22:387--391, 1967.

\bibitem{p4}
J.~M. Ball.
\newblock Constitutive inequalities and existence theorems in nonlinear
  elastostatics.
\newblock In R.~J. Knops, editor, {\em Nonlinear Analysis and Mechanics,
  Heriot-Watt Symposium, Vol. 1}. Pitman, 1977.

\bibitem{j8}
J.~M. Ball.
\newblock Convexity conditions and existence theorems in nonlinear elasticity.
\newblock {\em Arch. Ration. Mech. Anal.}, 63:337--403, 1977.

\bibitem{j19}
J.~M. Ball.
\newblock Discontinuous equilibrium solutions and cavitation in nonlinear
  elasticity.
\newblock {\em Phil. Trans. Royal Soc. London A}, 306:557--611, 1982.

\bibitem{p21}
J.~M. Ball.
\newblock A version of the fundamental theorem for {Y}oung measures.
\newblock In M.~Rascle, D.~Serre, and M.~Slemrod, editors, {\em Proceedings of
  conference on `Partial differential equations and continuum models of phase
  transitions'}, pages 3--16. Springer Lecture Notes in Physics. No. 359, 1989.

\bibitem{j35}
J.~M. Ball.
\newblock Sets of gradients with no rank-one connections.
\newblock {\em J. Math. Pures et Appliqu{\`{e}}es}, 69:241--259, 1990.

\bibitem{p31}
J.~M. Ball.
\newblock Some open problems in elasticity.
\newblock In {\em Geometry, Mechanics, and Dynamics}, pages 3--59. Springer,
  New York, 2002.

\bibitem{j43}
J.~M. Ball, C.~Chu, and R.~D. James.
\newblock Hysteresis during stress-induced variant rearrangement.
\newblock {\em J. de Physique IV}, C8:245--251, 1995.

\bibitem{j60}
J.~M. Ball and E.~C.~M. Crooks.
\newblock Local minimizers and planar interfaces in a phase-transition model
  with interfacial energy.
\newblock {\em Calc. Var. Partial Differential Equations}, 40(3-4):501--538,
  2011.

\bibitem{u6}
J.~M. Ball and R.~D. James.
\newblock Varying volume fractions of gradient {Y}oung measures.
\newblock In preparation.

\bibitem{j32}
J.~M. Ball and R.~D. James.
\newblock Fine phase mixtures as minimizers of energy.
\newblock {\em Arch. Ration. Mech. Anal.}, 100:13--52, 1987.

\bibitem{j44}
J.~M. Ball and R.~D. James.
\newblock Local minimizers and phase transformations.
\newblock {\em Zeitschrift f\"{u}r Angewandte Mathematik und Mechanik}, 76
  (Suppl. 2):389--392, 1996.

\bibitem{j68}
J.~M. Ball and K.~Koumatos.
\newblock Quasiconvexity at the boundary and the nucleation of austenite.
\newblock {\em Arch. Ration. Mech. Anal.}
\newblock To appear.

\bibitem{j24}
J.~M. Ball and J.~E. Marsden.
\newblock Quasiconvexity at the boundary, positivity of the second variation,
  and elastic stability.
\newblock {\em Arch. Ration. Mech. Anal.}, 86:251--277, 1984.

\bibitem{j58}
J.~M. Ball and C.~Mora-Corral.
\newblock A variational model allowing both smooth and sharp phase boundaries
  in solids.
\newblock {\em Communications on Pure and Applied Analysis}, 8:55--81, 2009.
\newblock \url{http://aimsciences.org/journals/cpaa/}.

\bibitem{j26}
J.~M. Ball and F.~Murat.
\newblock {$W^{1,p}$}-quasiconvexity and variational problems for multiple
  integrals.
\newblock {\em J. Functional Analysis}, 58:225--253, 1984.

\bibitem{Bhattetal94}
K.~Bhattacharya, N.~B. Firoozye, R.~D. James, and R.~V. Kohn.
\newblock Restrictions on microstructure.
\newblock {\em Proc. Royal Soc. Edinburgh}, 124A:843--878, 1994.

\bibitem{chaudhurimuller2004}
N.~Chaudhuri and S.~M{\"u}ller.
\newblock Rigidity estimate for two incompatible wells.
\newblock {\em Calc. Var. Partial Differential Equations}, 19(4):379--390,
  2004.

\bibitem{chaudhurimuller2006}
N.~Chaudhuri and S.~M{\"u}ller.
\newblock Scaling of the energy for thin martensitic films.
\newblock {\em SIAM J. Math. Anal.}, 38(2):468--477 (electronic), 2006.

\bibitem{james13}
X.~Chen, Y.~Song, V.~Dabade, and R.~D. James.
\newblock Study of the cofactor conditions: Conditions of supercompatibility
  between phases.
\newblock {\em Journal of the Mechanics and Physics of Solids}, 61:2566 --
  2587, 2013.

\bibitem{chen14}
X.~Chen, Y.~Song, R.~D. James, and N.~Tamura.
\newblock Determination of the transformation stretch tensor for structural
  transformations.
\newblock {\em Physical Review Letters}, 2015.
\newblock Manuscript submitted for publication.

\bibitem{chlebik}
M.~Chleb{\'{\i}}k and B.~Kirchheim.
\newblock Rigidity for the four gradient problem.
\newblock {\em J. Reine Angew. Math.}, 551:1--9, 2002.

\bibitem{chu93}
C.~Chu.
\newblock {\em Hysteresis and Microstructure: a Study of Biaxial Loading on
  Compound Twins of Copper-Aluminium-Nickel Single Crystals}.
\newblock PhD thesis, Department of Aerospace Engineering and Mechanics,
  University of Minnesota, 1993.

\bibitem{chujames93}
C.~Chu and R.~D. James.
\newblock Biaxial loading experiments on {Cu-Al-Ni} single crystals.
\newblock In {\em Experiments in Smart Materials and Structures}, pages 61--69.
  ASME, 1993.
\newblock AMD-Vol. 181.

\bibitem{cui06}
J.~Cui, Y.~S. Chu, O.~Famodu, Y.~Furuya, J.~Hattrick-Simpers, R.~D. James,
  A.~Ludwig, S.~Thienhaus, M.~Wuttig, Z.~Zhang, and I.~Takeuchi.
\newblock {Combinatorial search of thermoelastic shape memory alloys with
  extremely small hysteresis width}.
\newblock {\em Nature Materials}, pages 286--290, 2006.

\bibitem{dacorogna}
B.~Dacorogna.
\newblock {\em Direct Methods in the Calculus of Variations}, volume~78 of {\em
  Applied Mathematical Sciences}.
\newblock Springer, New York, second edition, 2008.

\bibitem{daveyetal93}
R.~J. Davey, S.~J. Maginn, S.~J. Andrews, A.~M. Buckley, D.~Cottler,
  P.~Dempsey, J.~E. Rout, D.~R. Stanley, and A.~Taylor.
\newblock Stabilization of a metastable phase by twinning.
\newblock {\em Nature}, 366:248--250, 1993.

\bibitem{delellisszekelyhidi2006}
C.~De~Lellis and L.~Sz{\'e}kelyhidi, Jr.
\newblock Simple proof of two-well rigidity.
\newblock {\em C. R. Math. Acad. Sci. Paris}, 343(5):367--370, 2006.

\bibitem{delville10}
R.~Delville, S.~Kasinathan, Z.~Zhang, V.~Humbeeck, R.~D. James, and
  D.~Schryvers.
\newblock {A transmission electron microscopy study of phase compatibility in
  low hysteresis shape memory alloys}.
\newblock {\em Philosophical Magazine}, pages 177--195, 2010.

\bibitem{desimonekruzik13}
A.~DeSimone and M.~Kru{\v z}{\'\i}k.
\newblock Domain patterns and hysteresis in phase-transforming solids: analysis
  and numerical simulations of a sharp interface dissipative model via
  phase-field approximation.
\newblock {\em Networks and Heterogeneous Media}, 8:481--489, 2013.

\bibitem{dolzmann00}
G.~Dolzmann, B.~Kirchheim, S.~M\"uller, and V.~{\v{S}ver\'ak}.
\newblock The two-well problem in three dimensions.
\newblock {\em Calc. Var.}, 10:21--40, 2000.

\bibitem{dugginrachinger64}
M.~J. Duggin and W.~A. Rachinger.
\newblock The nature of the martensitic transformation in a
  copper-nickel-aluminum alloy.
\newblock {\em Acta metall.}, 12:529--535, 1964.

\bibitem{faracoszek}
D.~Faraco and L.~Sz{\'e}kelyhidi.
\newblock Tartar's conjecture and localization of the quasiconvex hull in
  {$\Bbb R^{2\times 2}$}.
\newblock {\em Acta Math.}, 200(2):279--305, 2008.

\bibitem{firoozye90}
N.~Firoozye.
\newblock {\em Optimal Translations and Relaxations of some Multiwell
  Energies}.
\newblock PhD thesis, Courant Institute, New York University, 1990.

\bibitem{forclaz14}
A.~Forclaz.
\newblock Local minimizers and the {S}chmid law in corner-shaped domains.
\newblock {\em Arch. Ration. Mech. Anal.}, 211:555--591, 2014.

\bibitem{fraenkel79}
L.~E. Fraenkel.
\newblock On regularity of the boundary in the theory of {S}obolev spaces.
\newblock {\em Proc. London Math. Soc. (3)}, 39(3):385--427, 1979.

\bibitem{grabovskymengesha09}
Y.~Grabovsky and T.~Mengesha.
\newblock Sufficient conditions for strong local minima: the case of {$C^1$}
  extremals.
\newblock {\em Trans. Amer. Math. Soc.}, 361(3):1495--1541, 2009.

\bibitem{heinz}
S.~Heinz.
\newblock On the structure of the quasiconvex hull in planar elasticity.
\newblock {\em Calc. Var.}, 50:481--489, 2014.

\bibitem{james05}
R.~D. James and Z.~Zhang.
\newblock A way to search for multiferroic materials with unlikely combinations
  of physical properties.
\newblock In A.~Planes, L.~Man{\~o}sa, and A.~Saxena, editors, {\em Magnetism
  and Structure in Functional Materials, Springer Series in Materials Science},
  volume~9, pages 159--175. Springer-Verlag, Berlin, 2005.

\bibitem{kinderlehrerpedregal91}
D.~Kinderlehrer and P.~Pedregal.
\newblock Characterizations of {Young} measures generated by gradients.
\newblock {\em Arch. Ration. Mech. Anal.}, 115:329--365, 1991.

\bibitem{kinderlehrerpedregal94}
D.~Kinderlehrer and P.~Pedregal.
\newblock Gradient {Y}oung measures generated by sequences in {S}obolev spaces.
\newblock {\em J. Geom. Anal.}, 4:59--90, 1994.

\bibitem{kirchheim01}
B.~Kirchheim.
\newblock Deformations with finitely many gradients and stability of
  quasiconvex hulls.
\newblock {\em C. R. Acad. Sci. Paris S\'er. I Math.}, 332:289--294, 2001.

\bibitem{kirchheim03}
B.~Kirchheim.
\newblock {\em Rigidity and Geometry of Microstructures}.
\newblock Habilitation, University of Leipzig, 2003.

\bibitem{kirchheimszekelyhidi08}
B.~Kirchheim and L.~Sz\'ekelyhidi, Jr.
\newblock On the gradient set of {L}ipschitz maps.
\newblock {\em J. reine angew. Math.}, 625:215--229, 2008.

\bibitem{kneupfer11}
H.~Kn\"upfer and R.~V. Kohn.
\newblock Minimal energy for elastic inclusions.
\newblock {\em Proc. Royal Soc. Lond. Ser. A: Math. Phys. Eng. Sci.},
  2127:695--717, 2011.

\bibitem{kneupfer13}
H.~Kn\"upfer, R.~V. Kohn, and F.~Otto.
\newblock Nucleation barriers for the cubic to tetragonal phase transformation.
\newblock {\em Communications on Pure and Applied Mathematics}, 66:867--904,
  2013.

\bibitem{kohnetal2000}
R.~V. Kohn, V.~Lods, and A.~Haraux.
\newblock Some results about two incompatible elastic strains.
\newblock Unpublished mansuscript, 2000.

\bibitem{kohnsternberg89}
R.~V. Kohn and P.~Sternberg.
\newblock Local minimizers and singular perturbations.
\newblock {\em Proc. Royal Soc. Edinburgh}, 111A:69--84, 1989.

\bibitem{kristensen94}
J.~Kristensen.
\newblock {\em Lower Semicontinuity of Variational Integrals}.
\newblock PhD thesis, Technical University of Lyngby, 1994.

\bibitem{Kuratowski}
K.~Kuratowski and C.~Ryll-Nardzewski.
\newblock A general theorem on selectors.
\newblock {\em Bull. Acad. Polon. Sci. S\'er. Sci. Math. Astronom. Phys.},
  13:397--403, 1965.

\bibitem{matos92}
J.~P. Matos.
\newblock Young measures and the absence of fine microstructure in a class of
  phase transitions.
\newblock {\em European J. Applied Maths}, 3:31--54, 1992.

\bibitem{mazya2011}
V.~Maz'ya.
\newblock {\em Sobolev Spaces with Applications to Elliptic Partial
  Differential Equations}, volume 342 of {\em Grundlehren der Mathematischen
  Wissenschaften [Fundamental Principles of Mathematical Sciences]}.
\newblock Springer, Heidelberg, augmented edition, 2011.

\bibitem{mcshane}
E.~J. McShane and T.~A. Botts.
\newblock {\em Real Analysis}.
\newblock van Nostrand, Princeton, N.J., 1959.
\newblock reprinted Dover, 2005.

\bibitem{mielketheil04}
A.~Mielke and F.~Theil.
\newblock On rate-independent hysteresis models.
\newblock {\em Nonlinear Differential Equations and Applications}, 11:151--189,
  2004.

\bibitem{morrey66}
C.~B. Morrey.
\newblock {\em Multiple Integrals in the Calculus of Variations}.
\newblock Springer, 1966.

\bibitem{muller99a}
S.~M{\"u}ller.
\newblock A sharp version of {Z}hang's theorem on truncating sequences of
  gradients.
\newblock {\em Trans. Amer. Math. Soc.}, 351(11):4585--4597, 1999.

\bibitem{muller99}
S.~M\"{u}ller.
\newblock Variational methods for microstructure and phase transitions.
\newblock In {\em Calculus of Variations and Geometric Evolution problems},
  volume 1713 of {\em Lecture Notes in Math.}, pages 85--210. Springer, Berlin,
  1999.

\bibitem{otsukashimizu74}
K.~Otsuka and K.~Shimizu.
\newblock Morphology and crystallography of thermoelastic {C}u-{A}l-{N}i
  martensite analyzed by the phenomenological theory.
\newblock {\em Trans. Japan Inst. Metals}, 15:103--108, 1974.

\bibitem{Parthasarathy}
K.~R. Parthasarathy.
\newblock {\em Probability Measures on Metric Spaces}.
\newblock Probability and Mathematical Statistics, No. 3. Academic Press, Inc.,
  New York-London, 1967.

\bibitem{pedregal94}
P.~Pedregal.
\newblock Jensen's inequality in the calculus of variations.
\newblock {\em Differential Integral Equations}, 7:57--72, 1994.

\bibitem{reshetnyak67}
Y.~G. Reshetnyak.
\newblock Liouville's theorem on conformal mappings under minimal regularity
  assumptions.
\newblock {\em Siberian Math. J.}, 8:631--653, 1967.

\bibitem{rueland13}
A.~R\"uland.
\newblock The cubic-to-orthorhombic phase transition –- rigidity and
  non-rigidity properties in the linear theory of elasticity.
\newblock {\em Arch. Ration. Mech. Anal.}
\newblock submitted.

\bibitem{schmidboas50}
E.~Schmid and W.~Boas.
\newblock {\em Plasticity of Crystals (translation of the 1935 text in
  German)}.
\newblock F. A. Hughes, London, 1950.

\bibitem{song13}
Y.~Song, X.~Chen, V.~Dabade, T.~W. Shield, and R.~D. James.
\newblock Enhanced reversibility and unusual microstructure of a
  phase-transforming material.
\newblock {\em Nature}, 502:85--88, 2013.

\bibitem{stein70}
E.~M. Stein.
\newblock {\em Singular Integrals and Differentiability Properties of
  Functions}.
\newblock Princeton University Press, Princeton, New Jersey, 1970.

\bibitem{strang82}
G.~Strang.
\newblock The width of a chair.
\newblock {\em Amer. Math. Monthly}, 89:529--534, 1982.

\bibitem{Sverak93a}
V.~{\v{S}ver\'{a}k}.
\newblock On regularity for the {M}onge-{A}mp\`ere equation without convexity
  assumptions.
\newblock Heriot-Watt University preprint, 1991.

\bibitem{sverak93e}
V.~{\v{S}ver\'{a}k}.
\newblock New examples of quasiconvex functions.
\newblock {\em Arch. Rational Mech. Anal.}, 119:293--300, 1992.

\bibitem{sychev99}
M.~A. Sychev.
\newblock A new approach to {Y}oung measure theory, relaxation and convergence
  in energy.
\newblock {\em Ann. Inst. H. Poincar\'e Anal. Non Lin\'eaire}, 16(6):773--812,
  1999.

\bibitem{szekelyhidi2005}
L.~Sz{\'e}kelyhidi, Jr.
\newblock Rank-one convex hulls in {$\Bbb R^{2\times 2}$}.
\newblock {\em Calc. Var. Partial Differential Equations}, 22(3):253--281,
  2005.
\newblock Erratum, same journal 28(2007) 545-546.

\bibitem{tartar93}
L.~Tartar.
\newblock Some remarks on separately convex functions.
\newblock In {\em Proceedings of Conference on Microstructures and Phase
  Transitions, IMA, Minneapolis, 1990}, 1993.

\bibitem{wagner}
D.~H. Wagner.
\newblock Survey of measurable selection theorems.
\newblock {\em SIAM J. Control Optimization}, 15(5):859--903, 1977.

\bibitem{Wang2007871}
J.~C. Wang, M.~Osawa, T.~Yokokawa, H.~Harada, and M.~Enomoto.
\newblock Modeling the microstructural evolution of {N}i-base superalloys by
  phase field method combined with {CALPHAD} and {CVM}.
\newblock {\em Computational Materials Science}, 39:871 -- 879, 2007.

\bibitem{Wechsler53}
M.~S. Wechsler, D.~S. Lieberman, and T.~A. Read.
\newblock On the theory of the formation of martensite.
\newblock {\em Trans. AIME J. Metals}, 197:1503--1515, 1953.

\bibitem{yasunagaetal82}
M.~Yasunaga, Y.~Funatsu, S.~Kojima, K.~Otsuka, and T.~Suzuki.
\newblock Ultrasonic velocity near the martensitic transformation temperature.
\newblock {\em J. de Physique}, C4:603--608, 1982.

\bibitem{yasunagaetal83}
M.~Yasunaga, Y.~Funatsu, S.~Kojima, K.~Otsuka, and T.~Suzuki.
\newblock Measurement of elastic constants.
\newblock {\em Scripta Met.}, 17:1091--1094, 1983.

\bibitem{zarnetta10}
R.~Zarnetta, R.~Takahashi, M.~L. Young, A.~Savan, Y.~Furuya, S.~Thienhaus,
  B.~Maa\ss, M.~Rahim, J.~Frenzel, H.~Brunken, Y.~S. Chu, V.~Srivastava, R.~D.
  James, I.~Takeuchi, G.~Eggeler, and A.~Ludwig.
\newblock {Identification of quaternary shape memory alloys with near zero
  thermal hysteresis and unprecedented functional stability}.
\newblock {\em Advanced Functional Materials}, pages 1917--1923, 2010.

\bibitem{zhang91}
K.~Zhang.
\newblock Rank 1 connections and the three ``well'' problem.
\newblock Unpublished manuscript, 1991.

\bibitem{zhang92}
K.~Zhang.
\newblock A construction of quasiconvex functions with linear growth at
  infinity.
\newblock {\em Ann. Scuola. Norm. Sup. Pisa}, 19:313--326, 1992.

\bibitem{zhang2003b}
K.~Zhang.
\newblock Neighborhoods of parallel wells in two dimensions that separate
  gradient {Y}oung measures.
\newblock {\em SIAM J. Math. Anal.}, 34(5):1207--1225 (electronic), 2003.

\bibitem{zhang2003}
K.~Zhang.
\newblock On separation of gradient {Y}oung measures.
\newblock {\em Calc. Var. Partial Differential Equations}, 17(1):85--103, 2003.

\bibitem{zhang2003a}
K.~Zhang.
\newblock Separation of gradient {Y}oung measures and the {BMO}.
\newblock In {\em International {C}onference on {H}armonic {A}nalysis and
  {R}elated {T}opics ({S}ydney, 2002)}, volume~41 of {\em Proc. Centre Math.
  Appl. Austral. Nat. Univ.}, pages 161--169. Austral. Nat. Univ., Canberra,
  2003.

\bibitem{james09}
Z.~Zhang, R.~D. James, and S.~M\"uller.
\newblock Energy barriers and hysteresis in martensitic phase transformations.
\newblock {\em Acta Materialia (Invited Overview)}, 57:2332--4352, 2009.

\bibitem{zwicknagl13}
B.~Zwicknagl.
\newblock Microstructures in low-hysteresis shape memory alloys: scaling
  regimes and optimal needle shapes.
\newblock {\em Arch. Ration. Mech. Anal.}, 213:355--421, 2014.

\end{thebibliography}
\bibliographystyle{abbrv}

\end{document}